\documentclass[preprint,12pt]{elsarticle}

\usepackage[colorlinks=true]{hyperref}

\hypersetup{
    linkcolor=blue,      
    citecolor=blue,      
    urlcolor=blue
}

\usepackage{amssymb}
\usepackage{amsmath,amsthm}
\usepackage{nicefrac}
\usepackage{longtable}
\usepackage{bm}
\usepackage{subfigure}
\newtheorem{thm}{Theorem}[section]

\newtheorem{prop}{Proposition}[section]
\newtheorem{rem}{Remark}[section]
\newtheorem{lem}{Lemma}[section]
\newtheorem{defi}{Definition}[section]

\newcommand\be {\begin{equation}}
	\newcommand\ee {\end{equation}}

\def\del{\partial}
\def\DOT{\!\cdot\!}
\def\dt{{\Delta t}} 

\def\u{\mbox{\boldmath $u$}}
\def\U{\mbox{\boldmath $U$}}

\def\f{\mbox{\boldmath $f$}}

\def\pomega{\mbox{\boldmath $\omega$}}
\def\ppsi{\mbox{\boldmath $\psi$}}

\numberwithin{equation}{section}

\allowdisplaybreaks[4]
\journal{}

\begin{document}

\begin{frontmatter}

\title{Long-time stability and convergence analysis of an IMEX BDF3 scheme for 2-D incompressible 
	Navier-Stokes equation}
\author[Cheng]{Kelong Cheng} 
\affiliation[Cheng]{organization={School of Science, Civil Aviation Flight University of China},
            city={Guanghan},
            postcode={618307}, 
            state={Sichuan},
            country={China}}
\author[Sun]{Jingwei Sun}
\author[Sun]{Hong Zhang\corref{cor1}}
\cortext[cor1]{Corresponding author\ead{zhanghnudt@163.com}}
\affiliation[Sun]{organization={National University of Defense Technology},
	city={Changsha},
	postcode={410073}, 
	state={Hunan},
	country={China}}
\begin{abstract}
High-order time-stepping schemes are crucial for simulating incompressible fluid flows due to their ability to capture complex turbulent behavior and unsteady motion. In this work, we propose a third-order accurate numerical scheme for the two-dimensional incompressible Navier-Stokes equation. Spatial and temporal discretization is achieved using Fourier pseudo-spectral approximation and the BDF3 stencil, combined with the Adams-Bashforth extrapolation for the nonlinear convection term, resulting in a semi-implicit, fully discrete formulation. This approach requires solving only a single Poisson-like equation per time step while maintaining the desired temporal accuracy. Classical numerical experiments demonstrate the advantage of our scheme in terms of permissible time step sizes. Moreover, we establish uniform-in-time bounds for the vorticity in both $L^2$ and higher-order $H^m$ norms ($m \geq 1$), provided the time step is sufficiently small. These bounds, in turn, facilitate the derivation of optimal convergence rates.
\end{abstract}



\begin{keyword}
	Incompressible Naiver-Stokes equation\sep backward differentiation formula \sep Fourier pseudo-spectral \sep Adams-Bashforth extrapolation \sep uniform-in-time estimate



\end{keyword}

\end{frontmatter}

\section{Introduction}
In this paper, we focus on the third-order backward differentiation (BDF) approximations for the unsteady incompressible Navier-Stokes equation (NSE) in bounded two-dimensional (2-D) domain subjected to the periodic boundary condition: 
\begin{eqnarray}
	\del_t \omega + \u \DOT \nabla \omega = \nu \Delta \omega + \f ,  \label{NSE1} 
	\\
	- \Delta \psi = \omega ,   \label{NSE2} 
	\\
	\u = \nabla^{\bot} \psi  
	= \left( \partial_y \psi , - \partial_x \psi \right) .  \label{NSE3}
\end{eqnarray} 
Here, $\u=(u,v)^{T}$ is the velocity field, $\omega = - u_y + v_x$ represents the scalar vorticity, $\psi$ stands for the scalar stream function, $\nu=1/{\mbox{Re}}$ represents the coefficient of the fluid kinematic viscosity with the Reynolds numbers $\mbox{Re}$, and the external force term $\f$ has a zero average: $\int_\Omega \, \f \, d {\bf x} =0$. For simplicity, a uniform-in-time $L^2$ bound is assumed for $\f$, i.e., $\f \in L^\infty (0,T; L^2)$ and $ \left\| \f ( \cdot , t ) \right\|  \le M$ for any $t >0$. 
In this formulation, the pressure gradient disappears,  and the velocity vector $\u = \nabla^{\bot} \psi $ is automatically divergence-free. This fact brings a great deal of numerical convenience. 
It is also noticed that, all three evolutionary terms in~\eqref{NSE1} have a zero average: 
\begin{equation} 
	\int_\Omega \, \u \DOT\nabla  \omega \, d {\bf x} 
	= \int_\Omega \, \nabla \cdot ( \u \omega ) \, d {\bf x} 
	= 0 , \quad \int_\Omega \, \Delta \omega \, d {\bf x}  = 0 , \quad 
	\int_\Omega \, \f \, d {\bf x} = 0 , 
	\label{average-1} 
\end{equation}  
which comes from the divergence-free property of $\u$, combined with the periodic boundary condition. In turn, we conclude that $\omega ({\bf x}, t)$ always keeps a zero-average: 
\begin{equation} 
	\int_\Omega \, \partial_t \omega \, d {\bf x} = 0 ,  \qquad \mbox{so that} \, \, \, 
	\int_\Omega \, \omega ({\bf x}, t) \, d {\bf x} 
	=  \int_\Omega \, \omega ({\bf x}, 0) \, d {\bf x}  = 0 ,  \, \, \, \forall t > 0 . 
	\label{average-2} 
\end{equation}  
As a result, the kinematic equation $\Delta \psi = \omega$ always has a unique solution, if we impose a zero-average requirement for the stream function: $\int_\Omega \, \psi \, d {\bf x}= 0$. 

It is well-known that long-time stable (uniformly bounded in time) quantity for the NSE \eqref{NSE1}-\eqref{NSE3} is the enstrophy variable, namely $\frac12\|\omega\|_2^2$, so that the dynamics possesses a global attractor and invariant measures \cite{constantin1988navier,foias2001navier,temam2012infinite}. Consequently, preserving the long-time stability of key physical variables, such as vorticity, is crucial for accurately capturing long-term dynamics. Numerous efforts have been made to ensure stability and dissipativity in various forms (see for instance \cite{cheng2009semi,foias1991dissipativity,foias1994some,ju2002global,shen1989convergence,shen1990long,tone2009long,tone2006long,geveci1989convergence,larsson1989long} among many others). In the existing works, there are many off-the-shelf efficient solvers available, since the problem essentially reduces to solving a Poisson equation at each time step. A more popular method is IMEX, as it strikes a balance between stability (implicit treatment of the diffusion term) and efficiency (explicit treatment of the nonlinear term). Below, we provide a few examples to illustrate this.
\begin{itemize}
	\item The simplest example is the following first-order IMEX Euler scheme \cite{gottlieb2012long}, \begin{equation}\label{ex1}
		\frac{\omega^{n+1} - \omega^n}{\Delta t} + \bm{u}^n\cdot\nabla\omega^n = \nu\Delta\omega^{n+1} + \bm{f}^{n+1},
	\end{equation}
where $\Delta t$ is the time-step size and $\omega^k$ is the approximation of the vorticity at $t^k$. This scheme is efficient because it requires only two Poisson solvers per time step: one for the vorticity and another for the stream function. Moreover, the scheme's convergence over any fixed time interval is well-established and follows standard procedures, see the derivations in earlier literature \cite{gresho1991incompressible,gresho1991some,gresho1991some,he2008euler,medjo1996navier}. More importantly, the authors \cite{gottlieb2012long} proved that \eqref{ex1} is long-time stable in $L^2$ and $H^1$ and that the global attractor as well as the invariant measures of the scheme converge to those of the NSE at a vanishing time step.\\
\item  The IMEX BDF2 scheme takes the following form \cite{ascher1995implicit,karniadakis1991high,varah1980stability,heister2017unconditional},
\begin{equation}\label{ex2}
	\frac{3\omega^{n+1} - 4\omega^n + 2\omega^{n-1}}{2\Delta t} + 2\bm{u}^n \cdot\nabla\omega^n - \bm{u}^{n-1}\cdot\omega^{n-1} = \nu\Delta\omega^{n+1} + \bm{f}^{n+1},
\end{equation}
where the Adams-Bashforth treatment of the nonlinear term is in the form of linear multi-step fashion, the so-called extrapolated Gear's scheme. The authors \cite{heister2017unconditional} proved that the scheme's vorticity and velocity are both long-time stable in the $L^2$ and $H^1$ norms, without any time-step restriction.
\item The IMEX BDF2 scheme,
\begin{equation}\label{ex3}
			\frac{3\omega^{n+1} - 4\omega^n + 2\omega^{n-1}}{2\Delta t} + (2\bm{u}^n - \bm{u}^{n-1})\cdot\nabla(2\omega^n - \omega^{n-1}) = \nu\Delta\omega^{n+1} + \bm{f}^{n+1}
\end{equation}
was investigated by Wang \cite{wang2012efficient}. Compared to \eqref{ex2}, \eqref{ex3} features a specialized explicit Adams-Bashforth treatment for the advection term. The author primarily analyzed the scheme’s convergence in terms of long-time statistical properties (stationary statistics).
\item The IMEX second-order multi-step scheme \cite{gottlieb2012stability,wang2016local,cheng2016long},
\begin{equation}
	\begin{aligned}
			&\frac{\omega^{n+1} - \omega^n}{\Delta t} + \frac34\big(\bm{u}^n\cdot\nabla\omega^n + \nabla\cdot(\bm{u}^n\omega^n)\big) - \frac14\big(\bm{u}^{n-1}\cdot\nabla\omega^{n-1} + \nabla\cdot(\bm{u}^{n-1}\omega^{n-1})\big)\\
			=&\nu\Delta\big(\frac34\omega^{n+1}  + \frac14\omega^{n-1}\big) + \bm{f}^{n+1/2}
	\end{aligned}
\end{equation} 
 has been widely investigated (the third- and fourth-order schemes mentioned are omitted here). It is observed that Adams-Moulton interpolation causes the diffusion term to be more concentrated at the next time step, meaning that the coefficient at $t^{n+1}$ dominates the sum of the remaining diffusion coefficients.
\end{itemize}

In fact, the IMEX BDF3 scheme which we will adopt in this work has been applied to the incompressible NSEs with various spatial approximations \cite{ascher1995implicit}. However, the study addressed only linear stability analysis and the effectiveness of aliasing reduction in spectral methods, while the nonlinear analysis remains unavailable, even for local-in-time error estimates. Therefore, the main objective of this paper is to apply the IMEX BDF3 scheme to the 2-D NSE \eqref{NSE1}-\eqref{NSE3}, and demonstrate that, provided the time-step size is sufficiently small, the scheme \eqref{scheme-BDF3-1}-\eqref{scheme-BDF3-3} is long-time stable. Specifically, we establish the uniform-in-time $L^2$ and $H^m$ bounds of the fully discrete pseudo-spectral numerical solution for $m\ge1$ by the global-in-time energy stability. Furthermore, we demonstrate that the desired optimal convergence rate is achieved and rigorously proved using some skills analogous to the stability analysis.

The paper is organized as follows. We begin in Section 2 with reviewing the IMEX BDF3 method with the Fourier pseudo-spectral approximation for the NSE \eqref{NSE1}-\eqref{NSE3}. A few preliminary estimates in the Fourier pseudo-spectral space are also presented. The following two sections are devoted to the long-time stability and convergence analysis of the fully discrete scheme \eqref{scheme-BDF3-1}-\eqref{scheme-BDF3-3}. Some numerical experiments are provided in Section 5. The final section gives some concluding remarks.
\section{Fully discrete numerical scheme}

\subsection{Fourier pseudo-spectral differentiation} 
We assume that the domain is given by $\Omega = [0,L]^2$, with a uniform mesh size: $N_x = N_y=N$. We also choose a mesh size $h=L/N$ with $N=2K+1$ being a positive odd integer, and set $N_x = N_y = N$. The case for an even $N$ could be similarly treated. All the spatial variables are evaluated at the regular numerical grid $\Omega_N$, in which $x_i = i h$, $y_j=jh$, $0 \le i , j \le N$. Without loss of generality, it is also assumed that $L=1$. For a 2-D periodic function $f$ with a discrete Fourier expansion 
\begin{equation} 
	f_{i,j} = \sum_{k_1,l_1=-K}^{K}  
	\hat{f}_{k_1,l_1} \exp \left( 2 k_1 \pi {\rm i} x_i \right)  
	\exp \left( 2 l_1 \pi {\rm i} y_j \right) , 
	\label{spectral-coll-1}
\end{equation}
its Fourier collocation spectral approximations to first and second order partial derivatives (in the $x$ direction) become  
\begin{align} 
	\left( {\cal D}_{Nx} f \right)_{i,j} =& \sum_{k_1,l_1=-K}^{K}  
	\left( 2 k_1 \pi {\rm i} \right) \hat{f}_{k_1,l_1} 
	\exp \left( 2 \pi {\rm i} ( k_1 x_i + l_1 y_j ) \right) ,   \label{spectral-coll-2-1}     
	\\
	\left( {\cal D}_{Nx}^2 f \right)_{i,j} =& \sum_{k_1,l_1=-K}^{K}  
	\left( - 4 \pi^2 k_1^2 \right) \hat{f}_{k_1,l_1} 
	\exp \left( 2 \pi {\rm i} ( k_1 x_i + l_1 y_j) \right) . \label{spectral-coll-2-3}   
\end{align}
The ${\cal D}_{Ny}$ and ${\cal D}_{Ny}$ operators could be similarly defined. In turn, the discrete gradient, divergence and Laplacian are formulated as  
\begin{equation} 
	\nabla_N f = \left(  \begin{array}{c} 
		{\cal D}_{Nx} f  \\ 
		{\cal D}_{Ny} f  
	\end{array}  \right)  ,  \quad 
	\nabla_N \cdot \left(  \begin{array}{c} 
		f _1 \\ 
		f _2 
	\end{array}  \right)  = {\cal D}_{Nx} f_1 + {\cal D}_{Ny} f_2 ,  
	\quad \Delta_N f =  \left( {\cal D}_{Nx}^2  + {\cal D}_{Ny}^2 \right) f , 
	\label{spectral-coll-3}  
\end{equation}
at the point-wise level. Of course, an obvious fact is observed: $\nabla_N \cdot \nabla_N f = \Delta_N f$, for any grid function $f$. See the derivations in the related references \cite{boyd2001chebyshev,canuto1982approximation,gottlieb1977numerical}

Moreover, given any periodic grid functions $f$ and $g$ (over the 2-D numerical grid), the spectral approximations to the $L^2$ inner product and $L^2$ norm are introduced as 
\begin{eqnarray} 
	\left\| f \right\|_2 = \sqrt{ \left\langle f , f \right\rangle } ,  \quad \mbox{with}~
	\left\langle f , g \right\rangle  = h^2 \sum_{i,j=0}^{N -1}   f_{i,j} g_{i,j} . 
	\label{spectral-coll-inner product-1}
\end{eqnarray}
In fact, such a discrete $L^2$ inner product can also be viewed in the Fourier space other than in physical space, with the help of Parseval equality:
\be 
\left\langle f , g \right\rangle 
=   \sum_{k_1,l_1=-K}^{K}   
\hat{f}_{k_1,l_1}  \overline{\hat{g}_{k_1,l_1}}   
=   \sum_{k_1,l_1=-K}^{K}   
\hat{g}_{k_1,l_1}  \overline{\hat{f}_{k_1,l_1}}  , 
\label{spectral-coll-inner product-2} 
\ee 
in which $\hat{f}_{k_1,l_1}$, $\hat{g}_{k_1,l_1}$ are the Fourier coefficients of the grid functions $f$ and  $g$ in the expansion as in (\ref{spectral-coll-1}). Furthermore, a detailed calculation reveals the following formulas of integration by parts formulas at the discrete level (see the related discussions \cite{cheng2015fourier,gottlieb2012stability}): 
\begin{equation} 
	\left\langle f ,  \nabla_N \cdot \left(  \begin{array}{c} 
		g_1 \\ 
		g_2 
	\end{array}  \right)  \right\rangle  = - \left\langle \nabla_N f ,  \left(  \begin{array}{c} 
		g_1 \\ 
		g_2 
	\end{array}  \right)  \right\rangle ,   \qquad 
	\left\langle f ,  \Delta_N  g  \right\rangle  
	= - \left\langle \nabla_N f ,  \nabla_N g   \right\rangle  .  
	\label{spectral-coll-inner product-3}  
\end{equation} 

In addition, a discrete average is introduced for any periodic grid function $f$: $\overline{f} = \langle f, {\bf 1} \rangle$. Such a notation will facilitate the notations in the later analysis. 

\subsubsection{A few preliminary estimates in the Fourier pseudo-spectral space} 

It is well-known that the existence of aliasing error in the nonlinear term poses a serious challenge in the numerical analysis of Fourier pseudo-spectral scheme. To overcome this subtle difficulty, a periodic extension of a grid function and a Fourier collocation interpolation operator is introduced. For simplicity of presentation, we denote ${\cal B}^K$ as the space of trigonometric polynomials in $x$, $y$, and $z$ of degree up to $K$.

\begin{defi} 
	For any periodic grid function $f$ defined over a uniform 2-D numerical grid, we denote $\f_N$ as its periodic extension. In more detail, assume that the grid function $f$ has a discrete Fourier expansion as 
	\begin{equation} 
		f_{i,j}  = \sum_{k_1,l_1=-K}^{K}  
		\hat{f}_{k_1,l_1} 
		\exp \left( 2 \pi {\rm i} ( k_1 x_i + l_1 y_j) \right) , 
		\label{spectral-coll-projection-1}
	\end{equation} 
	its periodic extension (projection) into ${\cal B}^K$ is given by 
	\begin{equation} 
		\f_N ({\bf x})  = \sum_{k_1,l_1=-K}^{K}  
		\hat{f}_{k_1,l_1} 
		\exp \left( 2 \pi {\rm i} ( k_1 x + l_1 y ) \right) .
		\label{spectral-coll-projection-2}
	\end{equation} 
	Moreover, for any periodic continuous function ${\bf f}$, which may contain larger wave length, we define its collocation interpolation operator as  
	\begin{equation} \label{spectral-coll-projection-3}
	\begin{aligned} 
		{\bf f}_{i,j}  =& \sum_{k_1,l_1=-K}^{K}  
		(\hat{f}_c)_{k_1,l_1} 
		\exp \left( 2 \pi {\rm i} ( k_1 x_i + l_1 y_j ) \right) ,   \nonumber 
		\\
		{\cal I}_N {\bf f} ({\bf x})  =& \sum_{k_1,l_1=-K}^{K}  
		(\hat{f}_c)_{k_1,l_1} 
		\exp \left( 2 \pi {\rm i} ( k_1 x + l_1 y ) \right) .
\end{aligned}
	\end{equation} 
	Note that $\hat{f}_c$ may not be the Fourier coefficients of $\f$, because of the aliasing errors. 
\end{defi}

To overcome a key difficulty associated with the $H^k$ bound of the nonlinear term obtained by collocation interpolation, the following lemma is introduced. In fact, the case of $k=0$ was proven in Weinan E's earlier works \cite{weinan1992convergence,weinan1993convergence}. The case of $k \ge 1$ was analyzed in a recent work by C. Wang and S. Gottlieb \cite{gottlieb2012stability}. We cite the result here.

\begin{lem}  \label{lem:aliasing} 
	For any $\varphi \in {\cal B}^{2K}$ in dimension $d$, we have 
	\begin{equation} 
		\left\| {\cal I}_N \varphi \right\|_{H^k}  
		\le  ( \sqrt{2} )^d  \left\|  \varphi \right\|_{H^k} ,\quad \forall k \in\mathbb{Z},k\ge0.
		\label{spectral-coll-projection-4}
	\end{equation} 
\end{lem} 

The discrete version of discrete Poincar\'e inequality, elliptic regularity, as well as an $\ell^2$ orthogonality of a scalar field $f$ with its convection term, is stated in the following lemma. The proof will be straightforward, and we skip the details for the sake of brevity. 

\begin{lem} \label{lem:Poincare} 
The following statements hold:
\begin{enumerate}[(i)]
	\item (\textbf{Discrete Poincar\'e inequality})
	There exists a constant $C_2>0$, depending only on the domain size $L$
	(but independent of the mesh size $h$ and the time-step size $\Delta t$), such that
	\begin{equation}\label{lem 2-1}
		\| f \|_2 \le C_2 \| \nabla_N f \|_2 ,
		\qquad \text{for any grid function $f$ with } \overline{f} = 0 .
	\end{equation}
	\item (\textbf{Skew-symmetry of the convection operator})
	Let $\u$ be a discrete velocity field satisfying the discrete divergence-free
	condition $\nabla_N\cdot\u = 0$. Then, for any grid function $f$,
	\begin{equation}\label{lem 2-2}
		\left\langle  f ,  \u \DOT \nabla_N f 
		+ \nabla_N \cdot \left( \u f \right)   \right\rangle      
		= \left\langle  f ,  \u \DOT \nabla_N f  \right\rangle 
		- \left\langle \nabla_N f , \u f   \right\rangle   
		= 0  .
	\end{equation}
\end{enumerate}
\end{lem}

In addition, the following estimate for the skew-average convection term will be repeatedly applied in the later analysis. 

\begin{lem}  \label{lem:convection} 
Given a discrete velocity vector $\u = (u, v)$ with $\nabla_N \DOT \u \equiv 0$.
Assume its continuous extension is given by $\U = (U, V) = \u_N$, and the
associated stream function and vorticity, in the continuous version, are
determined to be $\ppsi$ and $\pomega$. In other words, we have
$\U = \nabla^{\bot} \ppsi$ and $\pomega = \nabla \times \U = \Delta \ppsi$.

Let $\delta>1$ be a fixed constant such that the Sobolev embedding
$H^{\delta}(\Omega)\hookrightarrow W^{1,\infty}(\Omega)$ holds for the
vorticity field in two dimensions. Then we have the estimate 
\begin{align}  
	\| \u \DOT \nabla_N f + \nabla_N \cdot ( \u f ) \|_2 
	\le \gamma_0 \| \pomega \|_{H^\delta}  
	\cdot \| \nabla_N f \|_2 , \quad 
	\mbox{for any $f$ with $\overline{f} =0$} . \label{lem 3-0} 
\end{align} 
\end{lem} 

\begin{proof} 
	In fact, it is noticed that the term $\nabla_N \cdot (\u f)$ cannot be expanded as $\u \cdot \nabla_N f$, as in the Fourier-Galerkin approximation, even though $\u$ is divergence-free at the discrete level. In the Fourier pseudo-spectral, we have to start from 
	\begin{equation} 
		\nabla_N \cdot (\u f) 
		= {\cal D}_{Nx} (u f)  + {\cal D}_{Ny} (v f)   . 
		\label{lem 3-1}
	\end{equation} 
	To obtain an estimate of these nonlinear expansions, we denote $\U= (U, V)$ and $f_{\bf F}$ as the continuous projection of $\u$ and $f$, respectively. In particular, we have $\int_\Omega \, f_{\bf F} \, d {\bf x} =0$, due to the fact that $\overline{f}=0$. Moreover, since $\U, f_{\bf F} \in {\cal B}^K$, an application of Lemma~\ref{lem:aliasing} indicates that 
	\begin{align} 
		\|  {\cal D}_{Nx} (u f)  \|_2  
		=& \|  \partial_x {\cal I}_N ( U f_{\bf F} )  \|_{L^2} 
		\le 2  \|  \partial_x ( U f_{\bf F} )  \|_{L^2}  ,  \nonumber 
		\\
		\|  {\cal D}_{Ny} (v f)  \|_2  
		=& \|  \partial_y {\cal I}_N ( V f_{\bf F}) \|_{L^2} 
		\le 2  \|  \partial_y ( V f_{\bf F} ) \|_{L^2}   .   \label{lem 3-2}
	\end{align}  
	Subsequently, a detailed expansion in continuous space and application of H\"older's inequality indicates that 
	\begin{align} 
		\|  \partial_x ( U f_{\bf F} )  \|_{L^2}  
		=&  \|  U_x f_{\bf F} +  U ( f_{\bf F} )_x  \|_{L^2}  \nonumber\\
		\le& \|  U_x f_{\bf F} \|_{L^2} + \| U (f_{\bf F} )_x  \|_{L^2}   \nonumber 
		\\
		\le& 
		\|  U_x  \|_{L^{2/(1-\delta)}}  \cdot  \| f_{\bf F} \|_{L^{2/\delta}} 
		+ \| U  \|_{L^\infty} \cdot  \| ( f_{\bf F} )_x \|_{L^2}  . 
		\label{lem 3-3}
	\end{align}  
	Meanwhile, an application of elliptic regularity implies that 
	\begin{equation} 
		\|  U_x  \|_{H^\delta} \le C \| U \|_{H^{1+\delta}} 
		\le  C \| \ppsi \|_{H^{2+\delta}} 
		\le  C  \|  \pomega  \|_{H^\delta}  , 
	\end{equation} 
	since $\Delta \ppsi = \pomega$, and $\U = \nabla^{\bot} \ppsi$. Subsequently, an application of 2-D Soblev embedding leads to 
	\begin{equation} 
		\|  U_x  \|_{L^{2/(1-\delta)}}  \cdot  \| f_{\bf F} \|_{L^{2/\delta}} 
		\le  C  \|  U_x  \|_{H^\delta} \cdot  \| f_{\bf F} \|_{H^1}   
		\le  C  \|  \pomega  \|_{H^\delta} 
		\cdot  \| \nabla f_{\bf F} \|_{L^2}   ,  \label{lem 3-4}
	\end{equation}
	in which Poincar\'e inequality was utilized in the last step. The second part in~\eqref{lem 3-3} can be handled in a straightforward way: 
	\begin{equation} 
		\| U \|_{L^\infty} \cdot  \| ( f_{\bf F} )_x  \|_{L^2}  
		\le   C \| U  \|_{H^{1+\delta}} \cdot  \| \nabla f_{\bf F}  \|_{L^2}    
		\le   C \| \pomega  \|_{H^\delta} \cdot  \| \nabla f_{\bf F}  \|_{L^2}  ,   \label{lem 3-5}
	\end{equation}
	with the the elliptic regularity~\eqref{est-coll-prelim-1} applied again in the second step. A combination of \eqref{lem 3-4} and \eqref{lem 3-5} yields 
	\begin{equation} 
		\|  \partial_x ( U f_{\bf F} )  \|_{L^2}   
		\le  C \| \pomega  \|_{H^\delta} \cdot  \| \nabla f_{\bf F}  \|_{L^2}  .  \label{lem 3-6}
	\end{equation}
	Similar estimates can be derived for $\|  \partial_y ( V f_{\bf F} )  \|_{L^2}$.
	Going back to~\eqref{lem 3-4}, we obtain
	\begin{equation} 
		\|  \nabla_N \cdot (\u f)   \|_2   
		\le \gamma_1 \| \pomega  \|_{H^\delta} \cdot  \| \nabla f_{\bf F}  \|_{L^2} 
		\le \gamma_1 \| \pomega^n  \|_{H^\delta} \cdot  \| \nabla_N f  \|_2 ,
		\label{lem 3-7}
	\end{equation}
	where $\pomega^n$ denotes the vorticity associated with the discrete velocity field $\u^n$ at time level $t^n$, $\gamma_1>0$ is a generic constant independent of the spatial grid size, and the second step comes from the fact that $\pomega, f_{\bf F} \in {\cal B}^K$, so that the corresponding $L^2$ and $H^1$ norms are equivalent between the continuous projection and the discrete version.In addition, the other nonlinear term can be estimated as follows:
	\begin{equation} 
		\begin{aligned} 
			\| \u^n \cdot \nabla_N f \|_2  
			\le & \left\| \u^n \right\|_{\infty} \, \left\| \nabla_N f \right\|_2
			\le \left\| \U^n \right\|_{\infty} \, \left\| \nabla_N f \right\|_2  \\
			\le & C \| \U^n \|_{H^{1+\delta}} \, \| \nabla_N f \|_2  
			\le \gamma_2 \| \pomega^n \|_{H^\delta} \, \| \nabla_N f \|_2 ,
		\end{aligned} 
		\label{lem 3-8}
	\end{equation}
	where $\u^n$ denotes the discrete velocity field at time level $t^n$, and
	$\U^n=(\u^n)_N$ is its continuous Fourier interpolation. The constant
	$\gamma_2>0$ depends only on the domain and the Sobolev embedding constants,
	but is independent of the grid size and the time step.
	Consequently, a combination of~\eqref{lem 3-7} and \eqref{lem 3-8} yields \eqref{lem 3-0}, by taking $\gamma_0 = \gamma_1 + \gamma_2$. This finishes the proof of Lemma~\ref{lem:convection}. 
\end{proof}

\subsection{IMEX BDF3 numerical scheme} 

For the 2-D incompressible NSE~\eqref{NSE1}-\eqref{NSE3}, we treat the nonlinear convection term explicitly for the sake of numerical convenience, and the diffusion term implicitly to preserve an $L^2$ stability. This semi-implicit scheme leads to a Poisson-like equation at each time step, which makes the numerical algorithm extremely efficient. After the scalar vorticity is updated, the stream function can be determined through the kinematic equation, a Poisson equation. In turn, the velocity is computed as the perpendicular gradient of the stream function. 

The IMEX BDF3 numerical scheme is given by 
\begin{align} 
	&
	\frac{\frac{11}{6} \omega^{n+1} - 3 \omega^n + \frac32 \omega^{n-1} - \frac13 \omega^{n-2}}{\dt}  
	+ \frac32 ( \u^n \DOT \nabla_N \omega^n  
	+ \nabla_N \cdot ( \u^n \omega^n ) - \overline{\u^n \DOT \nabla_N \omega^n} )  
	\nonumber 
	\\
	& 
	- \frac32 ( \u^{n-1} \DOT \nabla_N \omega^{n-1}  
	+ \nabla_N \cdot ( \u^{n-1} \omega^{n-1} ) 
	- \overline{\u^{n-1} \DOT \nabla_N \omega^{n-1}} )   \nonumber 
	\\
	& 
	+  \frac12 ( \u^{n-2} \DOT \nabla_N \omega^{n-2}  
	+ \nabla_N \cdot ( \u^{n-2} \omega^{n-2} ) 
	- \overline{\u^{n-2} \DOT \nabla_N \omega^{n-2}} ) 
	= \nu \Delta_N  \omega^{n+1} 
	+ \f_N^{n+1} ,   \label{scheme-BDF3-1}
	\\
	&
	- \Delta_N \psi^{n+1} = \omega^{n+1} ,   \label{scheme-BDF3-2} 
	\\
	&
	\u^{n+1} = \nabla_N^{\bot} \psi^{n+1}  
	= \left( {\cal D}_{Ny} \psi^{n+1} , -{\cal D}_{Nx} \psi^{n+1} \right) ,  \label{scheme-BDF3-3}
\end{align}  
in which $\f_N$ is the Fourier projection of the external force term $\f$ onto the space of ${\cal B}^K$. It is clear that 
\begin{equation} \nonumber
	\overline{\f_N^\ell} = \frac{1}{|\Omega|} \int_\Omega \, \f_N \, d {\bf x} =0 ,  
\end{equation}  
since $\f_N \in {\cal B}^K$, and $ \int_\Omega \, \f_N \, d {\bf x} 
		= \int_\Omega \, \f \, d {\bf x} = 0$. At each time step $t^\ell$, the nonlinear term turns out to be a spectral approximation to $\u \DOT \nabla \omega$ at time step $t^n$. The reason for such a skew-average treatment will be observed in the later analysis. In addition, an average constant, namely $ \overline{\u^\ell \DOT \nabla_N \omega^\ell}$, has to be subtracted at each time step. In fact, such a correction term is of spectral accuracy, due to the following fact: 
\begin{equation} 
	\u^\ell \DOT \nabla_N \omega^\ell = \nabla_N \cdot ( \u^\ell \DOT \omega^\ell ) + O (h^m), 
	\quad \mbox{and} \, \, \, \overline{ \nabla_N \cdot ( \u^\ell \DOT \omega^\ell )} =0 , 
	\label{average-4} 
\end{equation}  
while the second identity comes from the summation-by-parts formula. Such a spectrally accurate, constant correction term is to ensure the zero-average property for the numerical solution at a discrete level, as stated in the following proposition. 


\begin{prop} \label{prop:average} 
	If the initial data satisfies $\overline{\omega^0} = \overline{\omega^{-1}} = \overline{\omega^{-2}} =0$, we have $\overline{\omega^n} = 0$ ($\forall n \ge 3$), for the multi-step numerical scheme~\eqref{scheme-BDF3-1}-\eqref{scheme-BDF3-3}. 
\end{prop} 

\begin{proof} 
	It is clear that, at each time step $t^\ell$, $\ell = n, n-1, n-2$, the following identities are valid: 
	\begin{equation}  \nonumber
			\overline{\u^\ell \DOT \nabla_N \omega^\ell - \overline{\u^\ell \DOT \nabla_N \omega^\ell} } 
			= \overline{\u^\ell \DOT \nabla_N \omega^\ell} 
			- \overline{\u^\ell \DOT \nabla_N \omega^\ell} = 0 ,  \quad 
			\overline{ \nabla_N \cdot ( \u^\ell \DOT \omega^\ell )} =0 ,  
	\end{equation} 
which in turn gives $\overline{\u^\ell \DOT \nabla_N \omega^\ell + \nabla_N \cdot ( \u^\ell \DOT \omega^\ell ) - \overline{\u^\ell \DOT \nabla_N \omega^\ell} } = 0$. 
	In addition, the diffusion term and the external force term are always average-free: 
	\begin{equation} \nonumber
		\overline{\Delta_N \omega^{n+1}} = 0 , \quad 
		\overline{\f_N^{n+1}} = 0 .
	\end{equation} 
	Therefore, by taking a discrete average of the numerical algorithm~\eqref{scheme-BDF3-1} on both sides, we obtain 
	\begin{equation*} 
			\frac{11}{6} \overline{\omega^{n+1}} - 3 \overline{\omega^n} 
			+ \frac32 \overline{\omega^{n-1}} - \frac13 \overline{\omega^{n-2}} = 0 ,  
	\end{equation*}
so that
	\begin{equation*}
		\overline{\omega^{n+1}} = 0 ,  \quad \mbox{if} \, \, \,  
		\overline{\omega^n} = \overline{\omega^{n-1}} = \overline{\omega^{n-2}} = 0 ,  \quad \forall n \ge 2 . 
	\end{equation*} 
	This finishes the proof of Proposition~\ref{prop:average}. 
\end{proof}  
\begin{rem}
	We emphasize that Proposition~\ref{prop:average} is concerned with the BDF3 recurrence \eqref{scheme-BDF3-1} for $n\ge2$, which yields $\overline{\omega^{n+1}}=0$ for all $n\ge2$ once the three previous time levels are mean-free. In practice, the first two time levels $\omega^1$ and $\omega^2$ are generated by the prescribed initial guesses $\omega^{-2},\omega^{-1},\omega^{0}$ (and the corresponding starting-up procedure), which are chosen to be mean-free. Hence the relation $\overline{\omega^1}=\overline{\omega^2}=0$ holds by construction, and the zero-average property is preserved for all $n\ge1$.
\end{rem}
\begin{rem}
For the three-dimensional incompressible fluid, the formulation of the stream function potential and vorticity vector would lead to a double curl problem, and its ill-posed feature has been well-known. In more details, both the stream function potential and vorticity are three-dimensional vector fields, and the kinematic equation is given by
\begin{equation} \nonumber
  \nabla \times (\nabla \times \vec{\psi}) = - \Delta \vec{\psi} + \nabla (\nabla \cdot \vec{\psi} ) = \nabla \times \u = \vec{\omega} . 
\end{equation} 
In comparison with the standard Poisson equation, this equation is ill-posed, due to a point-wise constraint $\nabla \cdot \vec{\psi} \equiv 0$. Of course, this kinematic equation is reduced to a standard Poisson equation, if such a divergence-free constraint is satisfied for the stream function potential. However, an enforcement of the divergence-free condition for the stream function potential turns out to be a very challenging issue, at both the PDE analysis and numerical levels.

On the other hand, in the two-dimensional space, such a difficulty associated with the ill-posed feature of the double curl operator is automatically avoided. In this case, both the discrete stream function and vorticity are scalar functions. Again, if the discrete stream function and vorticity vectors are denoted as $\vec{\psi}^{n+1} = (0, 0, \psi^{n+1})^T$, $\vec{\omega}^{n+1} = ( 0, 0, \omega^{n+1})^T$, respectively, a direct calculation implies that the kinematic equation automatically becomes a standard Poisson equation:
\begin{equation} \nonumber
  \nabla \times (\nabla \times \vec{\psi}) = - \Delta \vec{\psi} + \nabla (\nabla \cdot \vec{\psi} ) = - \Delta \vec{\psi} = \vec{\omega} , 
\end{equation} 
 based on the fact that  $\nabla \cdot \vec{\psi} = \partial_z \psi \equiv 0$. In other words, for the two-dimensional incompressible fluid, the double curl problem is not ill-posed any more; only a standard Poisson solver is needed at each time step.
\end{rem}
As a result of this proposition, it is obvious that the kinematic equation~\eqref{scheme-BDF3-2} has a unique solution, with a requirement that $\overline{\psi^{n+1}} =0$. Moreover, it is observed that the numerical velocity $\u^\ell  = \nabla_N^{\bot} \psi^\ell$ is automatically divergence-free: 
\begin{equation}\nonumber
	\nabla_N \cdot \u^\ell = {\cal D}_{Nx} u^\ell + {\cal D}_{Ny} v^\ell  
	= {\cal D}_{Nx} ( {\cal D}_{Ny} \psi^\ell) - {\cal D}_{Ny} ( {\cal D}_{Nx} \psi^\ell)   
	=0 ,  \quad \forall \ell \ge 0 , 
\end{equation}  
due to the commutative property of the discrete differential operators. In terms of the skew-average treatment of the nonlinear terms, a careful application of summation-by-parts formula \eqref{lem 2-2} gives 
\begin{equation} \nonumber
	\left\langle  \omega ,  \u \DOT \nabla_N \omega 
	+ \nabla_N \cdot \left( \u \omega \right)   \right\rangle      
	= \left\langle  \omega ,  \u \DOT \nabla_N \omega  \right\rangle 
	- \left\langle \nabla_N \omega , \u \omega   \right\rangle    
	= 0  . 
\end{equation}    
In other words, the nonlinear convection term appearing in the numerical scheme~\eqref{scheme-BDF3-1}, so-called Temam technique, makes the nonlinear term orthogonal to the vorticity field in the discrete $L^2$ space, without considering the temporal discretization. This property is crucial in the stability analysis for the fully discrete numerical scheme; see the related references \cite{gottlieb2012long,wang2012efficient,cheng2016long}.

In addition, we denote $\U^\ell= (U^\ell, V^\ell)$, $\pomega^\ell$ and $\ppsi^\ell$ as the continuous projection of $\u^\ell$, $\omega^\ell$ and $\psi^\ell$, respectively, with the projection formula given by~\eqref{spectral-coll-projection-1}, \eqref{spectral-coll-projection-2}. It is clear that $\U^n, \pomega^n, \ppsi^n \in {\cal B}^K$ and the kinematic equation $\Delta \ppsi^\ell = \pomega^\ell$, $\U^\ell = \nabla^{\bot} \ppsi^\ell$ is satisfied at the continuous level. Because of these kinematic equations, an application of elliptic regularity shows that 
\begin{equation} 
	\left\| \ppsi^n \right\|_{H^{m+2}}  \le  C  \left\| \pomega^n \right\|_{H^m} ,  \quad 
	\left\| \ppsi^n \right\|_{H^{m+2+\alpha}}  
	\le  C  \left\| \pomega^n \right\|_{H^{m+\alpha}}  ,  
	\label{est-coll-prelim-1}
\end{equation} 
in which we normalize the stream function with $\int_\Omega \, \ppsi^\ell \, d {\bf x} = 0$. Meanwhile, since all the profiles have mean zero over the domain: 
\begin{equation} \nonumber
	\overline{\ppsi^n} = 0 ,  \quad 
	\overline{\U^n} = \left( \overline{ \partial_y \ppsi^n} , 
	- \overline{ \partial_x \ppsi^n} \right) = 0 ,  \quad 
	\overline{\pomega^n} = \overline{\Delta \ppsi^n} = 0 ,   
\end{equation} 
all the Poincar\'e inequality and elliptic regularity could be effectively applied.  

In terms of the long-time stability analysis, the main theoretical result is stated in the following theorem. 

\begin{thm}  \label{thm:stability} 
	Let $\omega_0 \in H^2$ and let $\omega^n$ be the discrete solution of the proposed third order numerical scheme~\eqref{scheme-BDF3-1}-\eqref{scheme-BDF3-3}, with initial guess data 
	$\omega^{-j}$ being bounded in $H^2$ norm ($1 \le j \le 2$). Denote $\pomega^n$ as the continuous extension of $\omega^n$ in space, given by (\ref{spectral-coll-projection-3}). Also, let  $\f \in L^\infty(0,T;L^2)$and set $\|\f\|_{\infty} := \|\f\|_{L^\infty(0,T;L^2)}$. Then, there exists $M_0 = M_0(\|\omega_0\|_{H^2}, \nu, \|\f\|_{\infty})$ such that if
	\begin{equation} \label{constraint-BDF3-dt}
		\dt \leq \frac{\nu}{C_w M_0^2},  \quad 
		\mbox{$C_w$ is a constant only dependent on $\Omega$}, 
	\end{equation}
	then
	\begin{align}  
		\| \pomega^n \| \le& \alpha_0^{-1} \Big( 1 + \beta_0 \nu \dt  \Big)^{-\frac{n}{2}}  
		(F^0)^\frac12  + R^{(0)} , \quad  \forall\, n\ge0, \label{BDF3-est-L2-2}  
		\\
		\| \nabla \pomega^n \| \le& \alpha_0^{-1} \Big( 1 + \beta_0 \nu \dt  \Big)^{-\frac{n}{2}}  
		(G_1^0)^\frac12  + R^{(1)} , \quad  \forall\, n\ge0, \label{BDF3-est-H1-2}     
	\end{align}
	in which the initial quantities $F^0$ and $G_1^0$ will be specified in~\eqref{est-BDF3-L2-8-6} and \eqref{est-BDF3-H1-8-4}, respectively, $\alpha_0$ and $\beta_0$ are constants associated with the numerical temporal stencil and elliptic regularity, $R^{(0)}$ and $R^{(1)}$ are uniform-in-time constants only dependent on $\| \pomega^0 \|_{H^2}$, $\| \f \|_\infty$, $\Omega$ and $\nu$.  
\end{thm} 
\begin{rem}
	For the continuous 2-D NSE, the classical theory guarantees a global-in-time $L^2$ bound on the vorticity under the sole assumption $\omega_0\in L^2$. However, the multi-step nature of the proposed IMEX BDF3 scheme \eqref{scheme-BDF3-1}-\eqref{scheme-BDF3-3} makes the numerical analysis much more complicated than the one for the continuous equation. In more details, the $L^2$-style quantity $F^n$, introduced in \eqref{est-BDF3-L2-8-6} and proved to have a global-in-time bound, contains a few $H^1$ terms of the numerical solution, with coefficients of order $\mathcal{O}(\Delta t)$. Therefore, to theoretically derive a uniform bound for $F^n$ in time, we need to require an $H^1$ regularity class of the initial data.

Moreover, based on the time step size constraint \eqref{constraint-BDF3-dt-1} (needed in the nonlinear estimate), a global-in-time $H^1$ estimate of the numerical solution becomes necessary to accomplish the theoretical analysis. Similarly, $H^1$-style quantity $G_1^n$, introduced in \eqref{est-BDF3-H1-8-4}, contains a few $H^2$ terms of the numerical solution, with coefficients of order $\mathcal{O}(\Delta t)$. As a consequence, to theoretically derive a uniform bound for $G_1^n$ in time, we have to require an $H^2$ regularity class of the initial data.

The global-in-time bounds of $F^n$ and $G_1^n$ are necessary steps to derive both the uniform-in-time $L^2$ and $H^1$ bounds of the numerical solutions, and these two estimates have to be combined together. In other words, because of the time step size constraint \eqref{constraint-BDF3-dt-1}, a global-in-time estimate for $F^n$ could not be independently accomplished, and the bound estimate for $G_1^n$ plays an essential role to recover the value of $\tilde{C}_1$, which in turn requires an $H^2$ initial data. Therefore, the stronger regularity assumption on the initial vorticity reflects a limitation of the current analytical technique rather than a contradiction with the classical continuous theory. This assumption is technical and is not intended to be optimal.
\end{rem}

\section{Proof of Theorem \ref{thm:stability}: long-time stability analysis for the third-order multi-level scheme \eqref{scheme-BDF3-1}-\eqref{scheme-BDF3-3}}	
	Before the detailed stability analysis, we present the telescope formula 
	for the third-order BDF temporal discretization operator in the following lemma. 
	
	\begin{lem} \label{lem:BDF3-stencil}
		For the third-order BDF temporal discretization operator, there exists $\alpha_i$, $i=1,\ldots,10$, $\alpha_1 \ne 0$, such that 
		\begin{equation} 
			\begin{aligned} 
				& \left\langle \frac{11}{6} f^{n+1} - 3 f^n + \frac32 f^{n-1} - \frac13 f^{n-2} ,
				2 f^{n+1} - f^n\right \rangle \\
				=&  \| \alpha_1 f^{n+1} \|_2^2 - \| \alpha_1 f^n \|_2^2
				+ \| \alpha_2 f^{n+1} + \alpha_3 f^n \|_2^2
				- \| \alpha_2 f^n + \alpha_3 f^{n-1} \|_2 ^2 
				\\
				&
				+ \| \alpha_4 f^{n+1} + \alpha_5 f^n + \alpha_6 f^{n-1} \|_2^2
				- \| \alpha_4 f^n + \alpha_5 f^{n-1} + \alpha_6 f^{n-2} \|_2^2 
				\\
				&+ \| \alpha_7 f^{n+1} + \alpha_8 f^n + \alpha_9 f^{n-1}
				+ \alpha_{10} f^{n-2} \|_2^2 . 
			\end{aligned} 
			\label{BDF3-telescope-1} 
		\end{equation} 
	\end{lem}
	
	Meanwhile, the telescope formula~\eqref{BDF3-telescope-1} is not sufficient to control the nonlinear inner product. To overcome this difficulty, we need additional stabilization terms in the temporal differentiation. In more details, we begin with the following equivalent form of the BDF3 temporal stencil
	\begin{equation*} 
		\frac{11}{6} f^{n+1} - 3 f^n + \frac32 f^{n-1} - \frac13 f^{n-2}  =  \frac23 ( f^{n+1} - f^n )  + \frac76 ( f^{n+1} - 2 f^n + f^{n-1} ) 
		+ \frac13 ( f^{n-1} - f^{n-2} ) , 
	\end{equation*}
	which in turn givers the following estimate
	\begin{equation}   
		\label{BDF3-telescope-2} 
		\begin{aligned} 
			&  \left\langle  \frac{11}{6} f^{n+1} - 3 f^n + \frac32 f^{n-1} 
			- \frac13 f^{n-2} ,  f^{n+1} - f^n  \right\rangle 
			\\
			= & 
			\frac23  \| f^{n+1} - f^n \|_2^2 
			+ \frac76 \langle f^{n+1} - 2 f^n + f^{n-1} , f^{n+1} - f^n \rangle   
			+ \frac13 \langle f^{n-1} - f^{n-2} , f^{n+1} - f^n \rangle 
			\\
			\ge &  \frac23  \| f^{n+1} - f^n \|_2^2  - \frac16 \| f^{n+1} - f^n \|_2^2 - \frac16 \| f^{n-1} - f^{n-2} \|_2^2   
			\\
			&  + \frac{7}{12} \left( \| f^{n+1} - f^n \|_2^2 
			- \| f^n - f^{n-1} \|_2^2  + \| f^{n+1} - 2 f^n + f^{n-1} \|_2^2  \right) 
			\\
			= & \frac{13}{12}  \| f^{n+1} - f^n \|_2^2 
			- \frac{7}{12} \| f^n - f^{n-1} \|_2^2 
			- \frac16 \| f^{n-1} - f^{n-2} \|_2^2    
			+ \frac{7}{12} \| f^{n+1} - 2 f^n + f^{n-1} \|_2^2  .  
		\end{aligned} 
	\end{equation} 
	A combination of~\eqref{BDF3-telescope-1} and \eqref{BDF3-telescope-2} results in 
	\begin{equation} 
		\begin{aligned} 
			& \left\langle \frac{11}{6} f^{n+1} - 3 f^n + \frac32 f^{n-1} - \frac13 f^{n-2} ,
			3 f^{n+1} - 2 f^n \right\rangle \\
			\ge &  \| \alpha_1 f^{n+1} \|_2^2 - \| \alpha_1 f^n \|_2^2
			+ \| \alpha_2 f^{n+1} + \alpha_3 f^n \|_2^2
			- \| \alpha_2 f^n + \alpha_3 f^{n-1} \|_2 ^2 
			\\
			&
			+ \| \alpha_4 f^{n+1} + \alpha_5 f^n + \alpha_6 f^{n-1} \|_2^2
			- \| \alpha_4 f^n + \alpha_5 f^{n-1} + \alpha_6 f^{n-2} \|_2^2 
			\\
			& 
			+ \frac{13}{12}  \| f^{n+1} - f^n \|_2^2 - \frac{7}{12} \| f^n - f^{n-1} \|_2^2 
			- \frac16 \| f^{n-1} - f^{n-2} \|_2^2    
			+ \frac{7}{12} \| f^{n+1} - 2 f^n + f^{n-1} \|_2^2 . 
		\end{aligned} 
		\label{BDF3-telescope-3} 
	\end{equation} 
	This inequality will play an important role in the global-in-time analysis. 
	
	To deal with a multi-step method, the $H^{\delta}$  a-priori bound is assumed for the numerical solution at all previous time steps: 
	\begin{equation}
		\| \pomega^k \|_{H^\delta} \le \tilde{C}_1 , 
		\quad  \forall 1 \le k \le n , \label{a priori-1}
	\end{equation}  
	for some $\delta > 0$. Note that $\tilde{C}_1$ is a global constant in time. We are going to prove that such a bound for the numerical solution is also available at time step $t^{n+1}$. Of course, an application of induction could justify such an a-priori assumption. 
	
	As a result of this a-priori assumption, a careful application of Lemma~\ref{lem:convection} reveals that 
	\begin{equation} 
		\| \u^\ell \DOT \nabla_N  \omega^\ell 
		+ \nabla_N \cdot ( \u^\ell  \omega^\ell ) \|_2 
		\le  \gamma_0 \| \pomega^\ell \|_{H^\delta}  \cdot  
		\|  \nabla_N  \omega^\ell  \|_2  
		\le \gamma_0 \tilde{C}_1 \|  \nabla_N  \omega^\ell  \|_2 , 
		\label{a priori-2} 
	\end{equation} 
	for $\ell = n, n-1, n-2$, due to the fact that $\pomega^\ell$ is the associated vorticity variable in the continuous version, corresponding to $\u^\ell$. 
	
	Meanwhile, a careful application of Sobolev interpolation inequality gives 
	\begin{equation} 
		\| \pomega^\ell  \|_{H^\delta} 
		\le  C \| \pomega^\ell \|_{L^2}^{1-\delta} \cdot  
		\|  \pomega^\ell  \|_{H^1}^\delta  
		\le Q_0 \| \pomega^\ell \|_{L^2}^{1-\delta} \cdot  
		\|  \nabla \pomega^\ell  \|^\delta  
		\le Q_0 \| \omega^\ell \|_2^{1-\delta} \cdot  
		\|  \nabla_N \omega^\ell  \|_2^\delta   , 
		\label{a priori-3} 
	\end{equation} 
	for $\ell = n, n-1, n-2$, in which the constant $Q_0$ only depends on $\Omega$. The second step comes from the Poincar\'e inequality, combined with the fact that $\int_\Omega \, \pomega^\ell \, d {\bf x}=0$; the third step is based on the fact that $\pomega^\ell \in {\cal B}^K$, so that $\| \pomega^\ell \|_{L^2} = \| \omega^\ell \|_2$, and $\| \nabla \pomega^\ell \| = \| \nabla_N \omega^\ell \|_2$. 
	
	The inequalities~\eqref{a priori-2} and \eqref{a priori-3} will play an important role in the nonlinear analysis. 
	
	In addition, since both $\omega^k$ and $\Delta_N \omega^k$ have a zero-average, at any time step $t^k$, as stated in Proposition~\ref{prop:average}, we see that their discrete inner product with the subtracted average constant term always vanishes: 
	\begin{equation} 
		\langle \overline{\u^\ell \DOT \nabla_N \omega^\ell} , \omega^k \rangle = 0 , \quad 
		\langle \overline{\u^\ell \DOT \nabla_N \omega^\ell} , \Delta_N \omega^k \rangle = 0  , 
		\quad 
		\mbox{since $\overline{\omega^k} = \overline{\Delta_N \omega^k} = 0$}  , 
		\quad \forall k, \ell \ge 0 . 
		\label{average-5} 
	\end{equation} 
	This identity will also be repeatedly used in the later analysis. 
	
	\subsection{The leading $\ell^\infty (0,T; L^2) \cap \ell^2 (0,T; H^1)$ estimate for the vorticity} 
	Taking a discrete inner product with~\eqref{scheme-BDF3-1} by $3 \omega^{n+1} - 2 \omega^n$ gives 
	\begin{align} 
		& 
		\Big\langle \frac{11}{6} \omega^{n+1} - 3 \omega^n + \frac32 \omega^{n-1} - \frac13 \omega^{n-2} ,
		3 \omega^{n+1} - 2 \omega^n \Big\rangle   
		+ \nu \dt   \big\langle \nabla_N  \omega^{n+1}  , 
		\nabla_N ( 3 \omega^{n+1} - 2 \omega^n ) \big\rangle 
		\nonumber 
		\\
		= & 
		- \frac32 \dt \left\langle \u^n \DOT \nabla_N \omega^n 
		+ \nabla_N \cdot ( \u^n \omega^n ) , 3 \omega^{n+1} - 2 \omega^n \right\rangle   \nonumber 
		\\
		&  
		+ \frac32 \dt \left\langle \u^{n-1} \DOT \nabla_N \omega^{n-1} 
		+ \nabla_N \cdot \left( \u^{n-1} \omega^{n-1} \right) , 
		3 \omega^{n+1} - 2 \omega^n \right\rangle   \nonumber 
		\\
		&
		- \frac12 \dt \left\langle \u^{n-2} \DOT \nabla_N \omega^{n-2} 
		+ \nabla_N \cdot \left( \u^{n-2} \omega^{n-2} \right) , 
		3 \omega^{n+1} - 2 \omega^n \right\rangle 
		+  \dt \left\langle  \f^{n+1} , 3 \omega^{n+1} - 2 \omega^n \right\rangle  .  
		\label{est-BDF3-L2-1}
	\end{align} 
	Notice that the identity~\eqref{average-5} has been applied. 
	
	The estimate for the temporal differentiation term follows the same idea as in the combined telescope formula~\eqref{BDF3-telescope-3}: 
	\begin{equation} 
		\begin{aligned} 
			& \langle \frac{11}{6} \omega^{n+1} - 3 \omega^n 
			+ \frac32 \omega^{n-1} - \frac13 \omega^{n-2} ,
			3 \omega^{n+1} - 2 \omega^n \rangle \\
			\ge &  \| \alpha_1 \omega^{n+1} \|_2^2 - \| \alpha_1 \omega^n \|_2^2
			+ \| \alpha_2 \omega^{n+1} + \alpha_3 \omega^n \|_2^2
			- \| \alpha_2 \omega^n + \alpha_3 \omega^{n-1} \|_2 ^2 
			\\
			&
			+ \| \alpha_4 \omega^{n+1} + \alpha_5 \omega^n + \alpha_6 \omega^{n-1} \|_2^2
			- \| \alpha_4 \omega^n + \alpha_5 \omega^{n-1} + \alpha_6 \omega^{n-2} \|_2^2 
			\\
			& 
			+ \frac{13}{12}  \| \omega^{n+1} - \omega^n \|_2^2 - \frac{7}{12} \| \omega^n - \omega^{n-1} \|_2^2 
			- \frac16 \| \omega^{n-1} - \omega^{n-2} \|_2^2    
			+ \frac{7}{12} \| \omega^{n+1} - 2 \omega^n + \omega^{n-1} \|_2^2 . 
		\end{aligned} 
		\label{est-BDF3-L2-2} 
	\end{equation} 
	The diffusion term could be analyzed in a straightforward way: 
	\begin{equation} 
		\begin{aligned}  
			& 
			\langle \nabla_N  \omega^{n+1}  , 
			\nabla_N ( 3 \omega^{n+1} - 2 \omega^n ) \rangle 
			=  \| \nabla_N \omega^{n+1} \|_2^2 
			+  2 \langle \nabla_N  \omega^{n+1}  , 
			\nabla_N ( \omega^{n+1} - \omega^n )  \rangle  
			\\
			= & 
			\| \nabla_N \omega^{n+1} \|_2^2 
			+  \|  \nabla_N  \omega^{n+1}  \|_2^2 - \| \nabla_N \omega^n \|_2^2 
			+ \| \nabla_N ( \omega^{n+1} - \omega^n )  \|_2^2 .   
		\end{aligned} 
		\label{est-BDF3-L2-3} 
	\end{equation} 
	A bound for the external force term can be derived as 
	\begin{equation} 
		\begin{aligned} 
			\left\langle  \f^{n+1} , 3 \omega^{n+1}  - 2 \omega^n \right\rangle  
			\le& \| \f^{n+1} \|_2  \cdot \| 3 \omega^{n+1}  - 2 \omega^n \|_2 
			\le C_2 \| \f^{n+1} \|_2  \cdot \| \nabla_N ( 3 \omega^{n+1} - 2 \omega^n ) \|_2    
			\\
			\le& 
			\frac{1}{32} \nu  \| \nabla_N ( 3 \omega^{n+1} - 2 \omega^n ) \|_2^2 
			+ 8 C_2^2  \nu^{-1} \| \f^{n+1} \|_2^2 
			\\
			\le  & 
			\frac{1}{16}  \nu  \| \nabla_N \omega^{n+1} \|_2^2 
			+ \frac14  \nu  \| \nabla_N ( \omega^{n+1} - \omega^n ) \|_2^2  
			+ 8 C_2^2 M^2 \nu^{-1}  , 
		\end{aligned} 
		\label{est-BDF3-L2-4}
	\end{equation}
	in which the discrete Poincar\'e inequality~\eqref{lem 2-1} was used in the second step, while the Cauchy inequality was applied in the last two steps.
	In terms of the nonlinear inner products, we first look at time step $t^n$.  Because of the following equality 
	\begin{equation} 
		\left\langle  \u^n \DOT \nabla_N \omega^n  
		+ \nabla_N \cdot \left( \u^n \omega^n \right) , \omega^n  \right\rangle       
		= 0  ,  \quad \mbox{(by~\eqref{lem 2-2})} , \label{est-BDF3-L2-5-1}
	\end{equation}
	we see that  
	\begin{equation} 
		\begin{aligned} 
			& 
			- \dt \left\langle \u^n \DOT \nabla_N \omega^n 
			+ \nabla_N \cdot \left( \u^n \omega^n \right) , 3 \omega^{n+1} - 2 \omega^n \right\rangle   
			\\
			=&
			- \dt \left\langle \u^n \DOT \nabla_N \omega^n 
			+ \nabla_N \cdot \left( \u^n \omega^n \right) , 3 \omega^{n+1} - 3 \omega^n \right\rangle 
			\\
			= & 
			-3 \dt \left\langle \u^n \DOT \nabla_N \omega^n 
			+ \nabla_N \cdot \left( \u^n \omega^n \right) ,  \omega^{n+1} -  \omega^n \right\rangle  .   
		\end{aligned} 
		\label{est-BDF3-L2-5-2} 
	\end{equation} 
	Meanwhile, an application of estimate~\eqref{a priori-2} (with $\ell=n$) indicates the following bound 
	\begin{equation} 
		\begin{aligned} 
			& 
			- \frac32 \dt \left\langle \u^n \DOT \nabla_N \omega^n 
			+ \nabla_N \cdot \left( \u^n \omega^n \right) , 3 \omega^{n+1} - 2 \omega^n \right\rangle   
			\\
			= & 
			-\frac92 \dt \left\langle \u^n \DOT \nabla_N \omega^n 
			+ \nabla_N \cdot \left( \u^n \omega^n \right) ,  \omega^{n+1} -  \omega^n \right\rangle  
			\\
			\le & 
			\frac92 \dt \| \u^n \DOT \nabla_N \omega^n 
			+ \nabla_N \cdot ( \u^n \omega^n ) \|_2 \cdot \| \omega^{n+1} -  \omega^n \|_2 
			\\
			\le & 
			\frac{9 \gamma_0 \dt}{2} \| \pomega^n \|_{H^\delta}  
			\cdot \| \nabla_N \omega^n \|_2 \cdot \| \omega^{n+1} -  \omega^n \|_2 \\
			\le &\frac{9 \gamma_0 \tilde{C}_1 \dt}{2} 
			\| \nabla_N \omega^n \|_2 \cdot \| \omega^{n+1} -  \omega^n \|_2 
			\\
			\le & 
			\frac{\nu}{4} \dt \|  \nabla_N \omega^n  \|_2^2
			+ \frac{81}{4} \gamma_0^2 \tilde{C}_1^2 \nu^{-1}  \dt    
			\| \omega^{n+1} - \omega^n \|_2^2 . 
		\end{aligned}   
		\label{est-BDF3-L2-5-3} 
	\end{equation} 
	
	The nonlinear terms at time steps $t^{n-1}$ and $t^{n-2}$ could be treated in a similar fashion. Some technical details are skipped for the sake of brevity.  
	\begin{equation} 
		\begin{aligned} 
			& 
			\dt \left\langle \u^{n-1} \DOT \nabla_N \omega^{n-1} 
			+ \nabla_N \cdot \left( \u^{n-1} \omega^{n-1} \right) , 3 \omega^{n+1} - 2 \omega^n \right\rangle   
			\\
			=&
			\dt \left\langle \u^{n-1} \DOT \nabla_N \omega^{n-1} 
			+ \nabla_N \cdot \left( \u^{n-1} \omega^{n-1} \right) ,  \omega^{n+1} - \omega^{n-1} 
			+ 2 (\omega^{n+1} -  \omega^n ) \right\rangle , 
		\end{aligned} 
		\label{est-BDF3-L2-6-1} 
	\end{equation} 
	\begin{equation} 
		\begin{aligned} 
			& 
			\frac32 \dt \left\langle \u^{n-1} \DOT \nabla_N \omega^{n-1} 
			+ \nabla_N \cdot \left( \u^{n-1} \omega^{n-1} \right) , 3 \omega^{n+1} - 2 \omega^n \right\rangle   
			\\
			= & 
			\frac32 \dt \left\langle \u^{n-1} \DOT \nabla_N \omega^{n-1} 
			+ \nabla_N \cdot \left( \u^{n-1} \omega^{n-1} \right) ,  \omega^{n+1} - \omega^{n-1} 
			+ 2 (\omega^{n+1} -  \omega^n )  \right\rangle  
			\\
			\le & 
			\frac32 \dt \| \u^{n-1} \DOT \nabla_N \omega^{n-1} 
			+ \nabla_N \cdot ( \u^{n-1} \omega^{n-1} ) \|_2 
			\cdot \| \omega^{n+1} - \omega^{n-1} 
			+ 2 (\omega^{n+1} -  \omega^n ) \|_2 
			\\
			\le & 
			\frac{3 \gamma_0 \dt}{2} \| \pomega^{n-1} \|_{H^\delta}  
			\cdot \| \nabla_N \omega^{n-1} \|_2 \cdot \| \omega^{n+1} - \omega^{n-1} 
			+ 2 (\omega^{n+1} -  \omega^n )  \|_2  
			\\
			\le & 
			\frac{3 \gamma_0 \tilde{C}_1 \dt}{2} 
			\| \nabla_N \omega^{n-1} \|_2 \cdot \| \omega^{n+1} - \omega^{n-1} 
			+ 2 (\omega^{n+1} -  \omega^n )  \|_2 
			\\
			\le & 
			\frac{\nu}{4} \dt \|  \nabla_N \omega^{n-1}  \|_2^2
			+ \frac94 \gamma_0^2 \tilde{C}_1^2 \nu^{-1}  \dt    
			\| \omega^{n+1} - \omega^{n-1} 
			+ 2 (\omega^{n+1} -  \omega^n )  \|_2^2 , 
		\end{aligned}   
		\label{est-BDF3-L2-6-2} 
	\end{equation} 
	\begin{equation} 
		\begin{aligned} 
			& 
			\dt \left\langle \u^{n-2} \DOT \nabla_N \omega^{n-2} 
			+ \nabla_N \cdot \left( \u^{n-2} \omega^{n-2} \right) , 3 \omega^{n+1} - 2 \omega^n \right\rangle   
			\\
			=&
			\dt \left\langle \u^{n-2} \DOT \nabla_N \omega^{n-2} 
			+ \nabla_N \cdot \left( \u^{n-2} \omega^{n-2} \right) ,  \omega^{n+1} - \omega^{n-2} 
			+ 2 (\omega^{n+1} -  \omega^n ) \right\rangle , 
		\end{aligned} 
		\label{est-BDF3-L2-7-1} 
	\end{equation} 
	\begin{equation} 
		\begin{aligned} 
			& 
			- \frac12 \dt \left\langle \u^{n-1} \DOT \nabla_N \omega^{n-1} 
			+ \nabla_N \cdot \left( \u^{n-1} \omega^{n-1} \right) , 3 \omega^{n+1} - 2 \omega^n \right\rangle   
			\\
			= & 
			- \frac12 \dt \left\langle \u^{n-2} \DOT \nabla_N \omega^{n-2} 
			+ \nabla_N \cdot \left( \u^{n-2} \omega^{n-2} \right) ,  \omega^{n+1} - \omega^{n-2} 
			+ 2 (\omega^{n+1} -  \omega^n )  \right\rangle  
			\\
			\le & 
			\frac12 \dt \| \u^{n-2} \DOT \nabla_N \omega^{n-2} 
			+ \nabla_N \cdot ( \u^{n-2} \omega^{n-2} ) \|_2 
			\cdot \| \omega^{n+1} - \omega^{n-2} 
			+ 2 (\omega^{n+1} -  \omega^n ) \|_2 
			\\
			\le & 
			\frac{\gamma_0 \dt}{2} \| \pomega^{n-2} \|_{H^\delta}  
			\cdot \| \nabla_N \omega^{n-2} \|_2 \cdot \| \omega^{n+1} - \omega^{n-2} 
			+ 2 (\omega^{n+1} -  \omega^n )  \|_2  
			\\
			\le & 
			\frac{\gamma_0 \tilde{C}_1 \dt}{2} 
			\| \nabla_N \omega^{n-2} \|_2 \cdot \| \omega^{n+1} - \omega^{n-2} 
			+ 2 (\omega^{n+1} -  \omega^n )  \|_2 
			\\
			\le & 
			\frac{3 \nu}{16} \dt \|  \nabla_N \omega^{n-2}  \|_2^2
			+ \frac13 \gamma_0^2 \tilde{C}_1^2 \nu^{-1}  \dt    
			\| \omega^{n+1} - \omega^{n-2} 
			+ 2 (\omega^{n+1} -  \omega^n )  \|_2^2 .  
		\end{aligned}   
		\label{est-BDF3-L2-7-2} 
	\end{equation} 
	
	Therefore, a substitution of~\eqref{est-BDF3-L2-2}-\eqref{est-BDF3-L2-4}, 
	\eqref{est-BDF3-L2-5-3}, \eqref{est-BDF3-L2-6-2} and \eqref{est-BDF3-L2-7-2} into \eqref{est-BDF3-L2-1} gives 
	\begin{equation} 
		\begin{aligned} 
			&
			\| \alpha_1 \omega^{n+1} \|_2^2 - \| \alpha_1 \omega^n \|_2^2
			+ \| \alpha_2 \omega^{n+1} + \alpha_3 \omega^n \|_2^2
			- \| \alpha_2 \omega^n + \alpha_3 \omega^{n-1} \|_2 ^2 
			\\
			&
			+ \| \alpha_4 \omega^{n+1} + \alpha_5 \omega^n + \alpha_6 \omega^{n-1} \|_2^2
			- \| \alpha_4 \omega^n + \alpha_5 \omega^{n-1} + \alpha_6 \omega^{n-2} \|_2^2 
			\\
			& 
			+ \frac{13}{12}  \| \omega^{n+1} - \omega^n \|_2^2 - \frac{7}{12} \| \omega^n - \omega^{n-1} \|_2^2 
			- \frac16 \| \omega^{n-1} - \omega^{n-2} \|_2^2    
			\\
			& 
			+ \frac{31 \nu \dt}{16} \| \nabla_N \omega^{n+1} \|_2^2 
			- \frac{5 \nu \dt}{4} \| \nabla_N \omega^n \|_2^2 
			+ \frac34 \nu \dt \| \nabla_N ( \omega^{n+1} - \omega^n )  \|_2^2 
			\\
			\le & 
			\frac{\nu \dt}{4} \| \nabla_N \omega^{n-1} \|_2^2 
			+ \frac{3 \nu \dt}{16} \| \nabla_N \omega^{n-2} \|_2^2 
			+ 8 C_2^2 M^2 \nu^{-1} \dt  
			+ \gamma_0^2 \tilde{C}_1^2 \nu^{-1} \dt  
			\Big( \frac{81}{4} \|  \omega^{n+1}  - \omega^n \|_2^2   
			\\
			& 
			+ \frac94 \|  \omega^{n+1}  - \omega^{n-1} + 2 ( \omega^{n+1} - \omega^n ) \|_2^2    
			+ \frac13 \|  \omega^{n+1}  - \omega^{n-2} + 2 ( \omega^{n+1} - \omega^n ) \|_2^2  \Big) 
			\\
			\le & 
			\gamma_0^2 \tilde{C}_1^2 \nu^{-1} \dt  
			\Big( \frac{209}{4} \|  \omega^{n+1}  - \omega^n \|_2^2  
			+ \frac{32}{3} \|  \omega^n  - \omega^{n-1} \|_2^2  
			+ \frac53 \| \omega^{n-1} - \omega^{n-2}  \|_2^2 \Big)  
			\\
			&  
			+ \frac{\nu \dt}{4} \| \nabla_N \omega^{n-1} \|_2^2 
			+ \frac{3 \nu \dt}{16} \| \nabla_N \omega^{n-2} \|_2^2 
			+ C_3 \dt ,  
		\end{aligned} 
		\label{est-BDF3-L2-8-1} 
	\end{equation} 
	in which $C_3 = 8 C_2^2 M^2 \nu^{-1} $ and the Cauchy inequality has been repeatedly applied: 
	\begin{equation} 
		\begin{aligned} 
			& 
			\|  \omega^{n+1}  - \omega^{n-1} + 2 ( \omega^{n+1} - \omega^n ) \|_2^2
			= \|  3 ( \omega^{n+1} - \omega^n ) + \omega^n - \omega^{n-1}  \|_2^2    
			\\
			\le & 
			12  \|  \omega^{n+1}  - \omega^n \|_2^2 + 4  \| \omega^n - \omega^{n-1} \|_2^2  , 
			\\
			& 
			\|  \omega^{n+1}  - \omega^{n-2} + 2 ( \omega^{n+1} - \omega^n ) \|_2^2
			= \|  3 ( \omega^{n+1} - \omega^n ) + \omega^n - \omega^{n-1} 
			+ \omega^{n-1} - \omega^{n-2} \|_2^2    
			\\
			\le & 
			15  \|  \omega^{n+1}  - \omega^n \|_2^2 + 5  \| \omega^n - \omega^{n-1} \|_2^2 
			+ 5  \| \omega^{n-1} - \omega^{n-2} \|_2^2 . 
		\end{aligned} 
		\label{est-BDF3-L2-8-2} 
	\end{equation} 
	Moreover, inequality~\eqref{est-BDF3-L2-8-1} could be rewritten as 
	\begin{equation} 
		\begin{aligned} 
			&
			\| \alpha_1 \omega^{n+1} \|_2^2 + \| \alpha_2 \omega^{n+1} + \alpha_3 \omega^n \|_2^2
			+ \| \alpha_4 \omega^{n+1} + \alpha_5 \omega^n + \alpha_6 \omega^{n-1} \|_2^2  
			+ \frac{31 \nu \dt}{16} \| \nabla_N \omega^{n+1} \|_2^2  
			\\
			& 
			+ \Big( \frac{13}{12}  - \frac{209}{4} \gamma_0^2 \tilde{C}_1^2 \nu^{-1} \dt \Big) 
			\| \omega^{n+1} - \omega^n \|_2^2 
			+ \frac34 \nu \dt \| \nabla_N ( \omega^{n+1} - \omega^n )  \|_2^2  
			\\
			\le & 
			\| \alpha_1 \omega^n \|_2^2
			+ \| \alpha_2 \omega^n + \alpha_3 \omega^{n-1} \|_2 ^2 
			+ \| \alpha_4 \omega^n + \alpha_5 \omega^{n-1} + \alpha_6 \omega^{n-2} \|_2^2 
			+ \frac{5 \nu \dt}{4} \| \nabla_N \omega^n \|_2^2 
			\\
			& 
			+ \frac{\nu \dt}{4} \| \nabla_N \omega^{n-1} \|_2^2 
			+ \frac{3 \nu \dt}{16} \| \nabla_N \omega^{n-2} \|_2^2 
			+ \Big( \frac{7}{12} + \frac{32 \gamma_0^2 \tilde{C}_1^2 \nu^{-1} \dt}{3}  \Big) 
			\|  \omega^n  - \omega^{n-1} \|_2^2
			\\ 
			&   
			+ \Big( \frac16 + \frac{5 \gamma_0^2 \tilde{C}_1^2 \nu^{-1} \dt}{3}  \Big)  
			\| \omega^{n-1} - \omega^{n-2}  \|_2^2  + C_3 \dt .   
		\end{aligned} 
		\label{est-BDF3-L2-8-3} 
	\end{equation} 
	Under a constraint for the time step 
	\begin{equation} 
		\frac{209}{4} \gamma_0^2 \tilde{C}_1^2 \dt  \le \frac{\nu}{6}  ,  \quad \mbox{i.e.}, \quad 
		\dt \le  \frac{2 \nu}{627 \gamma_0^2 \tilde{C}_1^2}  ,  \label{constraint-BDF3-dt-1}
	\end{equation}   
	we get  
	\begin{equation} 
		\begin{aligned} 
			&
			\| \alpha_1 \omega^{n+1} \|_2^2 + \| \alpha_2 \omega^{n+1} + \alpha_3 \omega^n \|_2^2
			+ \| \alpha_4 \omega^{n+1} + \alpha_5 \omega^n + \alpha_6 \omega^{n-1} \|_2^2  
			+ \frac{31 \nu \dt}{16} \| \nabla_N \omega^{n+1} \|_2^2  
			\\
			& 
			+ \frac{11}{12} \| \omega^{n+1} - \omega^n \|_2^2 
			+ \frac34 \nu \dt \| \nabla_N ( \omega^{n+1} - \omega^n )  \|_2^2  
			\\
			\le & 
			\| \alpha_1 \omega^n \|_2^2
			+ \| \alpha_2 \omega^n + \alpha_3 \omega^{n-1} \|_2 ^2 
			+ \| \alpha_4 \omega^n + \alpha_5 \omega^{n-1} + \alpha_6 \omega^{n-2} \|_2^2 
			+ \frac{5 \nu \dt}{4} \| \nabla_N \omega^n \|_2^2 
			\\ 
			& 
			+ \frac{\nu \dt}{4} \| \nabla_N \omega^{n-1} \|_2^2 
			+ \frac{3 \nu \dt}{16} \| \nabla_N \omega^{n-2} \|_2^2 
			+ \frac58 \|  \omega^n  - \omega^{n-1} \|_2^2  
			+ \frac{5}{24} \| \omega^{n-1} - \omega^{n-2}  \|_2^2  + C_3 \dt .     
		\end{aligned} 
		\label{est-BDF3-L2-8-4} 
	\end{equation} 
	An addition of $\nu \dt ( \frac12 \| \nabla_N \omega^n \|_2^2 +  \frac{7}{32} \| \nabla_N \omega^{n-1} \|_2^2 +  \frac{1}{64} \| \nabla_N \omega^{n-2} \|_2^2 ) + \frac14 \| \omega^n - \omega^{n-1} \|_2^2$ to both sides of this inequality yields 
	\begin{equation} 
		\begin{aligned} 
			&
			\| \alpha_1 \omega^{n+1} \|_2^2 + \| \alpha_2 \omega^{n+1} + \alpha_3 \omega^n \|_2^2
			+ \| \alpha_4 \omega^{n+1} + \alpha_5 \omega^n + \alpha_6 \omega^{n-1} \|_2^2  
			+ \frac{31 \nu \dt}{16} \| \nabla_N \omega^{n+1} \|_2^2  
			\\
			& 
			+  \frac{\nu \dt}{2} \| \nabla_N \omega^n \|_2^2 
			+  \frac{7 \nu \dt}{32} \| \nabla_N \omega^{n-1} \|_2^2 
			+  \frac{\nu \dt}{64} \| \nabla_N \omega^{n-2} \|_2^2  
			+ \frac{11}{12} \| \omega^{n+1} - \omega^n \|_2^2 
			\\
			& 
			+  \frac14 \|  \omega^n  - \omega^{n-1} \|_2^2  
			+ \frac34 \nu \dt \| \nabla_N ( \omega^{n+1} - \omega^n )  \|_2^2  
			\\
			\le & 
			\| \alpha_1 \omega^n \|_2^2
			+ \| \alpha_2 \omega^n + \alpha_3 \omega^{n-1} \|_2 ^2 
			+ \| \alpha_4 \omega^n + \alpha_5 \omega^{n-1} + \alpha_6 \omega^{n-2} \|_2^2 
			+ \frac{7 \nu \dt}{4} \| \nabla_N \omega^n \|_2^2  
			\\
			& 
			+  \frac{15 \nu \dt}{32} \| \nabla_N \omega^{n-1} \|_2^2 
			+  \frac{13 \nu \dt}{64} \| \nabla_N \omega^{n-2} \|_2^2   
			+ \frac78 \|  \omega^n  - \omega^{n-1} \|_2^2  
			+ \frac{5}{24} \| \omega^{n-1} - \omega^{n-2}  \|_2^2  + C_3 \dt .     
		\end{aligned} 
		\label{est-BDF3-L2-8-5} 
	\end{equation} 
	For simplicity of presentation, the following quantity is introduced 
	\begin{equation} 
		\begin{aligned} 
			F^n = & 
			\| \alpha_1 \omega^n \|_2^2
			+ \| \alpha_2 \omega^n + \alpha_3 \omega^{n-1} \|_2 ^2 
			+ \| \alpha_4 \omega^n + \alpha_5 \omega^{n-1} + \alpha_6 \omega^{n-2} \|_2^2 
			+ \frac{7 \nu \dt}{4} \| \nabla_N \omega^n \|_2^2 
			\\
			& 
			+  \frac{15 \nu \dt}{32} \| \nabla_N \omega^{n-1} \|_2^2 
			+  \frac{13 \nu \dt}{64} \| \nabla_N \omega^{n-2} \|_2^2   
			+ \frac78 \|  \omega^n  - \omega^{n-1} \|_2^2  
			+ \frac{5}{24} \| \omega^{n-1} - \omega^{n-2}  \|_2^2 , 
		\end{aligned} 
		\label{est-BDF3-L2-8-6} 
	\end{equation} 
	so that we obtain 
	\begin{equation} 
		\begin{aligned} 
			&
			F^{n+1}  
			+ \frac{3 \nu \dt}{16} \| \nabla_N \omega^{n+1} \|_2^2  
			+  \frac{\nu \dt}{32} \| \nabla_N \omega^n \|_2^2 
			+  \frac{\nu \dt}{64} \| \nabla_N \omega^{n-1} \|_2^2 
			\\
			& 
			+ \frac{1}{24} ( \| \omega^{n+1} - \omega^n \|_2^2 
			+ \|  \omega^n  - \omega^{n-1} \|_2^2  )  
			\le   F^n  + C_3 \dt .     
		\end{aligned} 
		\label{est-BDF3-L2-8-7} 
	\end{equation} 
	On the other hand, the following estimates are available: 
	\begin{align} 
		& 
		\| \omega^\ell \|_2^2 \le C_2^2 \| \nabla_N \omega^\ell \|_2^2 ,  \, \, \,  \ell = n+1 , n , n-1 , \label{est-BDF3-L2-8-8-1} 
		\\
		& 
		\| \alpha_1 \omega^{n+1} \|_2^2
		+ \| \alpha_2 \omega^{n+1} + \alpha_3 \omega^n \|_2 ^2 
		+ \| \alpha_4 \omega^{n+1} + \alpha_5 \omega^n + \alpha_6 \omega^{n-1} \|_2^2  
		\nonumber 
		\\
		\le & 
		\alpha_1^* \| \omega^{n+1} \|_2^2 
		+ \alpha_2^* \| \omega^n \|_2^2 + \alpha_3^* \| \omega^{n-1} \|_2^2 ,\nonumber 
		\\
		\le & 
		\beta_1 ( \frac18 \| \nabla_N \omega^{n+1} \|_2^2  
		+  \frac{1}{64} \| \nabla_N \omega^n \|_2^2 
		+  \frac{1}{128} \| \nabla_N \omega^{n-1} \|_2^2  \Big) , 
		\label{est-BDF3-L2-8-8-2}       
	\end{align} 
	where
	\begin{align*}
		&\alpha_1^* = \alpha_1^2 + 2 \alpha_2^2 + 3 \alpha_4^2 , \, \, 
		\alpha_2^* = 2 \alpha_3^2 + 3 \alpha_5^2 , \, \, 
		\alpha_3^* = 3 \alpha_6^2 \\
		&\beta_1 = \max \{8 ( \alpha_1^* )^{-1} , 64 (\alpha_2^* )^{-1} , 
		128 ( \alpha_3^* )^{-1}\} C_2^{-2} .
	\end{align*}
	Going back~\eqref{est-BDF3-L2-8-7}, we arrive at 
	\begin{equation} 
		F^{n+1}  + \beta_0  \nu \dt F^{n+1}  \le   F^n  + C_3 \dt ,  \quad 
		\beta_0 = \min \{\beta_1^{-1} , 24^{-1}\} ,   
		\label{est-BDF3-L2-8-9} 
	\end{equation} 
	provided that $\nu \le 1$, $\dt \le 1$. An application of recursive argument to the above inequality implies that 
	\begin{equation*}
		F^{n+1} \le (1+ \beta_0 \nu \dt )^{-(n+1)} F^0 
		+ \beta_0^{-1} \nu^{-1} C_3 
		\le F^0 +  \beta_0^{-1} \nu^{-1} C_3,
	\end{equation*}
	so that
	\begin{equation}
		\begin{aligned} 
			&
			\| \omega^{n+1} \|_2^2  \le \alpha_1^{-2} F^{n+1} 
			\le \alpha_1^{-2} ( F^0 +  \beta_0^{-1} \nu^{-1} C_3 ) , 
			\\
			& 
			\| \omega^{n+1} \|_2  \le \alpha_1^{-1}
			( F^0 +  \beta_0^{-1} \nu^{-1} C_3 )^{-\frac12} := C_4 , \quad   
			\forall n \ge 0 . 
		\end{aligned} 
		\label{est-BDF3-L2-8-10}
	\end{equation}
	It is obvious that $C_4$ is a time dependent value.	In addition, we also have an $\ell^2 (0, T; H^1)$ bound for the numerical solution: 
	\begin{equation}
		\frac{\nu}{8} \dt \sum_{k=3}^{N}  \|  \nabla_N  \omega^k  \|_2^2  
		\le   F^0  +  C_3  \, T^*,  \label{est-BDF3-L2-9}
	\end{equation}
where $T^* \le T$ denotes the final discrete time.
	However, it is observed that the a-priori estimate~\eqref{est-BDF3-L2-8-9} is not sufficient to bound the $H^\delta$ norm~\eqref{a priori-1} of the vorticity field. In turn, we have to perform a higher-order 
	estimate $\ell^\infty (0,T;  H^1) \cap \ell^2 (0, T; H^2)$ for the numerical solution.

	\subsection{The higher order $\ell^\infty (0, T;  H^1) \cap \ell^2 (0, T; H^2)$ estimate for the vorticity} 
	
	Taking the discrete inner product with~\eqref{scheme-BDF3-1} by $-\Delta_N ( 3 \omega^{n+1} - 2 \omega^n)$ gives 
	\begin{align} 
		& 
		- \Big\langle \frac{11}{6} \omega^{n+1} - 3 \omega^n + \frac32 \omega^{n-1} - \frac13 \omega^{n-2} ,
		\Delta_N ( 3 \omega^{n+1} - 2 \omega^n ) \Big\rangle   \nonumber\\
		&+ \nu \dt   \langle \Delta_N  \omega^{n+1}  , 
		\Delta_N ( 3 \omega^{n+1} - 2 \omega^n ) \rangle 
		\nonumber 
		\\
		= & 
		\frac32 \dt\left\langle \u^n \DOT \nabla_N \omega^n 
		+ \nabla_N \cdot ( \u^n \omega^n )  - \left(\u^{n-1} \DOT \nabla_N \omega^{n-1} 
		+ \nabla_N \cdot( \u^{n-1} \omega^{n-1})\right), 
		\Delta_N ( 3 \omega^{n+1} - 2 \omega^n ) \right\rangle 
		 \nonumber 
		\\
		&
		+ \frac12 \dt \left\langle \u^{n-2} \DOT \nabla_N \omega^{n-2} 
		+ \nabla_N \cdot \left( \u^{n-2} \omega^{n-2} \right) , 
		\Delta_N ( 3 \omega^{n+1} - 2 \omega^n ) \right\rangle\nonumber\\
		&-  \dt \left\langle  \f^{n+1} , \Delta_N ( 3 \omega^{n+1} - 2 \omega^n ) \right\rangle  .  
		\label{est-BDF3-H1-1}
	\end{align} 
	Again, the identity~\eqref{average-5} has been repeatedly applied. 
	
	Similar estimates could be derived for the inner products associated with the temporal differentiation term and the diffusion term, following the same techniques as in the combined telescope formula~\eqref{BDF3-telescope-3}, as well as the triangular equality:  
	\begin{equation} 
		\begin{aligned} 
			& - \left\langle \frac{11}{6} \omega^{n+1} - 3 \omega^n 
			+ \frac32 \omega^{n-1} - \frac13 \omega^{n-2} ,
			\Delta_N ( 3 \omega^{n+1} - 2 \omega^n ) \right\rangle \\
			= & \left\langle \nabla_N \left( \frac{11}{6} \omega^{n+1} - 3 \omega^n 
			+ \frac32 \omega^{n-1} - \frac13 \omega^{n-2} \right) ,
			\nabla_N ( 3 \omega^{n+1} - 2 \omega^n ) \right\rangle \\
			\ge &  \| \alpha_1 \nabla_N \omega^{n+1} \|_2^2 - \| \alpha_1 \nabla_N \omega^n \|_2^2
			+ \| \nabla_N ( \alpha_2 \omega^{n+1} + \alpha_3 \omega^n ) \|_2^2
			- \| \nabla_N ( \alpha_2 \omega^n + \alpha_3 \omega^{n-1} ) \|_2 ^2 
			\\
			&
			+ \| \nabla_N ( \alpha_4 \omega^{n+1} + \alpha_5 \omega^n 
			+ \alpha_6 \omega^{n-1} ) \|_2^2
			- \| \nabla_N ( \alpha_4 \omega^n + \alpha_5 \omega^{n-1} 
			+ \alpha_6 \omega^{n-2} ) \|_2^2 
			\\
			& 
			+ \frac{13}{12}  \| \nabla_N ( \omega^{n+1} - \omega^n ) \|_2^2 
			- \frac{7}{12} \| \nabla_N ( \omega^n - \omega^{n-1} ) \|_2^2 
			- \frac16 \| \nabla_N ( \omega^{n-1} - \omega^{n-2} ) \|_2^2  ,  
		\end{aligned} 
		\label{est-BDF3-H1-2} 
	\end{equation} 
	\begin{equation} 
		\begin{aligned}  
			& 
			\langle \Delta_N  \omega^{n+1}  , 
			\Delta_N ( 3 \omega^{n+1} - 2 \omega^n ) \rangle \\
			=&  \| \Delta_N \omega^{n+1} \|_2^2 
			+  2 \langle \Delta_N  \omega^{n+1}  , 
			\Delta_N ( \omega^{n+1} - \omega^n )  \rangle  
			\\
			= & 
			\| \Delta_N \omega^{n+1} \|_2^2 
			+  \|  \Delta_N  \omega^{n+1}  \|_2^2 - \| \Delta_N \omega^n \|_2^2 
			+ \| \Delta_N ( \omega^{n+1} - \omega^n )  \|_2^2 .   
		\end{aligned} 
		\label{est-BDF3-H1-3} 
	\end{equation} 
	Again, the Cauchy inequality is applied to obtain a bound for the inner product term associated with the external force: 
	\begin{equation} 
		\begin{aligned} 
			- \left\langle  \f^{n+1} , \Delta_N ( 3 \omega^{n+1}  - 2 \omega^n ) \right\rangle  
			\le& \| \f^{n+1} \|_2  \cdot \| \Delta_N ( 3 \omega^{n+1}  - 2 \omega^n ) \|_2   
			\\
			\le& 
			\frac{1}{32} \nu  \| \Delta_N ( 3 \omega^{n+1} - 2 \omega^n ) \|_2^2 
			+ 8 M^2 \nu^{-1} . 
		\end{aligned} 
		\label{est-BDF3-H1-4}
	\end{equation}
	
	Regarding the nonlinear term at time step $t^n$, an application of inequality~\eqref{a priori-2} (with $\ell =n$) gives 
	\begin{equation} 
		\begin{aligned} 
			& 
			\frac32 \left\langle \u^n \DOT \nabla_N \omega^n 
			+ \nabla_N \cdot ( \u^n \omega^n ) , 
			\Delta_N ( 3 \omega^{n+1} - 2 \omega^n ) \right\rangle  
			\\
			\le & 
			\frac32 \| \u^n \DOT \nabla_N \omega^n 
			+ \nabla_N \cdot ( \u^n \omega^n ) \|_2 \cdot  
			\| \Delta_N ( 3 \omega^{n+1} - 2 \omega^n ) \|_2  
			\\
			\le & 
			\frac32 \gamma_0 \| \pomega^n \|_{H^\delta} 
			\cdot \| \nabla_N \omega^n \|_2 
			\cdot \| \Delta_N ( 3 \omega^{n+1} - 2 \omega^n ) \|_2 
			\\
			\le & 
			18 \gamma_0^2 \nu^{-1} \| \pomega^n \|_{H^\delta}^2  
			\cdot \| \nabla_N \omega^n \|_2^2  
			+ \frac{\nu}{32} \| \Delta_N ( 3 \omega^{n+1} - 2 \omega^n ) \|_2^2 .  
		\end{aligned} 
		\label{est-BDF3-H1-5-1}
	\end{equation}
	Meanwhile, the Sobolev interpolation estimate leads to 
	\begin{align} 
		\| \nabla_N \omega^n \|_2  
		\le  & \| \omega^n \|_2^\frac12 \cdot \| \Delta_N \omega^n \|_2^\frac12 
		\le  C_4^\frac12 \| \Delta_N \omega^n \|_2^\frac12  ,  \label{est-BDF3-H1-5-2}
		\\
		\| \pomega^n \|_{H^\delta} 
		\le & Q_0  \| \omega^n \|_2^{1-\delta} \cdot \| \nabla_N \omega^n \|_2^\delta  
		\quad \mbox{(by~\eqref{a priori-3} with $\ell =n$)}  \nonumber 
		\\
		\le & 
		Q_0  \| \omega^n \|_2^{1-\frac{\delta}{2}} 
		\cdot \| \Delta_N \omega^n \|_2^{\frac{\delta}{2}} \nonumber \\
		\le & Q_0  C_4^{1-\frac{\delta}{2}} 
		\| \Delta_N \omega^n \|_2^{\frac{\delta}{2}} , \label{est-BDF3-H1-5-3} 
	\end{align} 
	in which the global-in-time $\ell^2$ bound~\eqref{est-BDF3-L2-8-9} has been repeatedly applied. In turn, the following bound becomes available: 
	\begin{equation} 
		\begin{aligned} 
			18 \gamma_0^2 \nu^{-1} \| \pomega^n \|_{H^\delta}^2  
			\cdot \| \nabla_N \omega^n \|_2^2   
			\le & 18 \gamma_0^2 Q_0^2 C_4^{3 - \delta} \nu^{-1} 
			\| \Delta_N \omega^n \|_2^{1 + \delta}   
			\le 
			C_{5, \nu} + \frac{\nu}{6} \| \Delta_N \omega^n \|_2^2 , 
		\end{aligned} 
		\label{est-BDF3-H1-5-4}
	\end{equation} 
	where the Young's inequality is applied in the last step, due to the fact that $1+\delta < 2$. Notice that $C_{5, \nu}$ is a uniform-in-time constant, dependent on $\gamma_0$, $Q_0$ and $C_4$, and its dependence on $\nu^{-1}$ is in a polynomial form. Subsequently, its combination with~\eqref{est-BDF3-H1-5-1} yields 
	\begin{equation} 
		\begin{aligned} 
			& 
			\frac32 \left\langle \u^n \DOT \nabla_N \omega^n 
			+ \nabla_N \cdot ( \u^n \omega^n ) , 
			\Delta_N ( 3 \omega^{n+1} - 2 \omega^n ) \right\rangle  
			\\
			\le  & 
			C_{5, \nu} + \frac{\nu}{6} \| \Delta_N \omega^n \|_2^2  
			+ \frac{\nu}{32} \| \Delta_N ( 3 \omega^{n+1} - 2 \omega^n ) \|_2^2 .  
		\end{aligned} 
		\label{est-BDF3-H1-5-5}
	\end{equation}
	
	The nonlinear inner products at time steps $t^{n-1}$ and $t^{n-2}$ could be analyzed using similar techniques. 
	\begin{align} 
		& 
		\| \nabla_N \omega^{n-1} \|_2  
		\le   \| \omega^{n-1} \|_2^\frac12 \cdot \| \Delta_N \omega^{n-1} \|_2^\frac12 
		\le  C_4^\frac12 \| \Delta_N \omega^{n-1} \|_2^\frac12  ,  \label{est-BDF3-H1-6-1}
		\\
		& 
		\| \pomega^{n-1} \|_{H^\delta} 
		\le  Q_0  \| \omega^{n-1} \|_2^{1-\delta} \cdot \| \nabla_N \omega^{n-1} \|_2^\delta  
		\quad \mbox{(by~\eqref{a priori-3} with $\ell =n-1$)}   \nonumber 
		\\
		& \qquad \qquad  \,\,
		\le   Q_0  \| \omega^{n-1} \|_2^{1-\frac{\delta}{2}} 
		\cdot \| \Delta_N \omega^{n-1} \|_2^{\frac{\delta}{2}} 
		\le Q_0  C_4^{1-\frac{\delta}{2}} 
		\| \Delta_N \omega^{n-1} \|_2^{\frac{\delta}{2}} , \label{est-BDF3-H1-6-2} 
		\\
		& 
		18 \gamma_0^2 \nu^{-1} \| \pomega^{n-1} \|_{H^\delta}^2  
		\cdot \| \nabla_N \omega^{n-1} \|_2^2   
		\le  18 \gamma_0^2 Q_0^2 C_4^{3 - \delta} \nu^{-1} 
		\| \Delta_N \omega^{n-1} \|_2^{1 + \delta}   
		\le 
		C_{5, \nu} + \frac{\nu}{6} \| \Delta_N \omega^{n-1} \|_2^2 , 
		\label{est-BDF3-H1-6-3}
	\end{align} 
	\begin{equation} 
		\begin{aligned} 
			& 
			- \frac32 \left\langle \u^{n-1} \DOT \nabla_N \omega^{n-1} 
			+ \nabla_N \cdot ( \u^{n-1} \omega^{n-1} ) , 
			\Delta_N ( 3 \omega^{n+1} - 2 \omega^n ) \right\rangle  
			\\
			\le & 
			\frac32 \| \u^{n-1} \DOT \nabla_N \omega^{n-1} 
			+ \nabla_N \cdot ( \u^{n-1} \omega^{n-1} ) \|_2 \cdot  
			\| \Delta_N ( 3 \omega^{n+1} - 2 \omega^n ) \|_2  
			\\
			\le & 
			\frac32 \gamma_0 \| \pomega^{n-1} \|_{H^\delta} 
			\cdot \| \nabla_N \omega^{n-1} \|_2 
			\cdot \| \Delta_N ( 3 \omega^{n+1} - 2 \omega^n ) \|_2 
			\\
			\le & 
			18 \gamma_0^2 \nu^{-1} \| \pomega^{n-1} \|_{H^\delta}^2  
			\cdot \| \nabla_N \omega^{n-1} \|_2^2  
			+ \frac{\nu}{32} \| \Delta_N ( 3 \omega^{n+1} - 2 \omega^n ) \|_2^2  
			\\
			\le & 
			C_{5, \nu} + \frac{\nu}{6} \| \Delta_N \omega^{n-1} \|_2^2  
			+ \frac{\nu}{32} \| \Delta_N ( 3 \omega^{n+1} - 2 \omega^n ) \|_2^2 , 
		\end{aligned} 
		\label{est-BDF3-H1-6-4}
	\end{equation}
	
	\begin{align} 
		& 
		\| \nabla_N \omega^{n-2} \|_2  
		\le   \| \omega^{n-2} \|_2^\frac12 \cdot \| \Delta_N \omega^{n-2} \|_2^\frac12 
		\le  C_4^\frac12 \| \Delta_N \omega^{n-2} \|_2^\frac12  ,  \label{est-BDF3-H1-7-1}
		\\
		& 
		\| \pomega^{n-2} \|_{H^\delta} 
		\le  Q_0  \| \omega^{n-2} \|_2^{1-\delta} \cdot \| \nabla_N \omega^{n-2} \|_2^\delta  
		\le Q_0  C_4^{1-\frac{\delta}{2}} 
		\| \Delta_N \omega^{n-2} \|_2^{\frac{\delta}{2}} , \label{est-BDF3-H1-7-2} 
		\\
		& 
		2 \gamma_0^2 \nu^{-1} \| \pomega^{n-2} \|_{H^\delta}^2  
		\cdot \| \nabla_N \omega^{n-2} \|_2^2   
		\le  2 \gamma_0^2 Q_0^2 C_4^{3 - \delta} \nu^{-1} 
		\| \Delta_N \omega^{n-2} \|_2^{1 + \delta}   
		\le 
		\frac19 C_{5, \nu} + \frac{\nu}{6} \| \Delta_N \omega^{n-2} \|_2^2 , 
		\label{est-BDF3-H1-7-3}
	\end{align} 
	\begin{equation} 
		\begin{aligned} 
			& 
			\frac12 \left\langle \u^{n-2} \DOT \nabla_N \omega^{n-2} 
			+ \nabla_N \cdot ( \u^{n-2} \omega^{n-2} ) , 
			\Delta_N ( 3 \omega^{n+1} - 2 \omega^n ) \right\rangle  
			\\
			\le & 
			\frac12 \| \u^{n-2} \DOT \nabla_N \omega^{n-2} 
			+ \nabla_N \cdot ( \u^{n-2} \omega^{n-2} ) \|_2 \cdot  
			\| \Delta_N ( 3 \omega^{n+1} - 2 \omega^n ) \|_2  
			\\
			\le & 
			\frac12 \gamma_0 \| \pomega^{n-2} \|_{H^\delta} 
			\cdot \| \nabla_N \omega^{n-2} \|_2 
			\cdot \| \Delta_N ( 3 \omega^{n+1} - 2 \omega^n ) \|_2 
			\\
			\le & 
			2 \gamma_0^2 \nu^{-1} \| \pomega^{n-2} \|_{H^\delta}^2  
			\cdot \| \nabla_N \omega^{n-2} \|_2^2  
			+ \frac{\nu}{32} \| \Delta_N ( 3 \omega^{n+1} - 2 \omega^n ) \|_2^2  
			\\
			\le & 
			\frac19 C_{5, \nu} + \frac{\nu}{6} \| \Delta_N \omega^{n-2} \|_2^2  
			+ \frac{\nu}{32} \| \Delta_N ( 3 \omega^{n+1} - 2 \omega^n ) \|_2^2 . 
		\end{aligned} 
		\label{est-BDF3-H1-7-4}
	\end{equation}
	
	Finally, a substitution of~\eqref{est-BDF3-H1-2}-\eqref{est-BDF3-H1-4}, 
	\eqref{est-BDF3-H1-5-5}, \eqref{est-BDF3-H1-6-4} and \eqref{est-BDF3-H1-7-4} into \eqref{est-BDF3-H1-1} results in  
	\begin{equation} 
		\begin{aligned} 
			&
			\| \alpha_1 \nabla_N \omega^{n+1} \|_2^2 - \| \alpha_1 \nabla_N \omega^n \|_2^2
			+ \| \nabla_N ( \alpha_2 \omega^{n+1} + \alpha_3 \omega^n ) \|_2^2
			- \| \nabla_N ( \alpha_2 \omega^n + \alpha_3 \omega^{n-1} ) \|_2 ^2 
			\\
			&
			+ \| \nabla_N ( \alpha_4 \omega^{n+1} + \alpha_5 \omega^n + \alpha_6 \omega^{n-1} ) \|_2^2
			- \| \nabla_N ( \alpha_4 \omega^n + \alpha_5 \omega^{n-1} + \alpha_6 \omega^{n-2} ) \|_2^2 
			\\
			& 
			+ \frac{13}{12}  \| \nabla_N ( \omega^{n+1} - \omega^n ) \|_2^2 
			- \frac{7}{12} \| \nabla_N ( \omega^n - \omega^{n-1} ) \|_2^2 
			- \frac16 \| \nabla_N ( \omega^{n-1} - \omega^{n-2} ) \|_2^2    
			\\
			& 
			+ 2 \nu \dt \| \Delta_N \omega^{n+1} \|_2^2 
			- \frac{7 \nu \dt}{6} \| \Delta_N \omega^n \|_2^2 
			+ \nu \dt \| \Delta_N ( \omega^{n+1} - \omega^n )  \|_2^2 
			\\
			\le & 
			\frac{\nu \dt}{6} ( \| \Delta_N \omega^{n-1} \|_2^2 
			+  \| \Delta_N \omega^{n-2} \|_2^2 ) 
			+ \frac{\nu \dt}{8} \| \Delta_N ( 3 \omega^{n+1} - 2 \omega^n ) \|_2^2
			+ ( 8 M^2 \nu^{-1} + \frac{19}{9} C_{5, \nu} ) \dt  
			\\
			\le & 
			\frac{\nu \dt}{6} ( \| \Delta_N \omega^{n-1} \|_2^2 
			+  \| \Delta_N \omega^{n-2} \|_2^2 ) 
			+ \frac{\nu \dt}{4} \| \Delta_N \omega^{n+1} \|_2^2 
			+ \nu \dt \| \Delta_N ( \omega^{n+1} - \omega^n ) \|_2^2 + C_6 \dt,
		\end{aligned} 
		\label{est-BDF3-H1-8-1} 
	\end{equation} 
	in which $C_6 = 8 M^2 \nu^{-1} + \frac{19}{9} C_{5, \nu} $ and the Cauchy inequality has been applied in the last step. Subsequently, inequality~\eqref{est-BDF3-H1-8-1} could be rewritten as 
	\begin{equation} 
		\begin{aligned} 
			&
			\| \alpha_1 \nabla_N \omega^{n+1} \|_2^2 
			+ \| \nabla_N ( \alpha_2 \omega^{n+1} + \alpha_3 \omega^n ) \|_2^2
			+ \| \nabla_N ( \alpha_4 \omega^{n+1} + \alpha_5 \omega^n + \alpha_6 \omega^{n-1} ) \|_2^2  
			\\
			& 
			+ \frac{7 \nu \dt}{4} \| \Delta_N \omega^{n+1} \|_2^2  
			+ \frac{13}{12}  \| \nabla_N ( \omega^{n+1} - \omega^n ) \|_2^2 
			\\
			\le & 
			\| \alpha_1 \nabla_N  \omega^n \|_2^2
			+ \| \nabla_N ( \alpha_2 \omega^n + \alpha_3 \omega^{n-1} ) \|_2 ^2 
			+ \| \nabla_N ( \alpha_4 \omega^n + \alpha_5 \omega^{n-1} + \alpha_6 \omega^{n-2} ) \|_2^2
			\\
			& 
			+  \frac{7 \nu \dt}{6} \| \Delta_N \omega^n \|_2^2 
			+ \frac{\nu \dt}{6} ( \| \Delta_N \omega^{n-1} \|_2^2 
			+  \| \Delta_N \omega^{n-2} \|_2^2 ) + \frac{7}{12} \| \nabla_N ( \omega^n - \omega^{n-1} ) \|_2^2 
			\\
			& 
			+ \frac16 \| \nabla_N ( \omega^{n-1} - \omega^{n-2} ) \|_2^2 
			+ C_6 \dt .         
		\end{aligned} 
		\label{est-BDF3-H1-8-2} 
	\end{equation} 
	Meanwhile, an addition of $\nu \dt ( \frac38 \| \Delta_N \omega^n \|_2^2 +  \frac{3}{16} \| \Delta_N \omega^{n-1} \|_2^2 +  \frac{1}{96} \| \Delta_N \omega^{n-2} \|_2^2 ) + \frac14 \| \nabla_N ( \omega^n - \omega^{n-1} ) \|_2^2 $ to both sides of this inequality gives  
	\begin{equation} 
		\begin{aligned} 
			&
			\| \alpha_1 \nabla_N \omega^{n+1} \|_2^2 
			+ \| \nabla_N ( \alpha_2 \omega^{n+1} + \alpha_3 \omega^n ) \|_2^2
			+ \| \nabla_N ( \alpha_4 \omega^{n+1} + \alpha_5 \omega^n + \alpha_6 \omega^{n-1} ) \|_2^2  
			\\
			& 
			+ \frac{7 \nu \dt}{4} \| \Delta_N \omega^{n+1} \|_2^2  
			+ \frac{3 \nu \dt}{8} \| \Delta_N \omega^n \|_2^2  
			+ \frac{3 \nu \dt}{16} \| \Delta_N \omega^{n-1} \|_2^2  
			+ \frac{\nu \dt}{96} \| \Delta_N \omega^{n-2} \|_2^2 
			\\
			&  
			+ \frac{13}{12}  \| \nabla_N ( \omega^{n+1} - \omega^n ) \|_2^2 
			+ \frac14  \| \nabla_N ( \omega^n - \omega^{n-1} ) \|_2^2 
			\\
			\le & 
			\| \alpha_1 \nabla_N  \omega^n \|_2^2
			+ \| \nabla_N ( \alpha_2 \omega^n + \alpha_3 \omega^{n-1} ) \|_2 ^2 
			+ \| \nabla_N ( \alpha_4 \omega^n + \alpha_5 \omega^{n-1} + \alpha_6 \omega^{n-2} ) \|_2^2
			\\
			& 
			+  \frac{37 \nu \dt}{24} \| \Delta_N \omega^n \|_2^2 
			+ \frac{17 \nu \dt}{48} \| \Delta_N \omega^{n-1} \|_2^2 
			+  \frac{17 \nu \dt}{96} \| \Delta_N \omega^{n-2} \|_2^2 
			\\
			& 
			+ \frac56 \| \nabla_N ( \omega^n - \omega^{n-1} ) \|_2^2 
			+ \frac16 \| \nabla_N ( \omega^{n-1} - \omega^{n-2} ) \|_2^2 
			+ C_6 \dt .         
		\end{aligned} 
		\label{est-BDF3-H1-8-3} 
	\end{equation} 
	Similar to the leading order $\ell^2$ estimate, we introduce the following quantity  
	\begin{equation} 
		\begin{aligned} 
			G_1^n = & 
			\| \alpha_1 \nabla_N  \omega^n \|_2^2
			+ \| \nabla_N ( \alpha_2 \omega^n + \alpha_3 \omega^{n-1} ) \|_2 ^2 
			+ \| \nabla_N ( \alpha_4 \omega^n + \alpha_5 \omega^{n-1} + \alpha_6 \omega^{n-2} ) \|_2^2
			\\
			& 
			+  \frac{37 \nu \dt}{24} \| \Delta_N \omega^n \|_2^2 
			+ \frac{17 \nu \dt}{48} \| \Delta_N \omega^{n-1} \|_2^2 
			+  \frac{17 \nu \dt}{96} \| \Delta_N \omega^{n-2} \|_2^2 
			\\
			& 
			+ \frac56 \| \nabla_N ( \omega^n - \omega^{n-1} ) \|_2^2 
			+ \frac16 \| \nabla_N ( \omega^{n-1} - \omega^{n-2} ) \|_2^2 , 
		\end{aligned} 
		\label{est-BDF3-H1-8-4} 
	\end{equation} 
	so that~\eqref{est-BDF3-H1-8-3} could be simplified as 
	\begin{equation} 
		\begin{aligned} 
			&
			G_1^{n+1}  
			+ \frac{5 \nu \dt}{24} \| \Delta_N \omega^{n+1} \|_2^2  
			+  \frac{\nu \dt}{48} \| \Delta_N \omega^n \|_2^2 
			+  \frac{\nu \dt}{96} \| \Delta_N \omega^{n-1} \|_2^2 
			\\
			& 
			+ \frac14 \| \nabla_N ( \omega^{n+1} - \omega^n ) \|_2^2 
			+ \frac{1}{12} \|  \nabla_N ( \omega^n  - \omega^{n-1} ) \|_2^2   
			\le   G_1^n  + C_6 \dt .     
		\end{aligned} 
		\label{est-BDF3-H1-8-5} 
	\end{equation} 
	Meanwhile, to obtain a useful uniform-in-time bound, the following estimates are needed: 
	\begin{align} 
		& 
		\| \nabla_N \omega^\ell \|_2^2 \le C_2^2 \| \Delta_N \omega^\ell \|_2^2 , \qquad \ell = n+1 , n , n-1, \label{est-BDF3-H1-8-6-1} 
		\\
		& 
		\| \alpha_1 \nabla_N \omega^{n+1} \|_2^2
		+ \left\| \nabla_N ( \alpha_2 \omega^{n+1} + \alpha_3 \omega^n ) \right\|_2 ^2 
		+ \left\| \nabla_N ( \alpha_4 \omega^{n+1} + \alpha_5 \omega^n 
		+ \alpha_6 \omega^{n-1} ) \right\|_2^2  
		\nonumber 
		\\
		\le & 
		\alpha_1^* \| \nabla_N \omega^{n+1} \|_2^2 
		+ \alpha_2^* \| \nabla_N \omega^n \|_2^2 + \alpha_3^* \| \nabla_N \omega^{n-1} \|_2^2 ,   
		\nonumber 
		\\
		\le & 
		\beta_1 ( \frac18 \| \Delta_N \omega^{n+1} \|_2^2  
		+  \frac{1}{64} \| \Delta_N \omega^n \|_2^2 
		+  \frac{1}{128} \| \Delta_N \omega^{n-1} \|_2^2  \Big) , 
		\label{est-BDF3-H1-8-6-2}   
	\end{align} 
	where 
	\begin{align*}
		&\alpha_1^* = \alpha_1^2 + 2 \alpha_2^2 + 3 \alpha_4^2 , \, \, 
		\alpha_2^* = 2 \alpha_3^2 + 3 \alpha_5^2 , \, \, 
		\alpha_3^* = 3 \alpha_6^2,\\
		&\beta_1 = \max\left\{8 ( \alpha_1^* )^{-1} , 64 (\alpha_2^* )^{-1} , 
		128 ( \alpha_3^* )^{-1} \right\} C_2^{-2}.
	\end{align*}
	 Then, we arrive at 
	\begin{equation} 
		G_1^{n+1}  + \beta_0  \nu \dt G_1^{n+1}  \le   G_1^n  + C_6 \dt ,  \quad 
		\beta_0 = \min \{ \beta_1^{-1} , 24^{-1} \},   
		\label{est-BDF3-H1-8-7} 
	\end{equation} 
	provided that $\nu \le 1$, $\dt \le 1$. Again, a recursive analysis to the above inequality reveals that  
	\begin{equation*}
		G_1^{n+1} \le (1+ \beta_0 \nu \dt )^{-(n+1)} G_1^0  
		+ \beta_0^{-1} \nu^{-1} C_6 
		\le G_1^0 +  \beta_0^{-1} \nu^{-1} C_6 ,
	\end{equation*}
	so that
	\begin{equation}
		\begin{aligned} 
			&
			\| \nabla_N \omega^{n+1} \|_2^2  \le \alpha_1^{-2} G_1^{n+1} 
			\le \alpha_1^{-2} ( G_1^0 +  \beta_0^{-1} \nu^{-1} C_6 ) , 
			\\
			& 
			\| \nabla_N \omega^{n+1} \|_2  \le \alpha_1^{-1}
			( G_1^0 +  \beta_0^{-1} \nu^{-1} C_6 )^{-\frac12} := C_7 , \quad   
			\forall n \ge 0 . 
		\end{aligned} 
		\label{est-BDF3-H1-8-8}
	\end{equation}
	Of course, $C_7$ is a time dependent value. 
	
	In addition, an $\ell^2 (0, T; H^2)$ bound is also available for the numerical solution: 
	\begin{equation}
		\frac{\nu}{8} \dt \sum_{k=3}^{N}  \|  \Delta_N  \omega^k  \|_2^2  
		\le   G_1^0  +  C_6  \, T^*.  \label{est-BDF3-H1-9}
	\end{equation}
	
	\subsection{Recovery of the a-priori $H^\delta$ assumption~\eqref{a priori-1}} 
	
	With the $\ell^\infty (0,T; \ell^2)$ and $L^\infty (0,T; H^1)$ estimate of the numerical solution for the vorticity variable, given by~\eqref{est-BDF3-L2-8-10} and \eqref{est-BDF3-H1-8-8}, respectively, we are able to recover the $H^\delta$ assumption~\eqref{a priori-1}: 
	\begin{align} 
		\|  \pomega^{n+1} \|_{H^\delta}  
		\le  & Q_0 \|  \pomega^{n+1}  \|^{1-\delta}    
		\cdot \|  \nabla \pomega^{n+1} \|^{\delta}  
		\le Q_0 C_4^{1-\delta} C_7^\delta ,  \quad  
		\mbox{(by~\eqref{a priori-3})}  . 
		\label{a priori-4}
	\end{align}
	For simplicity, by taking $\delta=\frac18$, we see that~\eqref{a priori-1} is also valid at time step $t^{n+1}$ if we set 
	\begin{equation} 
		\tilde{C}_1 = Q_0 C_4^\frac78 C_7^\frac18 .  \label{a priori-5}
	\end{equation}
	Meanwhile, it is noticed that both $C_4$ and $C_7$ are independent of $\tilde{C}_1$ in the theoretical analysis. Instead, the constant $\tilde{C}_1$ is only used in the time step constraint~\eqref{constraint-BDF3-dt-1}. As a result, an induction argument could be effectively applied so that the a-priori $H^\delta$ assumption \eqref{a priori-1} is valid at any time step under 
	such a global time step constraint.  

	In other words, under \eqref{constraint-BDF3-dt-1}, a global-in-time constant constraint for the time step, the proposed BDF3 scheme \eqref{scheme-BDF3-1}-\eqref{scheme-BDF3-3} is unconditionally stable (in terms of spatial grid size and final time), in both the discrete $L^2$ and $H^1$ norms. The quantities $R^{(0)}$ and $R^{(1)}$ could be taken as $R^{(0)} = \alpha_1^{-1} ( \beta_0^{-1} \nu^{-1} C_3)^\frac12$ and $R^{(1)} = \alpha_1^{-1} ( \beta_0^{-1} \nu^{-1} C_6)^\frac12$, as given by \eqref{est-BDF3-L2-8-10} and \eqref{est-BDF3-H1-8-8}, respectively. This validates the estimates~\eqref{BDF3-est-L2-2}  and \eqref{BDF3-est-H1-2}, and an asymptotic decay for the $L^2$ and $H^1$ norm for the vorticity (equivalent to $H^1$ and $H^2$ norms for the velocity), has also been theoretically established. The proof of Theorem~\ref{thm:stability} is finished. 
\begin{rem}
	Notice that the constants $C_4$ and $C_7$ are independent of $\tilde{C}_1$. In more details, $C_2$ is a discrete Poincar\'e inequality constant, only dependent on $\Omega$, and $C_3 = 8 C_2^2 M^2 \nu^{-1}$, as given by \eqref{est-BDF3-L2-8-1}, and $C_4 = \alpha_1^{-1} (F^0 + \beta_0^{-1} \nu^{-1} C_3 )^{-\frac12}$, as given by \eqref{est-BDF3-L2-8-10}. Therefore, $C_4$ depends only on $\Omega$, initial data, external force term $\f$, as well as the viscosity parameter $\nu$. Moreover, $C_{5, \nu}$ (as given by \eqref{est-BDF3-H1-5-4}) is a constant created in the application of Young's inequality, so that it is only dependent on $\Omega$, $C_4$ and $\nu$, therefore, only dependent on $\Omega$, initial data $\f$ and $\nu$. Subsequently, constant $C_6= 8M^2 \nu^{-1} + \frac{19}{9} C_{5, \nu}$, as introduced in \eqref{est-BDF3-H1-8-1}, only depends on $\Omega$, initial data, $\f$ and $\nu$ as well. Of course, $C_7 = \alpha_1^{-1} (G_1^0 + \beta_0^{-1} \nu^{-1} C_6 )^{-\frac12}$, as defined in \eqref{est-BDF3-H1-8-8}, becomes a constant only dependent on $\Omega$, initial data, $\f$ and $\nu$.
\end{rem}
	
	\subsection{Higher order $\ell^\infty (0, T; H^s)$ estimate} 
	
	To simplify the notations, the following notations of differential operators are recalled, at both the continuous and discrete levels:  
	\begin{eqnarray} \nonumber
		\nabla^s f = \left\{  \begin{array}{l} 
			\Delta^k  f  ,   \quad\,\,\,\,\, \mbox{if $s=2k$} , 
			\\
			\nabla \Delta^k f , \quad \mbox{if $s=2k+1$} , 
		\end{array} \right.  \qquad 
		\nabla_N^s f = \left\{  \begin{array}{l} 
			\Delta_N^k  f  ,   \qquad\,\,\, \mbox{if $s=2k$} , 
			\\
			\nabla_N \Delta_N^k f , \quad \mbox{if $s=2k+1$} . 
		\end{array} \right. 
	\end{eqnarray} 
	
	The uniform-in-time stability analysis for the BDF3 scheme~\eqref{scheme-BDF3-1}-\eqref{scheme-BDF3-3}, in the higher order $H^s$ (with $s \ge 2$), could be performed in a similar style as in the $\ell^\infty (0, T; H^1)$ estimate. The theoretical result is stated in the following theorem; the technical details are skipped for the sake of brevity, and left to interested readers. 
	
	\begin{thm} 
		For the IMEX BDF3 scheme~\eqref{scheme-BDF3-1}-\eqref{scheme-BDF3-3}, we assume that $\omega_0 \in H^{s+1}$, the initial guess data $\omega^{-j}$ are bounded in $H^{s+1}$ ($1 \le j \le 3$),  $\f \in L^\infty(0,T;H^{s-1})$ and set $\|\f\|_{\infty}^{(s-1)} := \|\f\|_{L^\infty(0,T;H^{s-1})} =M^{(s-1)}$. Under the time step constraint~\eqref{constraint-BDF3-dt}, we have 
		\begin{equation*}
				\| \nabla_N^s \pomega^n \| \le \alpha_1^{-1} \Big( 1 + \beta_0 \nu \dt  \Big)^{-\frac{n}{2}}  
			(G_s^0)^\frac12  + R^{(s)} , \quad  \forall\, n\ge0,
		\end{equation*}
		where
		\begin{equation}\nonumber
			\begin{aligned} 
				G_s^0 = & 
				\| \alpha_1 \nabla_N^s  \omega^0 \|_2^2
				+ \| \nabla_N^s ( \alpha_2 \omega^0 + \alpha_3 \omega^{-1} ) \|_2 ^2 
				+ \| \nabla_N^s ( \alpha_4 \omega^0 + \alpha_5 \omega^{-1} + \alpha_6 \omega^{-2} ) \|_2^2
				\\
				& 
				+  \frac{37 \nu \dt}{24} \| \nabla_N^{s+1} \omega^0 \|_2^2 
				+ \frac{17 \nu \dt}{48} \| \nabla_N^{s+1} \omega^{-1} \|_2^2 
				+  \frac{17 \nu \dt}{96} \| \nabla_N^{s+1} \omega^{-2} \|_2^2 
				\\
				& 
				+ \frac56 \| \nabla_N^s ( \omega^0 - \omega^{-1} ) \|_2^2 
				+ \frac16 \| \nabla_N^s ( \omega^{-1} - \omega^{-2} ) \|_2^2 , 
			\end{aligned} 
		\end{equation}
		with $R^{(s)}$ a uniform in time constant  only dependent on $\| \omega^0 \|_{H^{s+1}}$, $\| \f \|_\infty^{(s-1)}$, $\Omega$ and $\nu$.  
	\end{thm}  
	\begin{rem}
	In the theoretical derivation, it is clear that all the generative constants, from $C_3$, $C_4$, $C_{5, \nu}$, $C_6$ and $C_7$, singularly depends on $\nu^{-1}$ in a polynomial pattern. In turn, a larger value of $\nu >0$ leads to milder values of these generative constants, so that the global-in-time analysis becomes easier. Therefore, we only focus on the case of a smaller value of $0 < \nu \le 1$, since the estimate with $\nu > 1$ will become easier.

In fact, in the derivation of \eqref{est-BDF3-L2-8-9}, the assumption of $\nu \le 1$ is for simplicity of presentation. In the case of $\nu > 1$, we are able to replace the coefficient of $\beta_0 \nu$ by $\beta_0$ on the left hand side of \eqref{est-BDF3-L2-8-9}, and all the subsequent estimates could similarly go through. As a result, a less $nu$-dependent global-in-time $\ell^2$ and $H_h^1$ bounds will be established for the numerical solution, in the case of $nu >1$.

In the practical scientific computation problems, it is well known that a small value of $\nu > 0$ corresponds to a higher Reynolds number, and the scientific computing becomes more challenging. Because of this fact, we focus on the theoretical analysis with $0 < \nu \le 1$ for simplicity of presentation.

With the above arguments, we see that there is no need to modify the statement of Theorems 2.1 and 3.1.
	\end{rem}
	\begin{rem}
		With a uniform-in-time $H^m$ bound derived for the numerical solution \eqref{scheme-BDF3-1}-\eqref{scheme-BDF3-3}, a statistical convergence is expected to be available, using similar techniques presented in \cite{gottlieb2012stability,wang2012efficient}. The details are left to future works.
	\end{rem}

	\section{An optimal rate convergence analysis for the BDF3 scheme} 
	
	Now we proceed into the convergence analysis. Denote $(\omega_e, \u_e, \psi_e)$ as the exact solution to the original NSE system~\eqref{NSE1}-\eqref{NSE3}. With an initial data with sufficient regularity, the exact solution is assumed to preserve a regularity of class $\mathcal{R}$: 
	\begin{equation}
		\omega_{e} \in \mathcal{R} := C^4 (0,T; C^0) \cap H^3 (0, T; H^1) \cap L^\infty (0,T; H^{m+2}) . 
		\label{assumption:regularity.1}
	\end{equation}
	
	To facilitate the error estimate, we need to ensure the numerical error is mean free. Define $\U_N (\, \cdot \, ,t) := {\cal P}_N \u_e (\, \cdot \, ,t)$, $\ppsi_N (\, \cdot \, ,t) := {\cal P}_N \psi_e (\, \cdot \, ,t)$, $\pomega_N (\, \cdot \, ,t) := {\cal P}_N \omega_e (\, \cdot \, ,t)$, the (spatial) Fourier projection of the exact solution into ${\cal B}^K$, in the velocity, stream function, and vorticity variables, respectively.  The following projection approximation is standard: if $\omega_e \in L^\infty(0,T;H^\ell_{\rm per}(\Omega))$, for some $\ell\in\mathbb{N}$,
	\begin{equation} 
		\| \pomega_N - \omega_e \|_{L^\infty(0,T;H^k)}  
		\le C h^{\ell-k} \| \omega_e \|_{L^\infty(0,T;H^\ell)},  \quad \forall \ 0 \le k \le \ell . 
		\label{projection-est-0} 
	\end{equation} 
	By $\pomega_N^m$ we denote $\pomega_N(\, \cdot \, , t^m)$, with $t^m = m\cdot \dt$. Since $\pomega_N \in {\cal B}^K$, the zero-average property is available at the discrete level: 
	\begin{equation} 
		\overline{\pomega_N^m} = \frac{1}{|\Omega|}\int_\Omega \, \pomega_N ( \cdot, t^m) \, d {\bf x}
		= 0 ,  \quad \forall \ m \in\mathbb{N}.  
		\label{mass conserv-1} 
	\end{equation} 
	On the other hand, the solution of the BDF3 scheme~\eqref{scheme-BDF3-1}-\eqref{scheme-BDF3-3} is also zero-average at the discrete level, as stated in Proposition~\ref{prop:average}:  
	\begin{equation} 
		\overline{\omega^m} = 0 ,  \quad \forall \ m \in \mathbb{N} .  
		\label{mass conserv-2} 
	\end{equation} 
	Meanwhile, we use the mass conservative projection for the initial data:  
	\begin{equation}
		\omega^0_{i,j} = ( \pomega_N^0 )_{i,j} := \pomega_N (x_i, y_j, t=0) . 
		\label{initial data-0}
	\end{equation}	
	The error grid functions, for the vorticity, velocity and stream function variables, are defined as 
	\begin{equation} 
		e_\omega^m := \pomega_N^m - \omega^m ,  \, \, \, 
		e_{\u}^m := \U_N^m - \u^m = (e_u^m , e_v^m ), \, \, \, 
		e_\psi^m := \ppsi_N^m - \psi^m ,  \quad \forall \ m \in \mathbb{N}  .   
		\label{error function-1}
	\end{equation} 
	Therefore, it follows that  $\overline{e_\omega^m} =0$, for any $m \in \mathbb{N}$. 
	
	For the third order BDF scheme~\eqref{scheme-BDF3-1}-\eqref{scheme-BDF3-3}, the $\ell^\infty (0, T; \ell^2) \cap \ell^2 (0, T; H^1)$ convergence result is stated below. 
	
	\begin{thm}
		\label{thm:convergence}
		Given initial data $\pomega_N^{0} \in C_{\rm per}^{m +2} (\overline{\Omega})$, with periodic boundary conditions, suppose the unique solution for the the NSE equation~\eqref{NSE1}-\eqref{NSE3} is of regularity class $\mathcal{R}$. Then, provided $\dt$ and $h$ are sufficiently small, 
		for all positive integers $k$, such that $\dt \cdot k \le T$, the following convergence estimate is valid: 
		\begin{equation}
			\| e_{\omega}^k \|_2 +  \Big( \nu \dt   \sum_{j=1}^{k} \| \nabla_N e_{\omega}^j \|_2^2 \Big)^\frac12  
			\le C ( \dt^3 + h^m ),   \label{convergence-0}
		\end{equation}
		where $C>0$ depends on the final time $T$, but is independent of $\dt$ and $h$.
	\end{thm}
For the sake of readability, the detailed proof is relegated to Appendix~A.

\section{Numerical experiments} 
All numerical experiments in this section are conducted on periodic domains, 
which constitute the natural setting for the Fourier pseudo-spectral discretization adopted in this work.
Benchmark problems involving no-slip wall boundaries, such as the lid-driven cavity flow, 
would require fundamentally different spatial discretizations and are therefore not considered here.
\subsection{Convergence, conservation and long-time stability tests}
We consider the NSE \eqref{NSE1}-\eqref{NSE3} equipped with the following sufficiently smooth initial conditions \cite{wang2022linearly}:
\begin{equation}\label{4.1}
	\left\{\begin{aligned}
		& u(x,y,0) = \sin(2\pi x)\cos(2\pi y), \\
		& v(x,y,0) = -\cos(2\pi x)\sin(2\pi y), \quad(x,y)\in[0,1]\times[0,1].
	\end{aligned}\right.
\end{equation}
The initial condition for $\omega$ is deduced from the initial condition \eqref{4.1} for the velocity, i.e.,
\begin{equation*}
	\omega = \omega_0 = \partial_xv(x,y,0) - \partial_yu(x,y,0),
\end{equation*}
then the numerical solution is compared to the exact Taylor-Green vortex solution of the NSE \eqref{NSE1}-\eqref{NSE3}:
\begin{equation}\nonumber
	\left\{\begin{aligned}
		& u_e(x,y,t) = \sin(2\pi x)\cos(2\pi y)\exp(-8\nu\pi^2t), \\
		& v_e(x,y,t) = -\cos(2\pi x)\sin(2\pi y)\exp(-8\nu\pi^2t).
	\end{aligned}\right.
\end{equation}
For spatial discretization, we use uniform mesh with $N_x = N_y = 128$. In temporal discretization, we set the final time $T=1$ and several time-step size $\Delta t_i = 0.02\times2^{-i}~(i=0,\ldots,4)$. 
We employ the second-order Runge-Kutta method for computing $\omega_1$ and the IMEX-BDF2 method for computing $\omega_2$, which ensures the third order convergence in time. We set $\nu = 0.001$ to demonstrate the performance of the scheme for the convection-dominated NSE. 
Note that only one FFT-based Poisson solver is needed at each time step and the computational efficiency is significantly enhanced because of the explicit extrapolation.
The numerical results are plotted in Figs \ref{fig1} and \ref{fig2}, where
\begin{figure}[htbp]
	\centering
	\subfigure{\includegraphics[width=.4\linewidth]{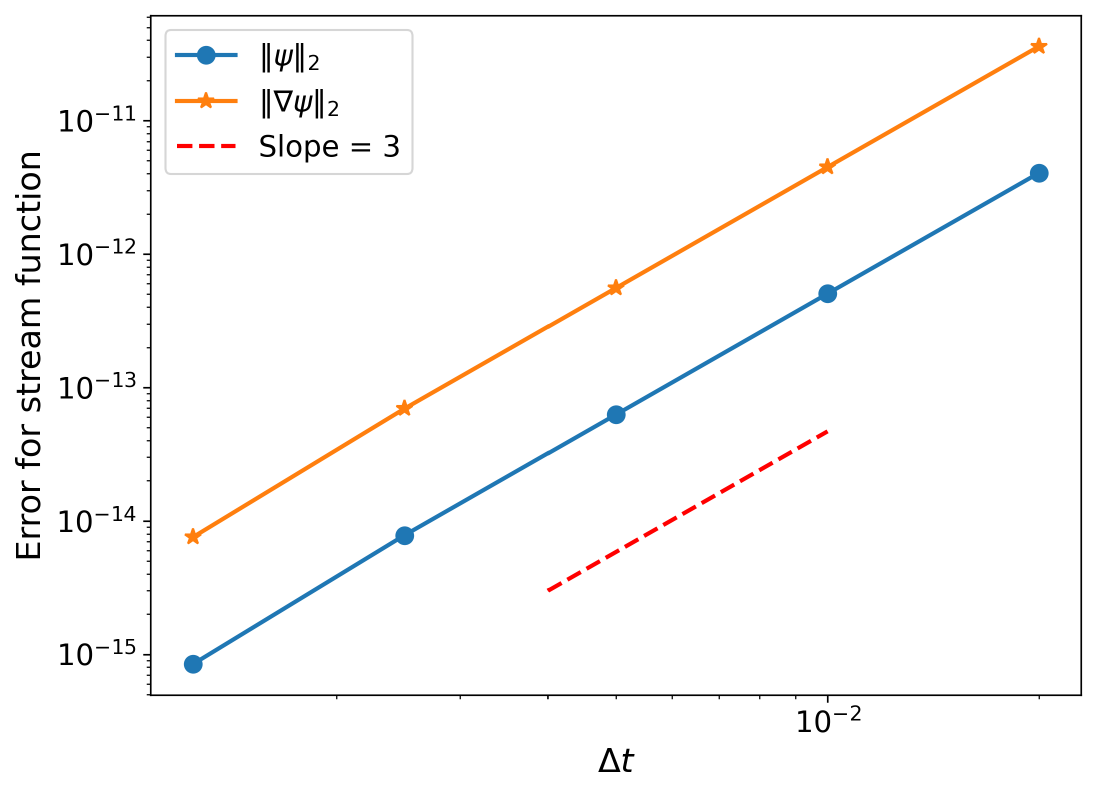}}
	\subfigure{\includegraphics[width=.4\linewidth]{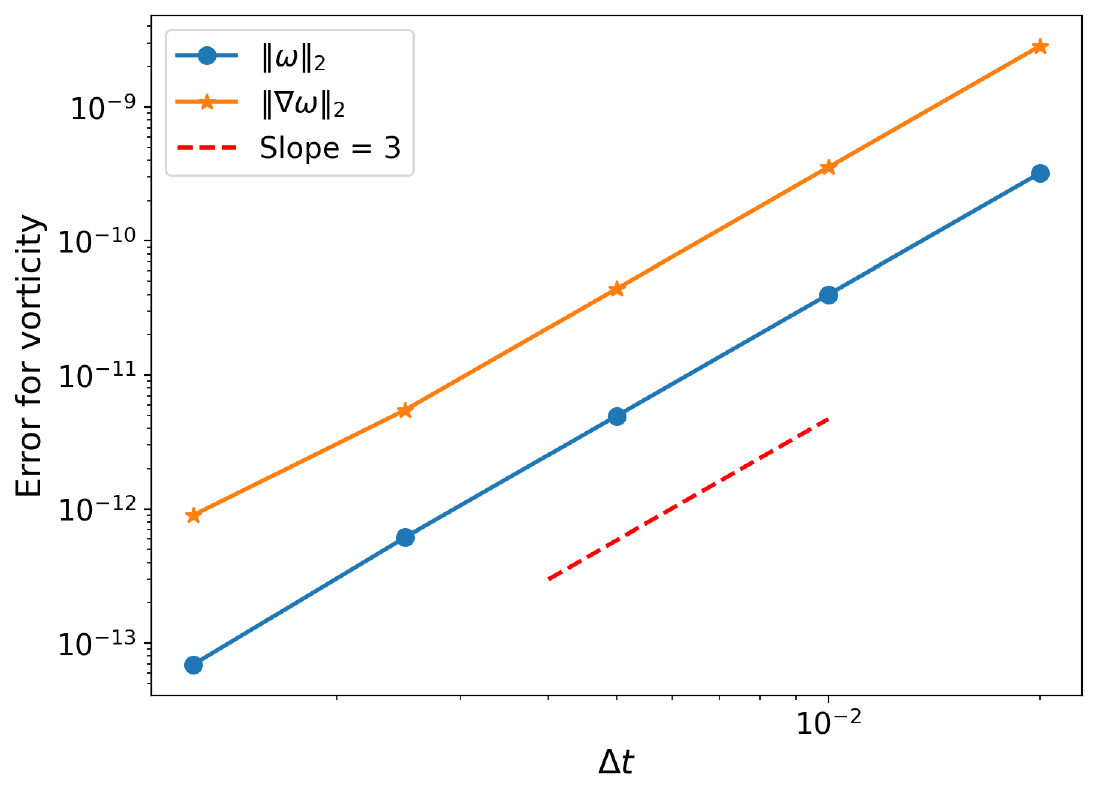}}
	\caption{$\ell^{\infty}(0,T;L^2)$ and $\ell^2(0,T;H^1)$ errors for stream function and vorticity by the IMEX BDF3 scheme \eqref{scheme-BDF3-1}-\eqref{scheme-BDF3-3}.}
	\label{fig1}
\end{figure} 
\begin{figure}[htbp]
	\centering
	\subfigure{\includegraphics[width=.4\linewidth]{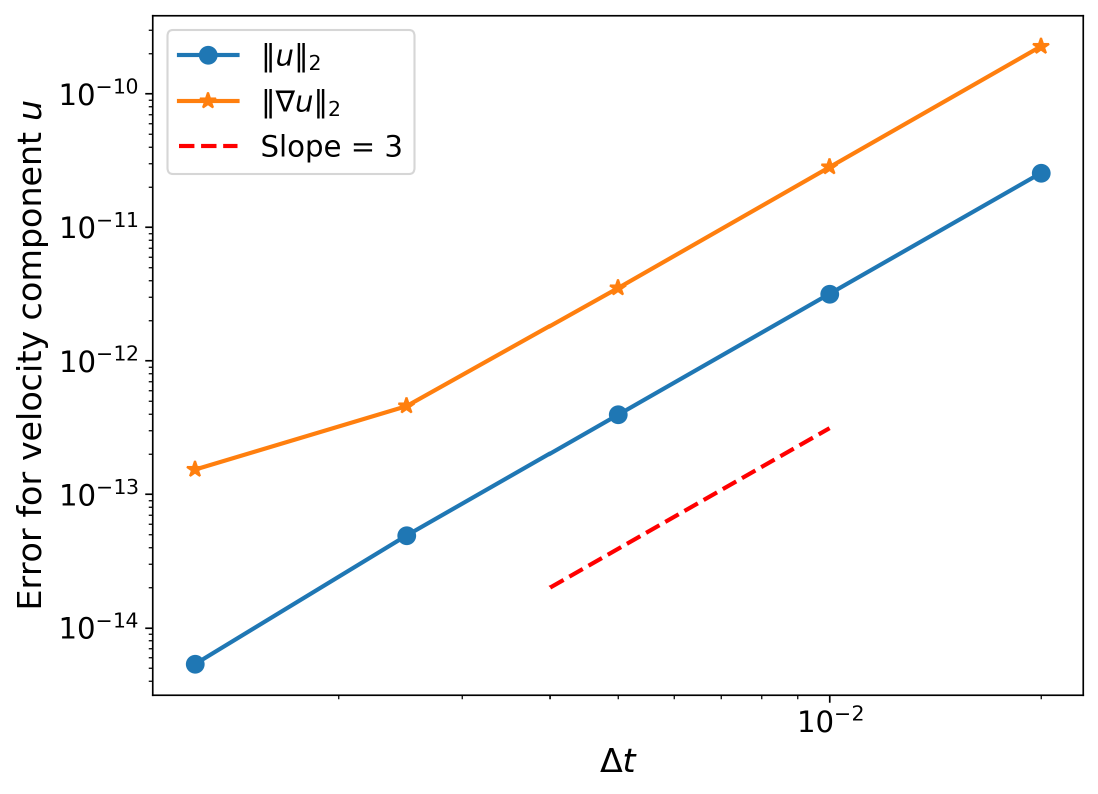}}
	\subfigure{\includegraphics[width=.4\linewidth]{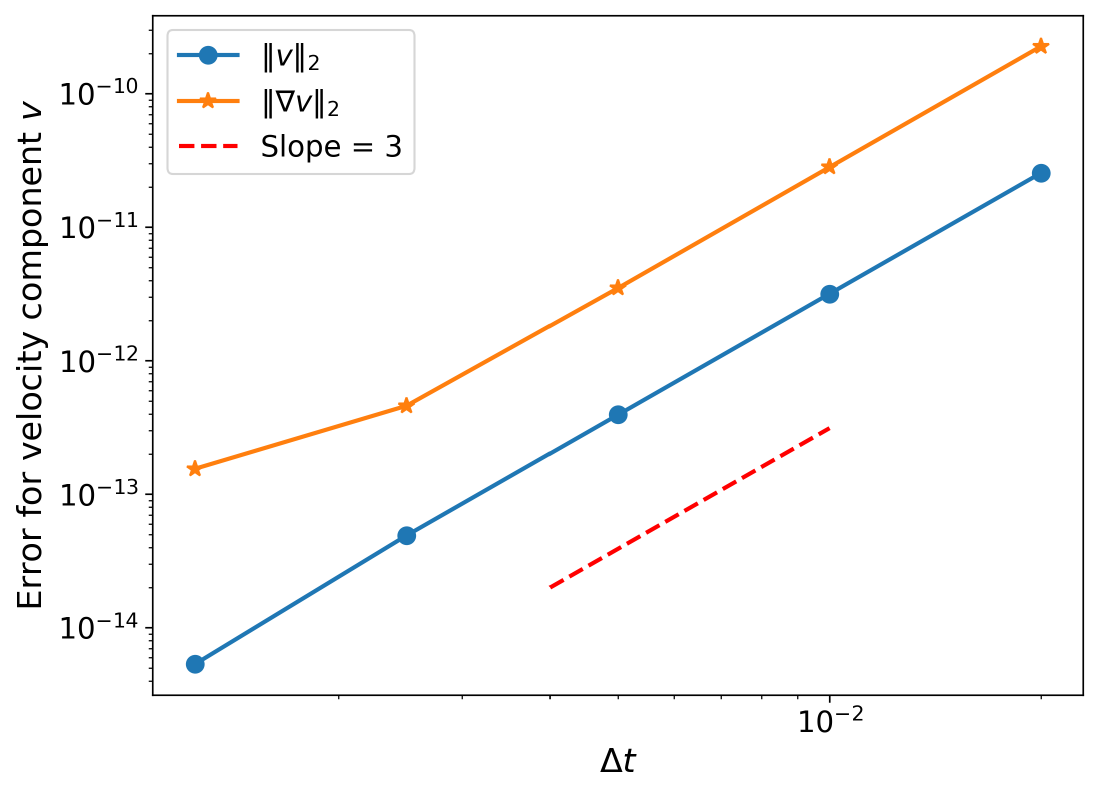}}
	\caption{$\ell^{\infty}(0,T;L^2)$ and $\ell^2(0,T;H^1)$ errors for velocity by the IMEX BDF3 scheme \eqref{scheme-BDF3-1}-\eqref{scheme-BDF3-3}.}
	\label{fig2}
\end{figure}
From Figures~\ref{fig1} and~\ref{fig2}, it is observed that the stream function, vorticity, and velocity exhibit the same error behavior; that is, the scheme is of optimal order 3 in both $\ell^{\infty}(0,T;L^2)$- and $\ell^2(0,T;H^1)$-norms, and a strong verification of Theorem \ref{thm:convergence} is given by these results.

Next, we present long-time numerical simulation results by performing the computation up to $T=100$. The local-in-time convergence result presented in \eqref{convergence-0} holds limited theoretical relevance for such extended time scales, as the associated convergence constant grows exponentially with the final time \( T \). To better understand the long-time behavior, we instead analyze the temporal evolution of the vorticity profile by monitoring its discrete \( L^2 \) and \( H^1 \) norms. The time evolution plot is presented in Figure \ref{fig3}, in which one could clearly observe that all monitored energy norms remain uniformly bounded over the entire simulation interval. This numerical simulation validates the long-time stability analysis given by Theorem \ref{thm:stability}.

In addition to the vorticity norms, we also monitor several physically relevant quantities, including the discrete kinetic energy 
\[
E(t)=\frac12\| \bm{u}(t)\|_2^2,
\]
the enstrophy 
\[
Z(t)=\frac12\|\omega(t)\|_2^2,
\]
and the discrete divergence error $\|\nabla_N\cdot \bm{u}\|_2$. For the periodic incompressible NSE without external forcing, 
the kinetic energy is expected to decay monotonically due to viscosity, while the divergence error should remain close to machine precision. In Figure \ref{extra}, the discrete divergence error is shown throughout the simulation and remains at the level of machine accuracy for all times, 
confirming that the proposed scheme preserves the incompressibility constraint in practice.

\begin{figure}[htbp]
	\centering
	\includegraphics[width=1\linewidth]{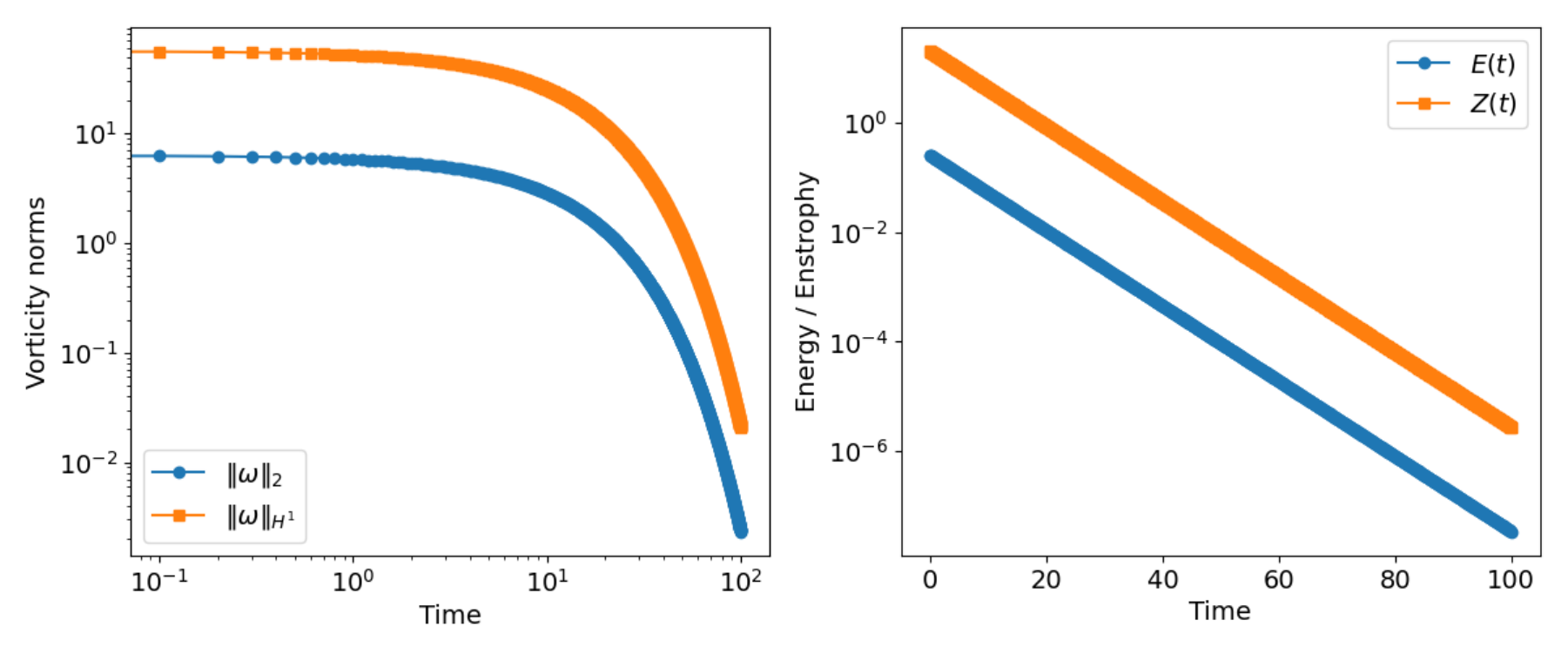}
\caption{Long-time diagnostics for the Taylor-Green vortex up to $T=100$ computed by the IMEX BDF3 Fourier pseudo-spectral scheme.
Left: time evolution of the discrete vorticity norms $\|\omega(t)\|_{2}$ and $\|\omega(t)\|_{H^1}$, which remain uniformly bounded for all times.
Right: kinetic energy $E(t)=\frac12\|\bm{u}(t)\|_2^2$ and enstrophy $Z(t)=\frac12\|\omega(t)\|_2^2$, both exhibiting the physically consistent exponential decay induced by viscosity.
These results provide quantitative evidence of the long-time stability and robustness of the proposed scheme.}
	\label{fig3}
\end{figure}
\begin{figure}[htbp]
	\centering
	\includegraphics[width=0.6\linewidth]{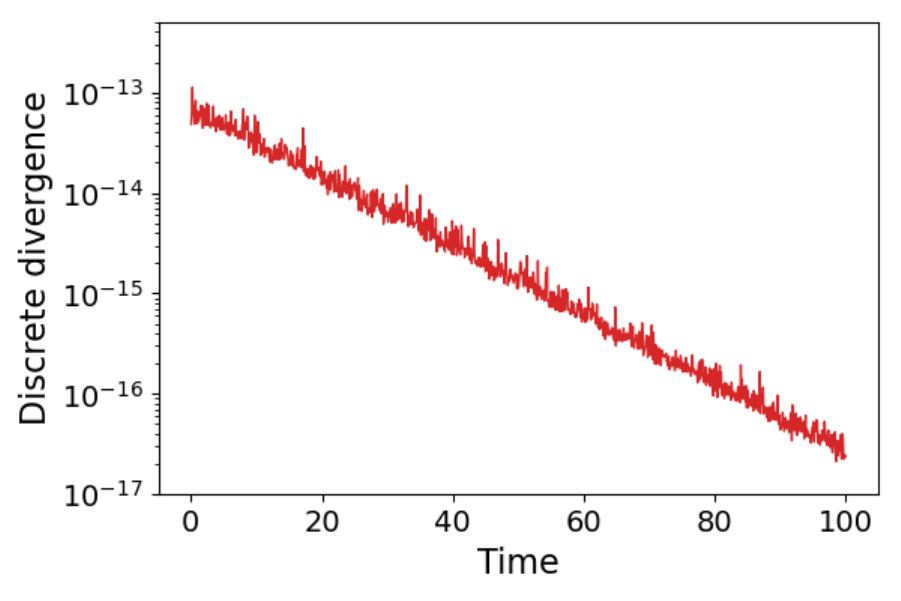}
\caption{Time evolution of the discrete divergence error 
$\|\nabla_N\cdot \bm{u}(t)\|_2$ for the long-time Taylor--Green vortex simulation.
The divergence error remains at the level of machine precision for all times, 
indicating that the incompressibility constraint is preserved in the discrete sense.}
	\label{extra}
\end{figure}
\subsection{Benchmark problem: double shear layer flow}
We further examine the robustness of the proposed IMEX BDF3 scheme in convection-dominated regimes by considering the classical double shear layer benchmark on the periodic domain $[0,1]^2$ \cite{minion1997performance}. 
This problem is widely used to assess nonlinear stability and vortex roll-up dynamics at high Reynolds numbers under periodic boundary conditions, which is fully compatible with the Fourier pseudo-spectral discretization adopted in this work.

The initial velocity is given by
\begin{equation}
\begin{aligned}
u(x,y,0) &=
\begin{cases}
\tanh(\rho(y-0.25)), & y\le 0.5,\\
\tanh(\rho(0.75-y)), & y>0.5,
\end{cases}
\qquad
v(x,y,0)=\delta\sin(2\pi x),
\end{aligned}
\end{equation}
where $\rho$ controls the layer thickness and $\delta$ is the perturbation amplitude.

\paragraph{Thick layer case (high Reynolds number)}
We first consider $\rho=30$ with $\nu=10^{-4}$ (corresponding to $Re=10^4$). 
The computation is performed with $128\times128$ Fourier modes up to $T=1.2$. 
Figure~\ref{fig4} compares the vorticity snapshots obtained by the first-, second-, and third-order IMEX schemes using the same time step. 
The IMEX Euler scheme produces an incorrect vortex pattern, and the IMEX BDF2 scheme exhibits visible spurious oscillations, whereas the proposed IMEX BDF3 scheme remains stable and captures the expected vortex roll-up structure.

\paragraph{Thin layer case (more challenging, higher Reynolds number)}
We then test the thin layer setting with $\rho=100$ and $\nu=5\times10^{-5}$ (i.e., $Re=2\times10^4$), using $256\times256$ Fourier modes and $\Delta t=4\times10^{-4}$ up to $T=1.2$. 
Figure~\ref{fig5} shows representative vorticity contours, which are in qualitative agreement with existing reference results reported in \cite{bell1989second,huang2021stability,di2023variable}. 
To further quantify robustness at high Reynolds numbers, we additionally report the time evolution of $\|\omega(t)\|_{\infty}$, demonstrating that the computation remains stable without nonphysical energy growth or high-frequency contamination.
Overall, these results indicate that the proposed IMEX BDF3 scheme offers a substantially relaxed time-step restriction and strong nonlinear stability for standard high-Re vortex-dynamics benchmarks on periodic domains.

As shown in Figure~\ref{fig:omega_inf_shearlayer}, $\|\omega(t)\|_{\infty}$ remains bounded and evolves smoothly in time without spurious growth or blow-up, 
which provides additional evidence of the nonlinear stability of the proposed IMEX BDF3 scheme in convection-dominated regimes.
\begin{figure}[htbp]
	\centering
	\subfigure[1st order]{\includegraphics[width=2.1in]{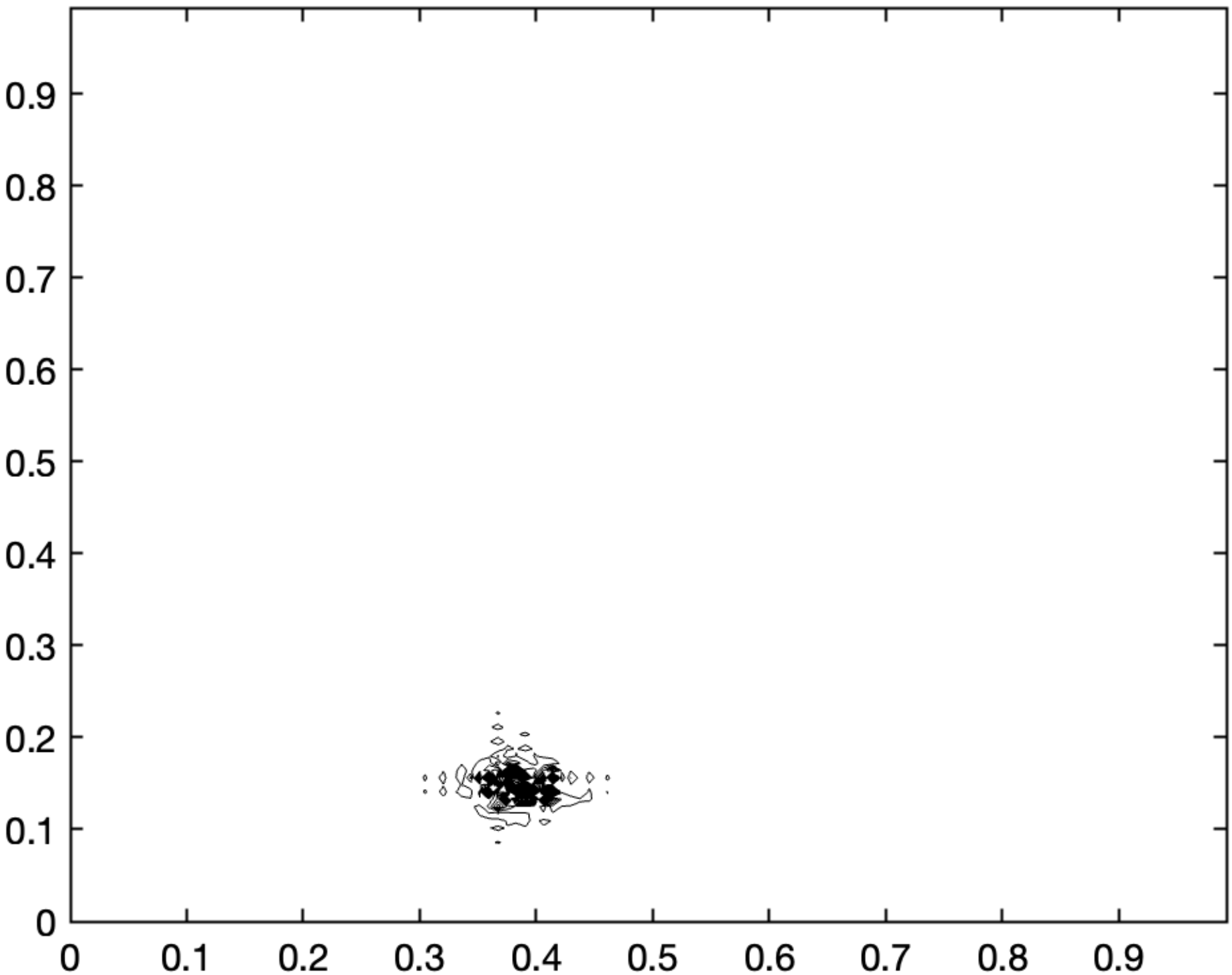}}
	\subfigure[2nd order]{\includegraphics[width=2.1in]{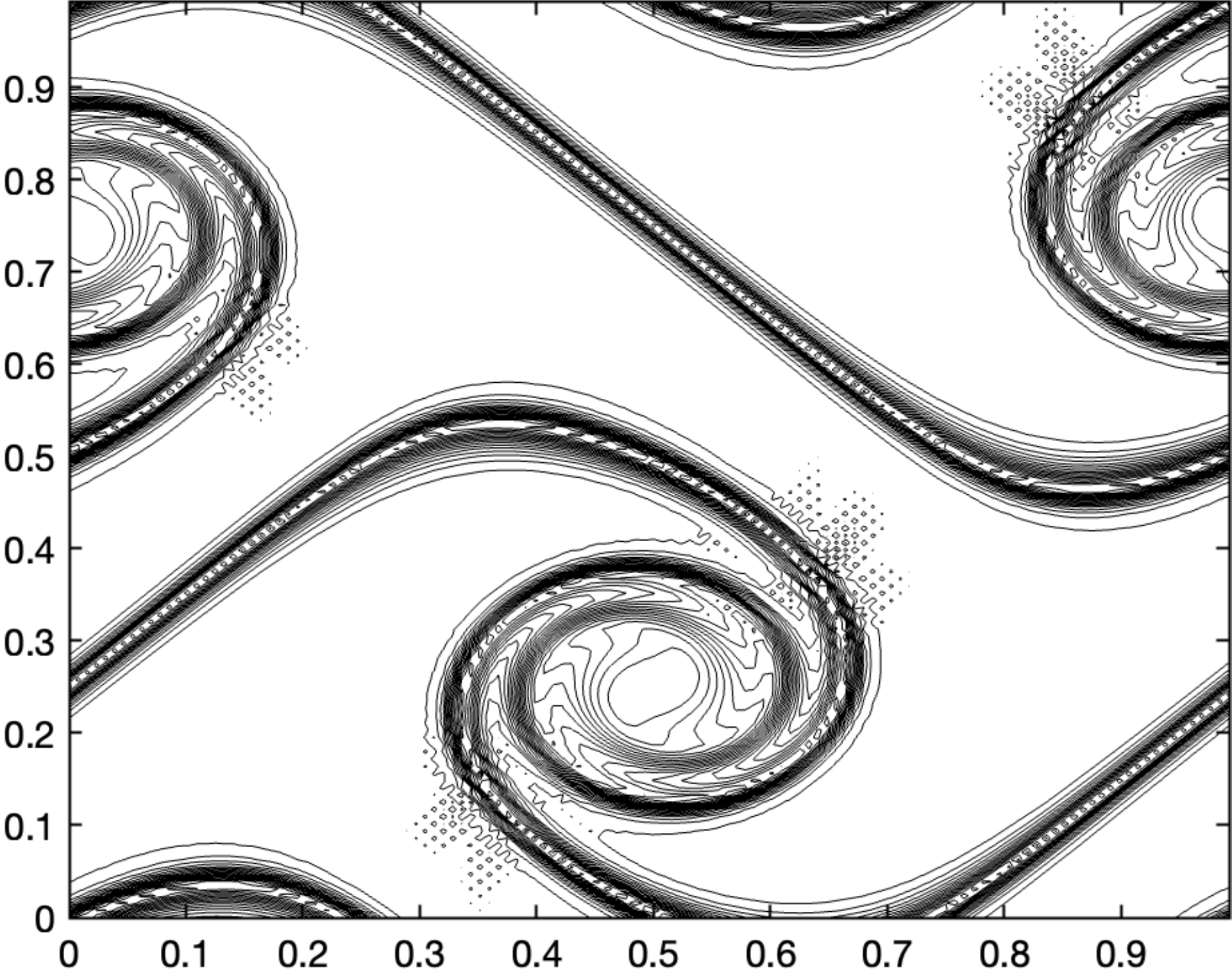}}
	\subfigure[3rd order]{\includegraphics[width=2.1in]{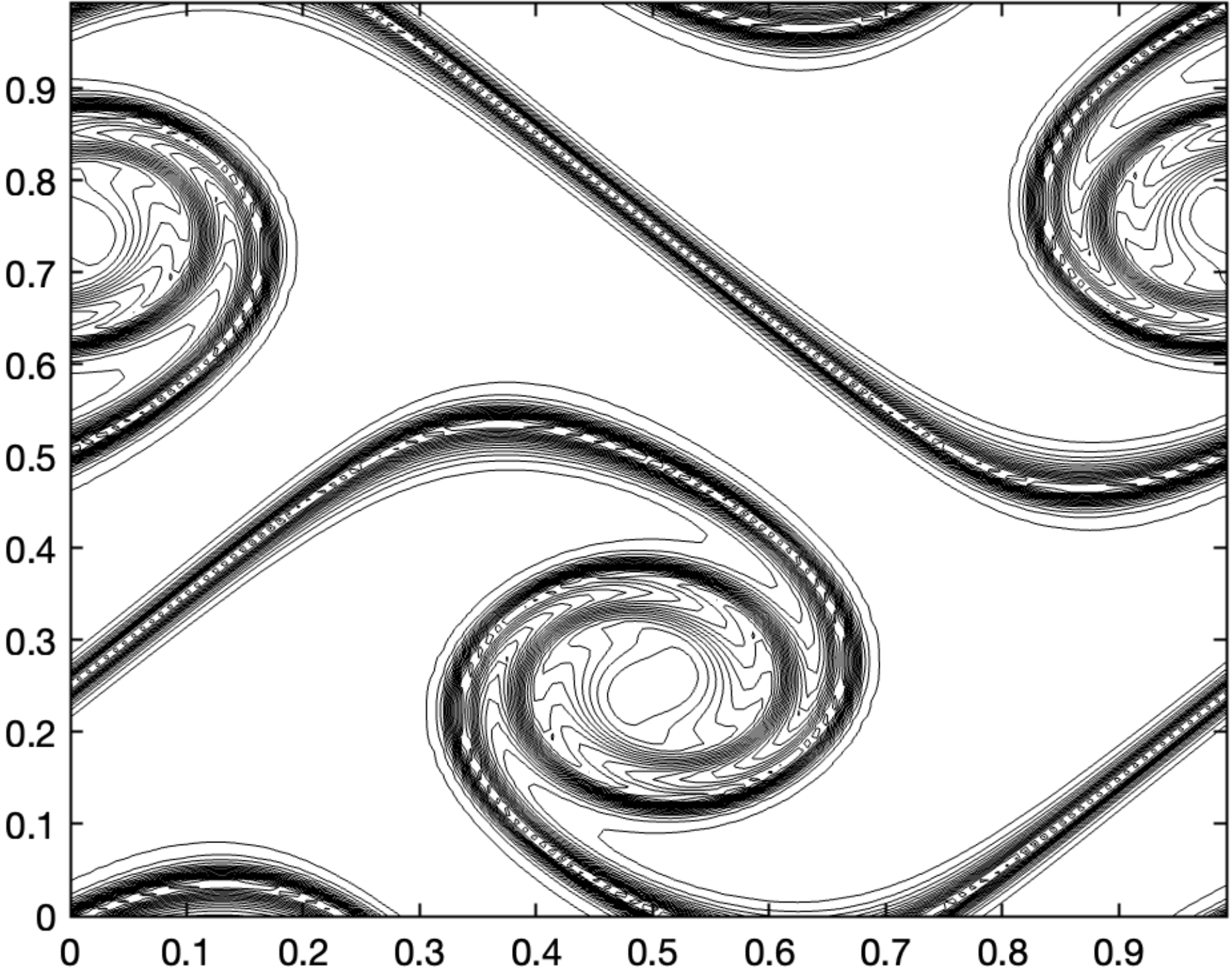}}
\caption{Snapshots of the vorticity field for the thick double shear layer benchmark at the final time $T=1.2$ ($\Delta t = 8\times 10^{-4}$), 
computed using the IMEX Euler scheme (left), IMEX BDF2 scheme (middle), and IMEX BDF3 scheme (right). 
While the first-order and second-order schemes exhibit noticeable numerical distortions, 
the third-order IMEX BDF3 scheme yields a smooth and physically consistent vortex structure, 
demonstrating its improved stability and robustness for convection-dominated flows.}
	\label{fig4}
\end{figure} 
\begin{figure}[htbp]
	\centering
	\subfigure[$t=0.2$]{\includegraphics[width=2.1in]{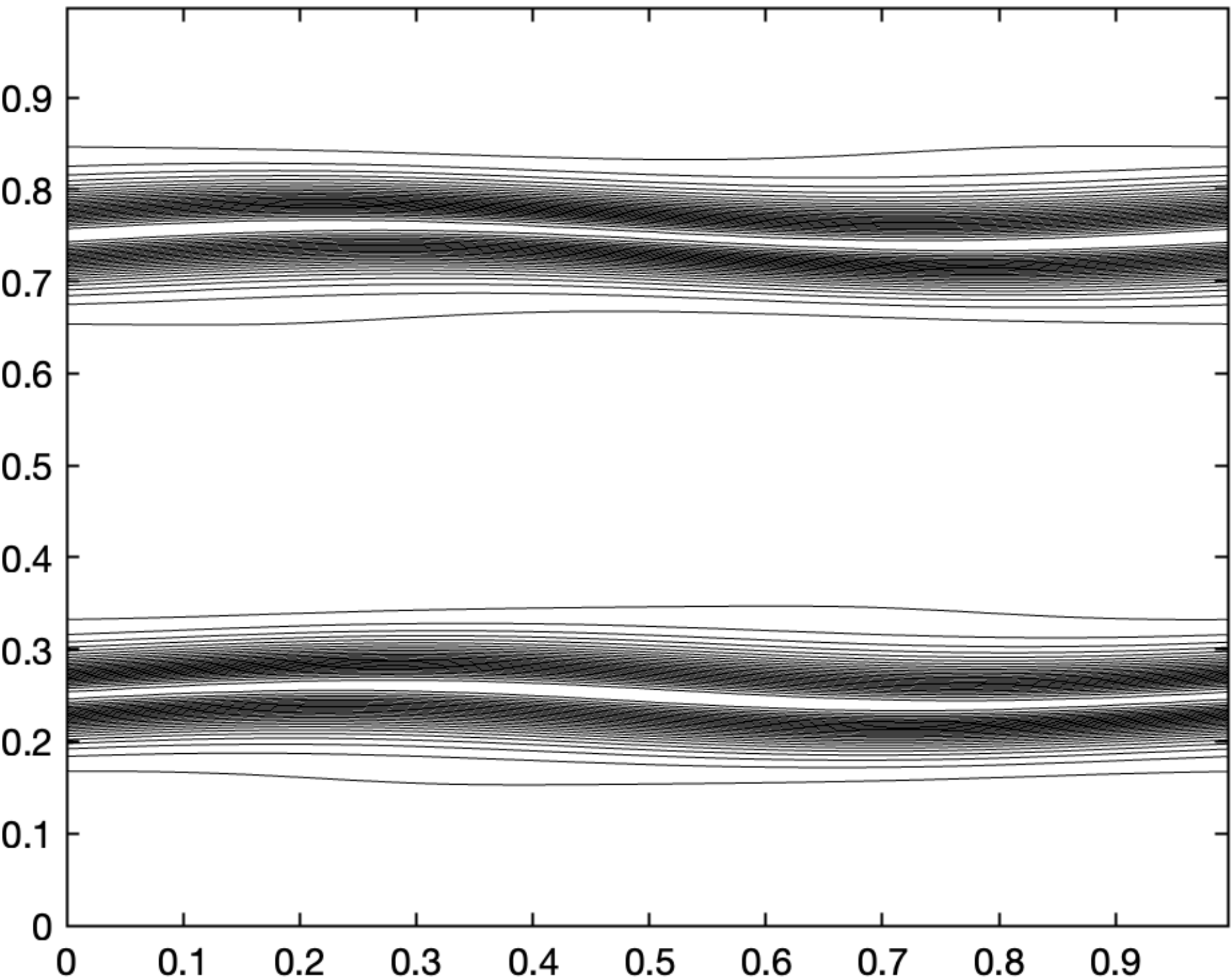}}
	\subfigure[$t=0.4$]{\includegraphics[width=2.1in]{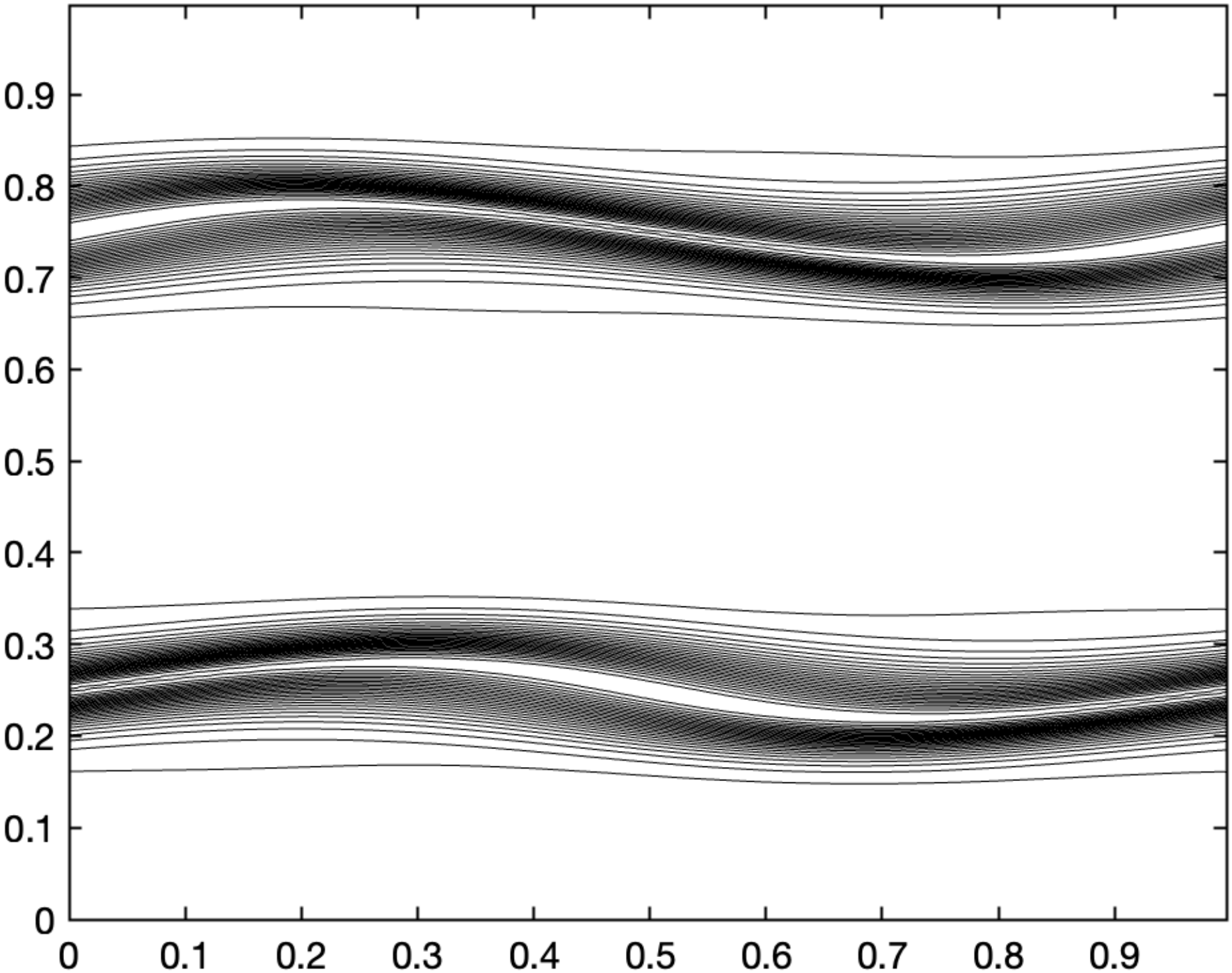}}
	\subfigure[$t=0.6$]{\includegraphics[width=2.1in]{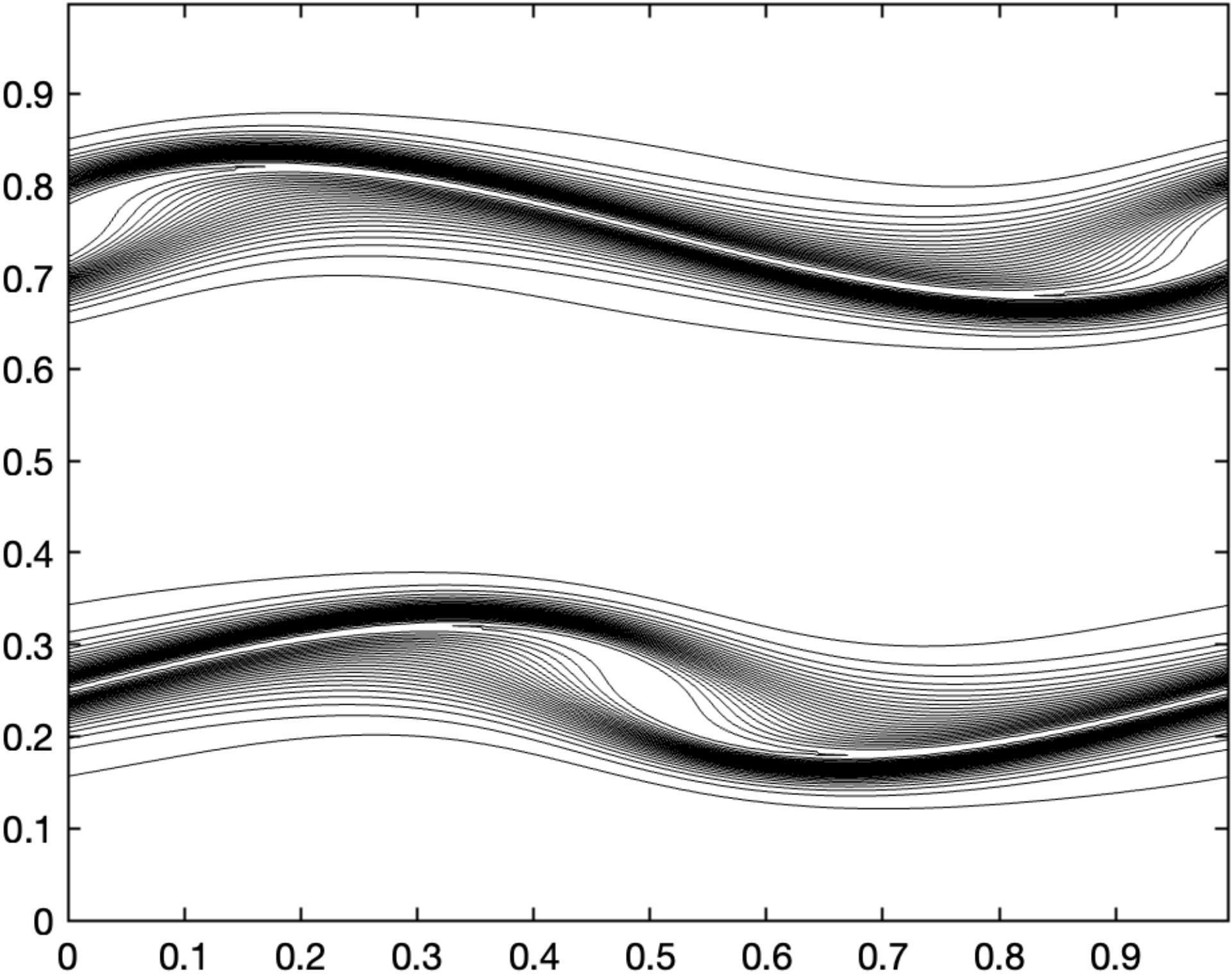}}
	\quad
	\subfigure[$t=0.8$]{\includegraphics[width=2.1in]{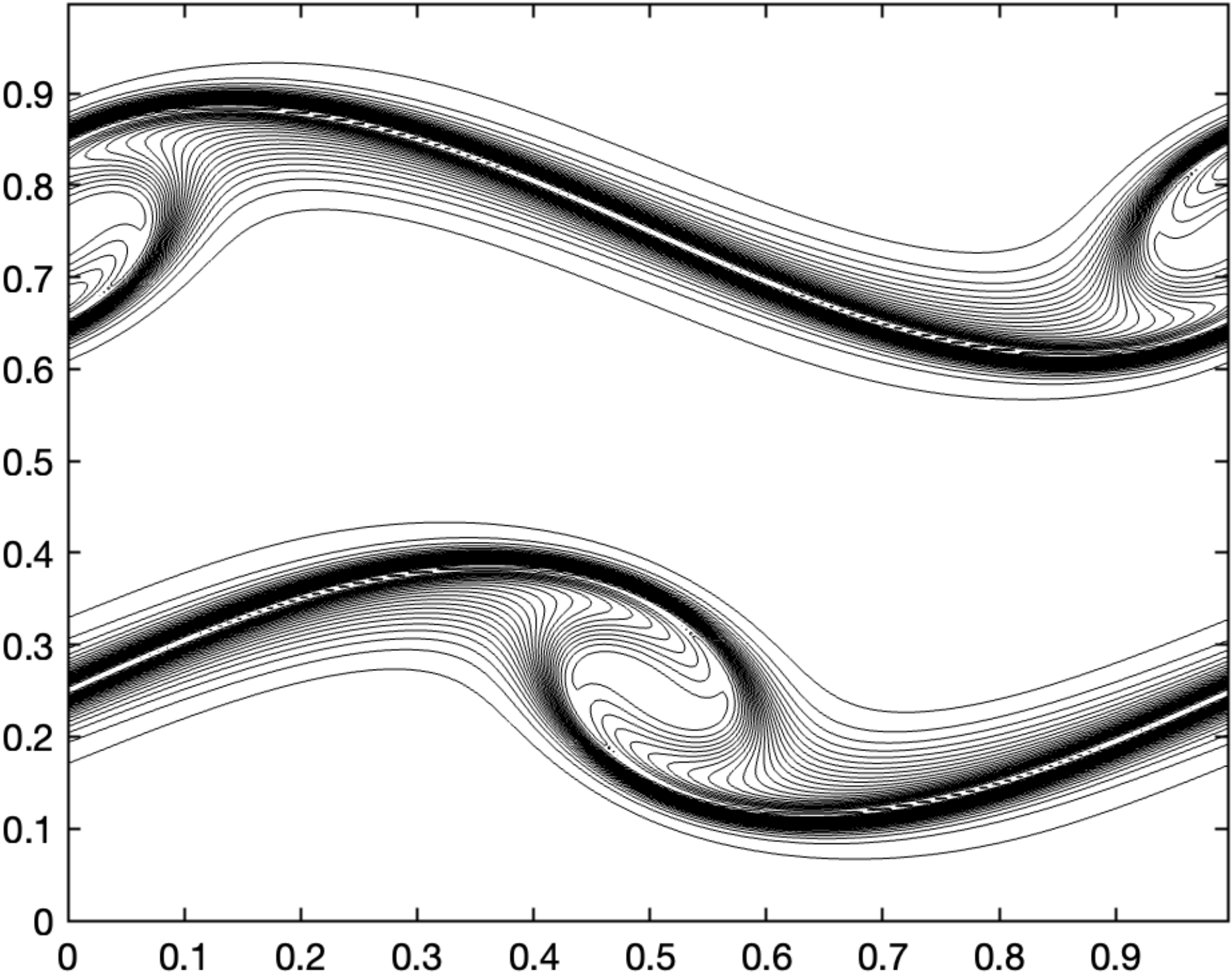}}
	\subfigure[$t=1.0$]{\includegraphics[width=2.1in]{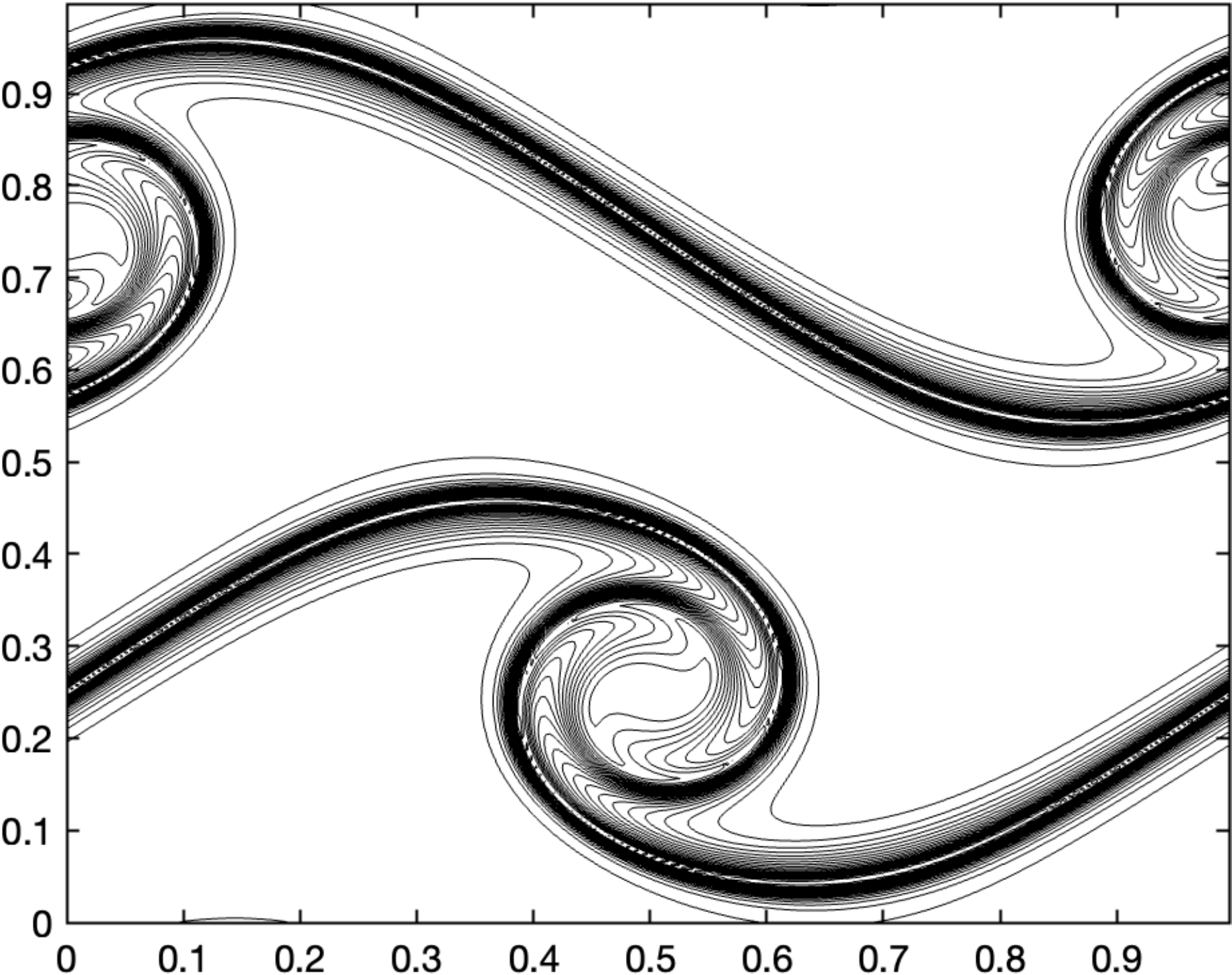}}
	\subfigure[$t=1.2$]{\includegraphics[width=2.1in]{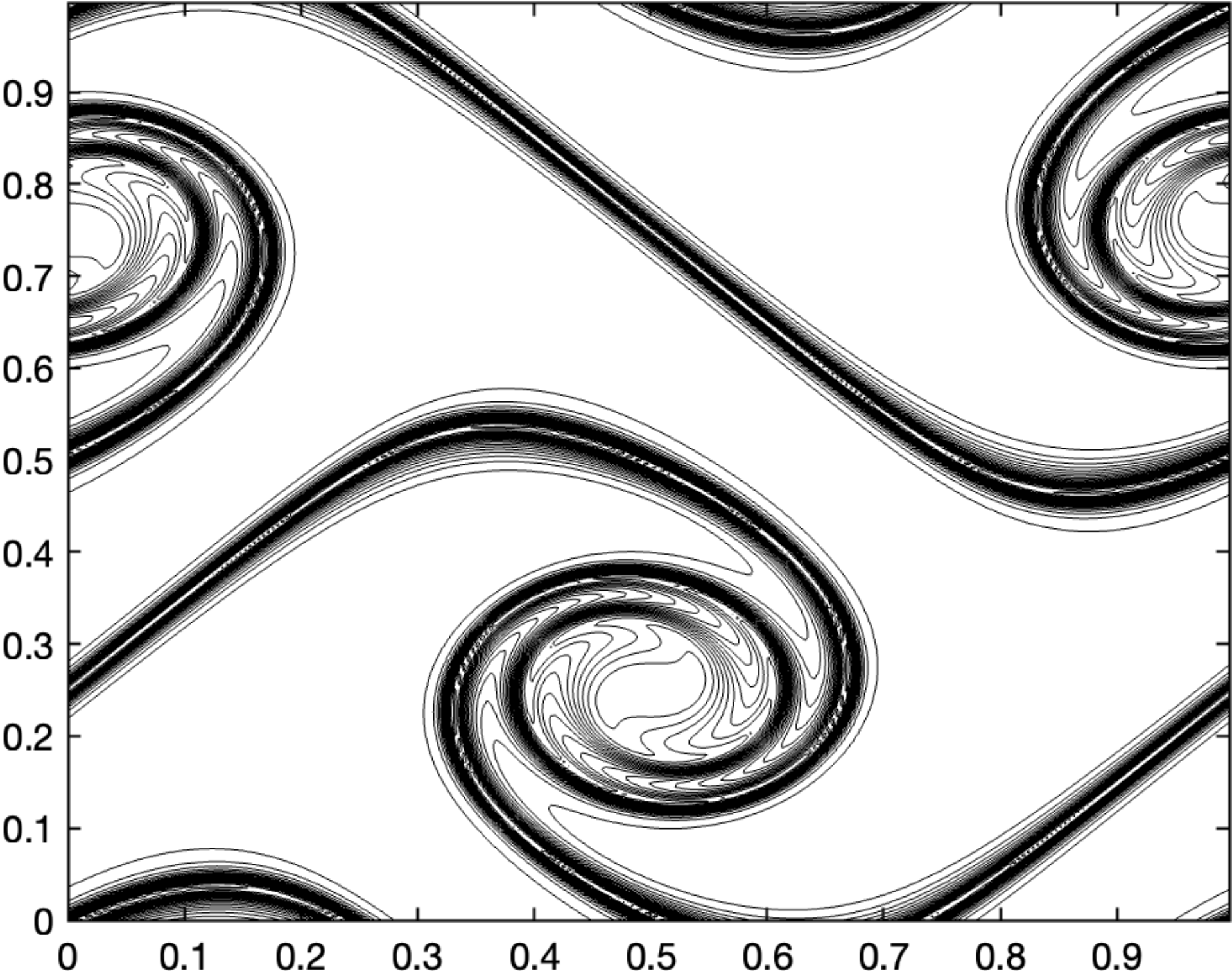}}
\caption{Snapshots of the vorticity field for the thin double shear layer benchmark at different times, 
computed using the third-order IMEX BDF3 scheme~\eqref{scheme-BDF3-1}-\eqref{scheme-BDF3-3} with $\nu=5\times10^{-5}$. 
The numerical solution exhibits the characteristic roll-up and interaction of vortical structures, 
and is qualitatively consistent with previously reported results in the literature.}
	\label{fig5}
\end{figure} 
\begin{figure}[htbp]
\centering
\includegraphics[width=0.6\linewidth]{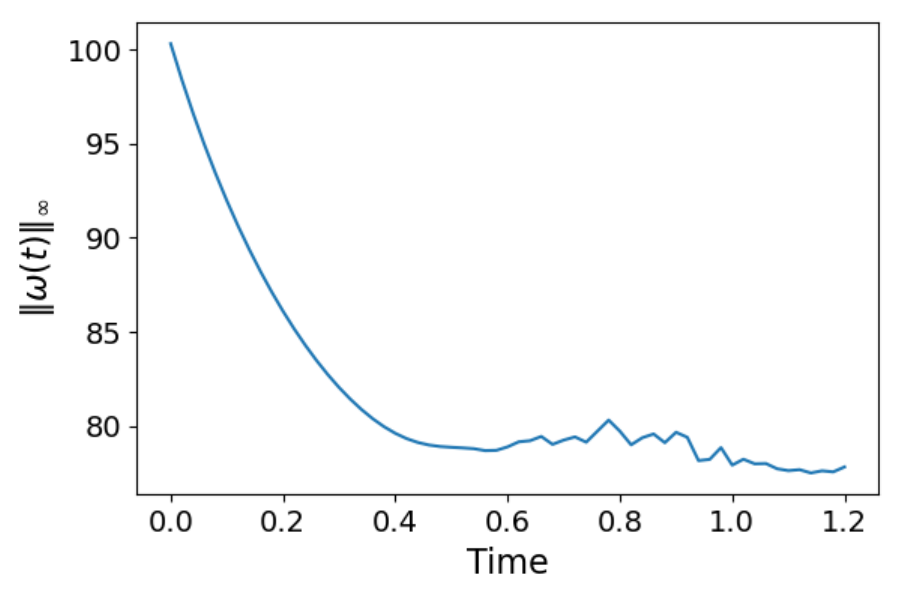}
\caption{Time evolution of the maximum vorticity magnitude $\|\omega(t)\|_{\infty}$ for the thin double shear layer benchmark.
The quantity $\|\omega(t)\|_{\infty}$ remains bounded and shows no spurious growth over the whole simulation interval.}
\label{fig:omega_inf_shearlayer}
\end{figure}

\section{Concluding remarks}
In this paper, we proposed an IMEX BDF3 scheme for the 2-D incompressible Navier-Stokes equation. This scheme improves the computational efficiency on the premise of ensuring stability and accuracy. In addition, we provided a long-time stability analysis for the vorticity and a rigorous convergence analysis for the fully discrete scheme, combined with Fourier pseudo-spectral spatial approximation. The numerical experiments verified the main theorems and demonstrated the performance of the proposed scheme.
\section*{Acknowledgement}
This work was supported by the Natural Science Foundation of China (Grant Nos. 12271523, 12371374, 12501504), and National University of Defense Technology (Grant Nos: 2023-lxy-fhjj-002, 25-ZZCX-JDZ-13), and the Hunan Province Postgraduate Research and Innovation Project (Grant No. XJJC2025003), and the Fundamental Research Funds for the Central Universities of Civil Aviation Flight University of China (Grant No. 25CAFUC03055).
\section*{Declaration of competing interest}
The authors declare that they have no known competing financial interests or personal relationships that could have appeared to influence the work reported in this paper.
\section*{Data availability}
Data will be made available on request.
\appendix
\renewcommand{\theequation}{A.\arabic{equation}}
\setcounter{equation}{0}
\section*{Appendix A. Proof of Theorem \ref{thm:convergence}}
		For the Fourier projection solution $(\pomega_N, \U_N, \ppsi_N)$ and their point-wise interpolation, a careful consistency analysis implies that 
		\begin{align} 
			&
			\frac{\frac{11}{6} \pomega_N^{n+1} - 3 \pomega_N^n 
				+ \frac32 \pomega_N^{n-1} - \frac13 \pomega_N^{n-2}}{\dt}  
			+ \frac32 ( \U_N^n \DOT \nabla_N \pomega_N^n  
			+ \nabla_N \cdot ( \U_N^n \pomega_N^n )  )   \nonumber\\
			& - \frac32 ( \U_N^{n-1} \DOT \nabla_N \pomega_N^{n-1}  
			+ \nabla_N \cdot ( \U_N^{n-1} \pomega_N^{n-1} ) )   +  \frac12 ( \U_N^{n-2} \DOT \nabla_N \pomega_N^{n-2}  
			+ \nabla_N \cdot ( \U_N^{n-2} \pomega_N^{n-2} )  ) \nonumber 
			\\
			=& \nu \Delta_N  \pomega_N^{n+1} 
			+ \f_N^{n+1} + \tau^{n+1} ,   \label{BDF3-consistency-1}
			\\
			&
			\Delta_N \ppsi_N^{n+1} = \pomega_N^{n+1} ,   \label{BDF3-consistency-2} 
			\\
			&
			\U_N^{n+1} = \nabla_N^{\bot} \ppsi_N^{n+1}  
			= \left( {\cal D}_{Ny} \ppsi_N^{n+1} , {\cal D}_{Nx} \ppsi_N^{n+1} \right) ,  
			\label{BDF3-consistency-3}
		\end{align}  
		with $\| \tau^{n+1} \|_2 \le C (\dt^3 + h^m)$. In turn, subtracting the numerical algorithm~\eqref{scheme-BDF3-1}-\eqref{scheme-BDF3-3} from the consistency estimate~\eqref{BDF3-consistency-1}-\eqref{BDF3-consistency-3} gives 
		\begin{align} 
			&
			\frac{\frac{11}{6} e_\omega^{n+1} - 3 e_\omega^n 
				+ \frac32 e_\omega^{n-1} - \frac13 e_\omega^{n-2}}{\dt}  
			+ \frac32 {\cal NLE}^n - \frac32 {\cal NLE}^{n-1}   
			+  \frac12 {\cal NLE}^{n-2} 
			= \nu \Delta_N  e_\omega^{n+1} 
			+ \tau^{n+1} ,   \label{BDF3-error-1}
			\\
			&
			\Delta_N e_\psi^{n+1} = e_\omega^{n+1} ,   \label{BDF3-error-2} 
			\\
			&
			e_{\u}^{n+1} = \nabla_N^{\bot} e_\psi^{n+1}  
			= \left( {\cal D}_{Ny} e_\psi^{n+1} , {\cal D}_{Nx} e_\psi^{n+1} \right) ,  
			\label{BDF3-error-3} 
			\\
			& 
			{\cal NLE}^\ell = e_{\u}^\ell \DOT \nabla_N \pomega_N^\ell  
			+ \u^\ell \DOT \nabla_N e_\omega^\ell 
			+ \nabla_N \cdot ( e_{\u}^\ell \pomega_N^\ell + \u^\ell e_\omega^\ell)  
			+ \overline{\u^\ell \DOT \nabla_N \u^\ell} ,  \quad \ell = n, n-1, n-2 . 
			\label{BDF3-error-4} 
		\end{align}  
		
		Due to the regularity requirement~\eqref{assumption:regularity.1} for the exact solution, we see that the projection solution $( \pomega_N, \U_N, \ppsi_N)$ also preserves this regularity, which comes from the approximation estimate~\eqref{projection-est-0}. In turn, a discrete $W_h^{1,\infty}$ norm will stay bounded for $\pomega_N$: 
		\begin{equation} 
			\| \pomega_N^k \|_\infty \le C^*  , \, \, \, 
			\| \nabla_N \pomega_N^k \|_\infty \le C^*,  \quad \forall k \ge 0 .   
			\label{assumption:W1-infty bound}
		\end{equation}  
		In terms of the numerical solution $\u^\ell$ for the velocity vector, we recall that its continuous extension is given by $\U^\ell$, and the associated vorticity variable (in the continuous extension) becomes $\pomega^\ell$. The global-in-time energy stability analysis has recovered that 
		\begin{equation} 
			\| \pomega^\ell \|_{H^\delta} \le \tilde{C}_1 ,  \quad \forall \ell \ge 0 . 
			\label{convergence-a priori-1} 
		\end{equation} 
		Moreover, for the numerical error functions $(e_\omega^\ell, e_{\u}^\ell, e_\psi^\ell)$, for the vorticity, velocity and stream function variables, we assume the associated continuous extension functions as $(e_{\pomega}^\ell, e_{\U}^\ell, e_{\ppsi}^\ell) \in {\cal B}^K$. Based on the kinematic equation $\Delta e_{\ppsi}^\ell = e_{\pomega}^\ell$, $e_{\U}^\ell = \nabla^{\bot} e_{\ppsi}^\ell$, and a careful application of elliptic regularity implies that 
		\begin{equation} 
			\| e_{\u}^\ell \|_2 = \| e_{\U}^\ell \| = \| \nabla^{\bot} e_{\ppsi}^\ell \| \le \| e_{\ppsi}^\ell \|_{H^1} 
			\le C \| e_{\ppsi}^\ell \|_{H^2} \le \gamma^* \| \Delta e_{\ppsi}^\ell \| 
			= \gamma^* \| e_{\pomega}^\ell \| = \gamma^* \| e_\omega^\ell \|_2 . 
			\label{convergence-a priori-2} 
		\end{equation} 
		Subsequently, an application of inequality~\eqref{lem 3-0} (in Lemma~\ref{lem:convection}) implies that 
		\begin{equation} 
			\| \u^\ell \DOT \nabla_N e_\omega^\ell 
			+ \nabla_N \cdot ( \u^\ell e_\omega^\ell )  \|_2 
			\le \gamma_0 \| \pomega^\ell \|_{H^\delta}  \cdot \| \nabla_N e_\omega^\ell \|_2 
			\le \gamma_0 \tilde{C}_1 \| \nabla_N e_\omega^\ell \|_2 , \quad \ell = n, n-1, n-2 . 
			\label{convergence-a priori-3} 
		\end{equation} 
		
		Taking a discrete inner product with~\eqref{BDF3-error-1} by $3 e_\omega^{n+1} - 2 e_\omega^n$ yields 
		\begin{align} 
			& 
			\left\langle \frac{11}{6} e_\omega^{n+1} - 3 e_\omega^n 
			+ \frac32 e_\omega^{n-1} - \frac13 e_\omega^{n-2} ,
			3 e_\omega^{n+1} - 2 e_\omega^n \right\rangle   
			+ \nu \dt   \left\langle \nabla_N  e_\omega^{n+1}  , 
			\nabla_N ( 3 e_\omega^{n+1} - 2 e_\omega^n ) \right\rangle 
			\nonumber 
			\\
			= & 
			\dt \Big\langle - \frac32 {\cal NLE}^n + \frac32 {\cal NLE}^{n-1} 
			+ \frac12 {\cal NLE}^{n-2} , 3 e_\omega^{n+1} - 2 e_\omega^n \Big\rangle   
			+  \dt \left\langle  \tau^{n+1} , 3 e_\omega^{n+1} - 2 e_\omega^n \right\rangle  .  
			\label{convergence-1}
		\end{align} 
		
		In terms of the temporal differentiation for the numerical error, an application of the combined telescope formula~\eqref{BDF3-telescope-3} gives  
		\begin{equation} 
			\begin{aligned} 
				&\Big\langle \frac{11}{6} e_\omega^{n+1} - 3 e_\omega^n 
				+ \frac32 e_\omega^{n-1} - \frac13 e_\omega^{n-2} ,
				3 e_\omega^{n+1} - 2 e_\omega^n \Big\rangle \\
				\ge &  \| \alpha_1 e_\omega^{n+1} \|_2^2 - \| \alpha_1 e_\omega^n \|_2^2
				+ \| \alpha_2 e_\omega^{n+1} + \alpha_3 e_\omega^n \|_2^2
				- \| \alpha_2 e_\omega^n + \alpha_3 e_\omega^{n-1} \|_2 ^2 + \| \alpha_4 e_\omega^{n+1} + \alpha_5 e_\omega^n + \alpha_6 e_\omega^{n-1} \|_2^2
				\\
				&
				- \| \alpha_4 e_\omega^n + \alpha_5 e_\omega^{n-1} + \alpha_6 e_\omega^{n-2} \|_2^2 
				+ \frac{13}{12}  \| e_\omega^{n+1} - e_\omega^n \|_2^2 
				- \frac{7}{12} \| e_\omega^n - e_\omega^{n-1} \|_2^2  - \frac16 \| e_\omega^{n-1} - e_\omega^{n-2} \|_2^2    \\
				&
				+ \frac{7}{12} \| e_\omega^{n+1} - 2 e_\omega^n + e_\omega^{n-1} \|_2^2 . 
			\end{aligned} 
			\label{convergence-2} 
		\end{equation} 
		The diffusion term could be analyzed in a similar fashion as in~\eqref{est-BDF3-L2-3}: 
		\begin{equation} 
			\begin{aligned}  
				& 
				\langle \nabla_N  e_\omega^{n+1}  , 
				\nabla_N ( 3 e_\omega^{n+1} - 2 e_\omega^n ) \rangle \\
				=&  \| \nabla_N e_\omega^{n+1} \|_2^2 
				+  2 \langle \nabla_N  e_\omega^{n+1}  , 
				\nabla_N ( e_\omega^{n+1} - e_\omega^n )  \rangle  
				\\
				= & 
				\| \nabla_N e_\omega^{n+1} \|_2^2 
				+  \|  \nabla_N  e_\omega^{n+1}  \|_2^2 - \| \nabla_N e_\omega^n \|_2^2 
				+ \| \nabla_N ( e_\omega^{n+1} - e_\omega^n )  \|_2^2 .   
			\end{aligned} 
			\label{convergence-3} 
		\end{equation} 
		The local truncation error term could be bounded by a direct application of Cauchy inequality: 
		\begin{equation} 
			\begin{aligned} 
				\left\langle  \tau^{n+1} , 3 e_\omega^{n+1}  - 2 e_\omega^n \right\rangle  
				\le& \| \tau^{n+1} \|_2  \cdot \| 3 e_\omega^{n+1}  - 2 e_\omega^n \|_2 \\
				\le& \frac12 \left( \| \tau^{n+1} \|_2^2 + \| 3 e_\omega^{n+1}  - 2 e_\omega^n \|_2^2 \right)     
				\\
				\le& 
				\frac12 \| \tau^{n+1} \|_2^2 + \frac{15}{2} \| e_\omega^{n+1} \|_2^2 
				+ 5 \| e_\omega^n \|_2^2 . 
			\end{aligned} 
			\label{convergence-4}
		\end{equation}
		
		To control the numerical error for the nonlinear convection, we focus the analysis at time step $t^n$. First, the following identity is observed, based on the fact that the numerical error function $e_\omega^k$ is $\ell^2$ orthogonal to any mass average profile: 
		\begin{equation} 
			\langle \overline{\u^n \DOT \nabla_N \u^n} ,  3 e_\omega^{n+1} - 2 e_\omega^n \rangle = 0 , 
			\quad \mbox{since $\overline{e_\omega^k}=0$} ,  \quad \forall k \ge 0. 
			\label{convergence-5-1} 
		\end{equation} 
		In the nonlinear convection error expansion in~\eqref{BDF3-error-4}, we see that the estimate for the first and third terms turns out to be straightforward: 
		\begin{align} 
			& 
			- \langle e_{\u}^n \DOT \nabla_N \pomega_N^n , 
			3 e_\omega^{n+1} - 2 e_\omega^n \rangle  \nonumber\\
			\le& \| e_{\u}^n \|_2 \cdot \| \nabla_N \pomega_N^n \|_\infty 
			\cdot \| 3 e_\omega^{n+1} - 2 e_\omega^n \|_2   \nonumber 
			\\
			\le &   
			C^* \| e_{\u}^n \|_2 \cdot \| 3 e_\omega^{n+1} - 2 e_\omega^n \|_2   \quad 
			\mbox{(by~\eqref{assumption:W1-infty bound})}  \nonumber 
			\\
			\le & 
			\frac{C^*}{2} ( \| e_{\u}^n \|_2^2 + 15 \| e_\omega^{n+1} \|_2^2 
			+ 10 \| e_\omega^n \|_2^2 ) ,  \label{convergence-5-2} 
			\\
			&
			- \big\langle \nabla_N \cdot ( e_{\u}^n \pomega_N^n ), 
			3 e_\omega^{n+1} - 2 e_\omega^n \big\rangle  \nonumber\\
			=& \big \langle e_{\u}^n \pomega_N^n ,
			\nabla_N ( 3 e_\omega^{n+1} - 2 e_\omega^n ) \big\rangle  \nonumber 
			\\
			\le & 
			\| e_{\u}^n \|_2 \cdot \| \pomega_N^n \|_\infty  
			\cdot \| \nabla_N ( 3 e_\omega^{n+1} - 2 e_\omega^n ) \|_2 \nonumber\\
			\le& C^* \| e_{\u}^n \|_2  \cdot \| \nabla_N ( 3 e_\omega^{n+1} - 2 e_\omega^n ) \|_2  
			\quad  \mbox{(by~\eqref{assumption:W1-infty bound})}  \nonumber  
			\\
			\le & 
			12 ( C^* )^2 \nu^{-1} \| e_{\u}^n \|_2^2   
			+ \frac{\nu}{48}  \| \nabla_N ( 3 e_\omega^{n+1} - 2 e_\omega^n ) \|_2^2  
			\nonumber 
			\\
			\le & 
			12 ( C^* )^2 \nu^{-1} \| e_{\u}^n \|_2^2   
			+ \frac{\nu}{24}  \| \nabla_N e_\omega^{n+1} \|_2^2 
			+ \frac{\nu}{6} \| \nabla_N ( e_\omega^{n+1} - e_\omega^n ) \|_2^2  .  
			\label{convergence-5-3} 
		\end{align} 
		For the second and fourth nonlinear error expansion terms in~\eqref{BDF3-error-4}, we are able to follow similar ideas in the stability analysis. Of course, the following error equality is valid, which comes from~\eqref{lem 2-2}: 
		\begin{equation} 
			\left\langle  \u^n \DOT \nabla_N e_\omega^n  
			+ \nabla_N \cdot \left( \u^n e_\omega^n \right) , e_\omega^n  \right\rangle       
			= 0  . \label{convergence-5-4}
		\end{equation}
		This in turn gives 
		\begin{equation} 
			\begin{aligned} 
				& 
				- \frac32  \left\langle \u^n \DOT \nabla_N e_\omega^n 
				+ \nabla_N \cdot \left( \u^n e_\omega^n \right) , 3 e_\omega^{n+1} - 2 e_\omega^n \right\rangle   
				\\
				=&
				- \frac32 \left\langle \u^n \DOT \nabla_N e_\omega^n 
				+ \nabla_N \cdot \left( \u^n e_\omega^n \right) , 3 e_\omega^{n+1} - 3 e_\omega^n \right\rangle 
				\\
				=  & 
				- \frac92 \left\langle \u^n \DOT \nabla_N e_\omega^n 
				+ \nabla_N \cdot \left( \u^n e_\omega^n \right) ,  e_\omega^{n+1} -  e_\omega^n \right\rangle  
				\\
				\le & 
				\frac92 \| \u^n \DOT \nabla_N e_\omega^n 
				+ \nabla_N \cdot ( \u^n e_\omega^n ) \|_2 \cdot \| e_\omega^{n+1} -  e_\omega^n \|_2 
				\\
				\le & 
				\frac{9 \gamma_0 \tilde{C}_1}{2} 
				\| \nabla_N e_\omega^n \|_2 \cdot \| e_\omega^{n+1} - e_\omega^n \|_2  \quad 
				\mbox{(by the a-priori estimate~\eqref{convergence-a priori-3})} 
				\\
				\le & 
				\frac{\nu}{4}  \|  \nabla_N e_\omega^n  \|_2^2
				+ \frac{81}{4} \gamma_0^2 \tilde{C}_1^2 \nu^{-1}     
				\| e_\omega^{n+1} - e_\omega^n \|_2^2 . 
			\end{aligned}   
			\label{convergence-5-5} 
		\end{equation} 
		Therefore, a combination of~\eqref{convergence-5-1}-\eqref{convergence-5-5} results in 
		\begin{equation} 
			\begin{aligned} 
				&
				- \frac32 \langle {\cal NLE}^n , 3 e_\omega^{n+1} - 2 e_\omega^n \rangle  
				\\
				= & 
				- \frac32 \langle e_{\u}^n \DOT \nabla_N \pomega_N^n  
				+ \u^n \DOT \nabla_N e_\omega^n 
				+ \nabla_N \cdot ( e_{\u}^n \pomega_N^n + \u^n e_\omega^n)  
				+ \overline{\u^n \DOT \nabla_N \u^n} , 
				3 e_\omega^{n+1} - 2 e_\omega^n  \rangle  
				\\
				\le & 
				\left( \frac{3 C^*}{4} + 18 ( C^* )^2 \nu^{-1} \right) \| e_{\u}^n \|_2^2   
				+ \frac{3 C^*}{4} ( 15 \| e_\omega^{n+1} \|_2^2 
				+ 10 \| e_\omega^n \|_2^2 ) 
				+ \frac{\nu}{16}  \| \nabla_N e_\omega^{n+1} \|_2^2  
				\\
				& 
				+ \frac{\nu}{4} \| \nabla_N ( e_\omega^{n+1} - e_\omega^n ) \|_2^2
				+ \frac{\nu}{4}  \|  \nabla_N e_\omega^n  \|_2^2 
				+ \frac{81}{4} \gamma_0^2 \tilde{C}_1^2 \nu^{-1}     
				\| e_\omega^{n+1} - e_\omega^n \|_2^2 . 
			\end{aligned} 
			\label{convergence-5-6} 
		\end{equation} 
		
		The numerical error terms at time steps $t^{n-1}$ and $t^{n-2}$ could be similarly analyzed; the technical details are skipped for the sake of brevity. 
		\begin{equation} 
			\begin{aligned} 
				&
				\frac32 \langle {\cal NLE}^{n-1} , 3 e_\omega^{n+1} - 2 e_\omega^n \rangle  
				\\
				\le & 
				( \frac{3 C^*}{4} + 18 ( C^* )^2 \nu^{-1} ) \| e_{\u}^{n-1} \|_2^2   
				+ \frac{3 C^*}{4} ( 15 \| e_\omega^{n+1} \|_2^2 
				+ 10 \| e_\omega^n \|_2^2 ) 
				+ \frac{\nu}{16}  \| \nabla_N e_\omega^{n+1} \|_2^2  
				\\
				& 
				+ \frac{\nu}{4} \| \nabla_N ( e_\omega^{n+1} - e_\omega^n ) \|_2^2
				+ \frac{\nu}{4}  \|  \nabla_N e_\omega^{n-1}  \|_2^2 
				+ \frac94 \gamma_0^2 \tilde{C}_1^2 \nu^{-1}     
				\| e_\omega^{n+1} - e_\omega^{n-1} + 2 ( e_\omega^{n+1} - e_\omega^n ) \|_2^2 , 
			\end{aligned} 
			\label{convergence-6} 
		\end{equation} 
		\begin{equation} 
			\begin{aligned} 
				&
				- \frac12 \langle {\cal NLE}^{n-2} , 3 e_\omega^{n+1} - 2 e_\omega^n \rangle  
				\\
				\le & 
				( \frac{C^*}{4} + 2 ( C^* )^2 \nu^{-1} ) \| e_{\u}^{n-2} \|_2^2   
				+ \frac{C^*}{4} ( 15 \| e_\omega^{n+1} \|_2^2 
				+ 10 \| e_\omega^n \|_2^2 ) 
				+ \frac{\nu}{16}  \| \nabla_N e_\omega^{n+1} \|_2^2  
				\\
				& 
				+ \frac{\nu}{4} \| \nabla_N ( e_\omega^{n+1} - e_\omega^n ) \|_2^2
				+ \frac{3 \nu}{16}  \|  \nabla_N e_\omega^{n-2}  \|_2^2 
				+ \frac13 \gamma_0^2 \tilde{C}_1^2 \nu^{-1}     
				\| e_\omega^{n+1} - e_\omega^{n-2} + 2 ( e_\omega^{n+1} - e_\omega^n ) \|_2^2 .  
			\end{aligned} 
			\label{convergence-7} 
		\end{equation} 
		
		As a result, a substitution of~\eqref{convergence-2}-\eqref{convergence-4}, 
		\eqref{convergence-5-6}, \eqref{convergence-6} and \eqref{convergence-1} into \eqref{convergence-1} leads to 
		\begin{equation} 
			\begin{aligned} 
				&
				\| \alpha_1 e_\omega^{n+1} \|_2^2 - \| \alpha_1 e_\omega^n \|_2^2
				+ \| \alpha_2 e_\omega^{n+1} + \alpha_3 e_\omega^n \|_2^2
				- \| \alpha_2 e_\omega^n + \alpha_3 e_\omega^{n-1} \|_2 ^2 
				\\
				&
				+ \| \alpha_4 e_\omega^{n+1} + \alpha_5 e_\omega^n + \alpha_6 e_\omega^{n-1} \|_2^2
				- \| \alpha_4 e_\omega^n + \alpha_5 e_\omega^{n-1} + \alpha_6 e_\omega^{n-2} \|_2^2 
				\\
				& 
				+ \frac{13}{12}  \| e_\omega^{n+1} - e_\omega^n \|_2^2 
				- \frac{7}{12} \| e_\omega^n - e_\omega^{n-1} \|_2^2 
				- \frac16 \| e_\omega^{n-1} - e_\omega^{n-2} \|_2^2    
				\\
				& 
				+ \frac{29 \nu \dt}{16} \| \nabla_N e_\omega^{n+1} \|_2^2 
				- \frac{5 \nu \dt}{4} \| \nabla_N e_\omega^n \|_2^2 
				+ \frac14 \nu \dt \| \nabla_N ( e_\omega^{n+1} - e_\omega^n )  \|_2^2 
				\\
				\le & 
				\frac{C^* \dt}{4} ( 3 \| e_{\u}^n \|_2^2 + 3 \| e_{\u}^{n-1} \|_2^2  
				+ \| e_{\u}^{n-2} \|_2^2 ) 
				+ ( C^* )^2 \nu^{-1}  \dt ( 18 \| e_{\u}^n \|_2^2 
				+ 18  \| e_{\u}^{n-1} \|_2^2  + 2 \| e_{\u}^{n-2} \|_2^2 ) 
				\\
				& 
				+ \frac{(7 C^* +2) \dt}{4} ( 15 \| e_\omega^{n+1} \|_2^2 
				+ 10 \| e_\omega^n \|_2^2 ) 
				+ \frac{\nu \dt}{4} \| \nabla_N e_\omega^{n-1} \|_2^2 
				+ \frac{3 \nu \dt}{16} \| \nabla_N e_\omega^{n-2} \|_2^2 
				\\
				&   
				+ \frac{\dt}{2} \| \tau^{n+1} \|_2^2   
				+ \gamma_0^2 \tilde{C}_1^2 \nu^{-1} \dt  
				\Big( \frac{81}{4} \|  e_\omega^{n+1}  - e_\omega^n \|_2^2   
				+ \frac94 \|  e_\omega^{n+1}  - e_\omega^{n-1} 	+ 2 ( e_\omega^{n+1} - e_\omega^n ) \|_2^2   
				\\
				&  
				+ \frac13 \|  e_\omega^{n+1}  - e_\omega^{n-2} 
				+ 2 ( e_\omega^{n+1} - e_\omega^n ) \|_2^2  \Big) 
				\\
				\le & 
				\gamma_0^2 \tilde{C}_1^2 \nu^{-1} \dt  
				\Big( \frac{209}{4} \|  e_\omega^{n+1}  - e_\omega^n \|_2^2  
				+ \frac{32}{3} \|  e_\omega^n  - e_\omega^{n-1} \|_2^2  
				+ \frac53 \| e_\omega^{n-1} - e_\omega^{n-2}  \|_2^2 \Big)  
				\\
				&  
				+ \frac{\nu \dt}{4} \| \nabla_N e_\omega^{n-1} \|_2^2 
				+ \frac{3 \nu \dt}{16} \| \nabla_N e_\omega^{n-2} \|_2^2 
				+ \frac{\dt}{2} \| \tau^{n+1} \|_2^2   
				\\
				& 
				+ \breve{C}_1  \dt \|  e_\omega^{n+1}  \|_2^2 + \breve{C}_0  \dt \|  e_\omega^n  \|_2^2   
				+ \breve{C}_{-1}  \dt \|  e_\omega^{n-1}  \|_2^2  
				+ \breve{C}_{-2}  \dt \|  e_\omega^{n-2}  \|_2^2 , 
			\end{aligned} 
			\label{convergence-8-1} 
		\end{equation} 
		with 
		\begin{equation*} 
			\begin{aligned} 
				& 
				\breve{C}_1 = \frac{15 ( 7 C^* +2)}{4} , \quad 
				\breve{C}_0 = \frac14 \Big(10 ( 7 C^* +2) + (3 C^* +72 (C^*)^2 \nu^{-1} )  (\gamma^*)^2 \Big) ,  
				\\
				&
				\breve{C}_{-1} = \frac14 \Big( 3 C^* + 72 (C^*)^2 \nu^{-1} \Big) (\gamma^*)^2 , \quad 
				\breve{C}_{-2} = \frac14 \Big( C^* + 8 (C^*)^2 \nu^{-1} \Big) (\gamma^*)^2 . 
			\end{aligned} 
		\end{equation*} 
		In fact, the last step comes from the a-priori estimate~\eqref{convergence-a priori-2}, as well as a Cauchy inequality similar to~\eqref{est-BDF3-L2-8-2}. Of course, the above inequality could be rewritten as 
		\begin{equation} 
			\begin{aligned} 
				&
				\| \alpha_1 e_\omega^{n+1} \|_2^2 + \| \alpha_2 e_\omega^{n+1} + \alpha_3 e_\omega^n \|_2^2
				+ \| \alpha_4 e_\omega^{n+1} + \alpha_5 e_\omega^n + \alpha_6 e_\omega^{n-1} \|_2^2  
				+ \frac{29 \nu \dt}{16} \| \nabla_N e_\omega^{n+1} \|_2^2  
				\\
				& 
				+ \Big( \frac{13}{12}  - \frac{209}{4} \gamma_0^2 \tilde{C}_1^2 \nu^{-1} \dt \Big) 
				\| e_\omega^{n+1} -e_ \omega^n \|_2^2 
				+ \frac14 \nu \dt \| \nabla_N ( e_\omega^{n+1} - e_\omega^n )  \|_2^2  
				\\
				\le & 
				\| \alpha_1 e_\omega^n \|_2^2
				+ \| \alpha_2 e_\omega^n + \alpha_3 e_\omega^{n-1} \|_2 ^2 
				+ \| \alpha_4 e_\omega^n + \alpha_5 e_\omega^{n-1} + \alpha_6 e_\omega^{n-2} \|_2^2 
				\\
				& 
				+ \frac{5 \nu \dt}{4} \| \nabla_N e_\omega^n \|_2^2 
				+ \frac{\nu \dt}{4} \| \nabla_N e_\omega^{n-1} \|_2^2 
				+ \frac{3 \nu \dt}{16} \| \nabla_N e_\omega^{n-2} \|_2^2 
				\\ 
				& 
				+ \Big( \frac{7}{12} + \frac{32 \gamma_0^2 \tilde{C}_1^2 \nu^{-1} \dt}{3}  \Big) 
				\|  e_\omega^n  - e_\omega^{n-1} \|_2^2  
				+ \Big( \frac16 + \frac{5 \gamma_0^2 \tilde{C}_1^2 \nu^{-1} \dt}{3}  \Big)  
				\| e_\omega^{n-1} - e_\omega^{n-2}  \|_2^2  
				\\
				& 
				+ \breve{C}_1  \dt \|  e_\omega^{n+1}  \|_2^2 + \breve{C}_0  \dt \|  e_\omega^n  \|_2^2   
				+ \breve{C}_{-1}  \dt \|  e_\omega^{n-1}  \|_2^2  
				+ \breve{C}_{-2}  \dt \|  e_\omega^{n-2}  \|_2^2 
				+ \frac{\dt}{2} \| \tau^{n+1} \|_2^2   .  
			\end{aligned} 
			\label{convergence-8-3} 
		\end{equation} 
		Again, under the time step constraint~\eqref{constraint-BDF3-dt-1}, we obtain
		\begin{equation} 
			\begin{aligned} 
				&
				E^{n+1}  
				+ \frac{29 \nu \dt}{16} \| \nabla_N e_\omega^{n+1} \|_2^2   
				+ \frac{11}{12} \| e_\omega^{n+1} - e_\omega^n \|_2^2 
				- \frac58 \|  e_\omega^n  - e_\omega^{n-1} \|_2^2  
				- \frac{5}{24} \| e_\omega^{n-1} - e_\omega^{n-2}  \|_2^2
				\\
				\le & 
				E^n  
				+ \frac{5 \nu \dt}{4} \| \nabla_N e_\omega^n \|_2^2  
				+ \frac{\nu \dt}{4} \| \nabla_N e_\omega^{n-1} \|_2^2 
				+ \frac{3 \nu \dt}{16} \| \nabla_N e_\omega^{n-2} \|_2^2   
				\\
				& 
				+ \breve{C}_1  \dt \|  e_\omega^{n+1}  \|_2^2 + \breve{C}_0  \dt \|  e_\omega^n  \|_2^2   
				+ \breve{C}_{-1}  \dt \|  e_\omega^{n-1}  \|_2^2  
				+ \breve{C}_{-2}  \dt \|  e_\omega^{n-2}  \|_2^2 
				+ \frac{\dt}{2} \| \tau^{n+1} \|_2^2  , 
			\end{aligned} 
			\label{convergence-8-4} 
		\end{equation} 
		where $E^n := \| \alpha_1 e_\omega^n \|_2^2
				+ \| \alpha_2 e_\omega^n + \alpha_3 e_\omega^{n-1} \|_2 ^2 
				+ \| \alpha_4 e_\omega^n + \alpha_5 e_\omega^{n-1} 
				+ \alpha_6 e_\omega^{n-2} \|_2^2$. In turn, a summation in time indicates that 
		\begin{equation} 
			\begin{aligned} 
				&
				E^{n+1} + \frac{\nu \dt}{8} \sum_{k=1}^{n+1} \| \nabla_N e_\omega^k \|_2^2 
				+ \frac{1}{12} \sum_{k=1}^{n+1} \| e_\omega^k - e_\omega^{k-1} \|_2^2  
				\\
				\le & 
				E^0 + \tilde{C}_2 \dt \sum_{k=0}^{n+1} \| e_\omega^k \|_2^2  
				+ \frac{\dt}{2} \sum_{k=1}^{n+1} \| \tau^k \|_2^2 
				+ O ( (\dt^3 + h^m)^2 ) ,  \quad 
				\tilde{C}_2 := \breve{C}_1 + \breve{C}_0 + \breve{C}_{-1} + \breve{C}_{-2} . 
			\end{aligned} 
			\label{convergence-8-5} 
		\end{equation} 
		Meanwhile, by the fact that $\| e_\omega^k \|_2^2  \le \alpha_1^{-2} E^k$, we arrive at 
		\begin{equation} 
			\begin{aligned} 
				&
				E^{n+1} + \frac{\nu \dt}{8} \sum_{k=1}^{n+1} \| \nabla_N e_\omega^k \|_2^2 
				\le 
				E^0 + \alpha_1^{-2} \tilde{C}_2 \dt \sum_{k=0}^{n+1} E^k   
				+ \frac{\dt}{2} \sum_{k=1}^{n+1} \| \tau^k \|_2^2 
				+ O ( (\dt^3 + h^m)^2 ) . 
			\end{aligned} 
			\label{convergence-8-6} 
		\end{equation} 
		As a result, an application of discrete Gronwall inequality leads to the convergence estimate 
		\begin{equation} 
			E^{n+1} + \frac{\nu \dt}{8} \sum_{k=1}^{n+1} \| \nabla_N e_\omega^k \|_2^2  
			\le \hat{C} ( \dt^3 + h^m )^2 . 
			\label{convergence-8-7} 
		\end{equation} 
		Furthermore, its combination with definition~\eqref{convergence-8-4} (for $E^{n+1}$) indicates the desired result~\eqref{convergence-0}. This completes the proof of Theorem~\ref{thm:convergence}. 

\bibliographystyle{elsarticle-num}      
\bibliography{NSE}
\end{document}